\documentclass[letterpaper, 11pt]{amsart}

\usepackage{fullpage}
\usepackage{graphicx}
\usepackage{amsmath, amssymb, amsfonts, amsthm}
\usepackage{url}
\usepackage{enumitem}
\usepackage[foot]{amsaddr}

\usepackage{tikz} 

\usepackage[ruled,vlined]{algorithm2e}

\usepackage{color}
\definecolor{darkblue}{rgb}{0,0,0.4}
\usepackage[plainpages=false,colorlinks=true,citecolor=blue,filecolor=black,linkcolor=red,urlcolor=darkblue]{hyperref}

\theoremstyle{remark}

\pdfoutput=1
\newtheorem{theorem}{Theorem}[section]
\newtheorem{lemma}[theorem]{Lemma}
\newtheorem{corollary}[theorem]{Corollary}

\theoremstyle{definition}
\newtheorem{definition}[theorem]{Definition}

\theoremstyle{remark}

\numberwithin{equation}{section}

\begin{document}
\title{Bloch Waves in High Contrast Electromagnetic Crystals}

%Spectral theory without ellipticity%and Bloch waves in crystals

\author{Robert Lipton}\address{Department of Mathematics, 
Louisiana State University,
Baton Rouge, LA 70803, USA,
{\tt lipton@lsu.edu}}
%\and
\author{Robert Viator Jr.}
\address{Department of Mathematics \& Statistics,
Swarthmore College,
Swarthmore, PA 19081, USA,
{\tt rviator1@swarthmore.edu}}
%\and
\author{Silvia Jim\'enez Bola\~nos}
\address{Department of Mathematics,
Colgate University University,
Hamilton, NY 13346, USA,{\tt sjimenez@colgate.edu}}
%\and
\author{Abiti Adili}
\address{Department of Mathematics,
Franklin and Marshall College,
Lancaster, PA 17604-3003, USA,
{\tt abiti.adili@fandm.edu}}

%\subjclass{05B35}
\date{\today}

%\maketitle
\begin{abstract}
 Analytic representation formulas and power series are developed describing the band structure inside non-magnetic  periodic photonic three-dimensional crystals made from high dielectric contrast inclusions. Central to this approach is the identification and utilization of a resonance spectrum for quasiperiodic source-free modes. These modes are used to represent solution operators associated with electromagnetic and acoustic waves inside periodic high contrast media. A convergent power series for the Bloch wave spectrum is recovered from the representation formulas. Explicit conditions on the contrast are found that provide lower bounds on the convergence radius. These conditions are sufficient for the separation of spectral branches of the dispersion relation for any fixed quasi-momentum.
\end{abstract}

\maketitle

\section{Introduction}
\label{introduction}

We are interested in photonic crystals, or photonic band-gap materials, and their use in controlling the propagation of light.  A photonic crystal is an artificially created optical material, which can be considered as the optical analog of a semiconductor, since it behaves with respect to photon propagation in a similar fashion as the semiconductor behaves with respect to electron propagation.  Developments in optical materials provide benefits to a number of fields, including spectroscopy and high-speed computing, for example.  Several books and surveys have been written about the subject; see, for instance, \cite{Joannopoulos2008,Johnson2002,Kuchment2001,Kuchment2004,Sakoda2001,Slusher2003}.

A  photonic crystal is a periodic lattice of inclusions surrounded by a connected phase with the property that the contrast $k$ between the dielectric properties of the inclusions and the connected phase can be quite large.  Understanding the propagation of electromagnetic waves in photonic crystals is crucial since it might allow tailoring materials to obtain desired properties.  The Maxwell system is given by:
\begin{equation} 
\left\{
	\begin{array}{ll}
		\nabla\times\mathbf{E}=-\frac{1}{c}\frac{\partial\mathbf{B}}{\partial t},  & \nabla\cdot\mathbf{B}=0 \\\\
		\nabla\times\mathbf{H}=\frac{1}{c}\frac{\partial\mathbf{D}}{\partial t},  & \nabla\cdot\mathbf{D}=0,
	\end{array}
\right.
\label{Maxwell}
\end{equation}
\noindent where $c$ is the speed of light in free space, the vector-valued functions $\mathbf{E}$ and $\mathbf{H}$ are the macroscopic electric and magnetic fields, and $\mathbf{D}$ and $\mathbf{B}$ are the displacement and magnetic induction fields, respectively \cite{Jackson1962}.  To complete the Maxwell system the \textit{constitutive relations} describing the dependence of $\mathbf{D}$ and $\mathbf{B}$ on $\mathbf{E}$ and $\mathbf{H}$ are supplied.  We apply the linear constitutive relations, given by:
\begin{equation*}
%\label{linearBH} 
\mathbf{D}=\epsilon\mathbf{E}, \,\,\,\mathbf{B}=\mu\mathbf{H},
\end{equation*}
where $\epsilon$ is the dielectric constant and $\mu$ is the magnetic permeability.  In this treatment, it is assumed that the media is isotropic, the material is non-magnetic (i.e. $\mu=1$), and  the dielectric constant $\epsilon(x)$ is periodic.

%Applying the Fourier transform to the system (\ref{Maxwell}), 
We consider the case of monochromatic waves $\mathbf{E}(x,t)=e^{i\omega t}\mathbf{E}(x)$, $\mathbf{H}(x,t)=e^{i\omega t}\mathbf{H}(x)$, where $\omega$ is the time frequency, and the system (\ref{Maxwell}) becomes:
\begin{equation*} 
\left\{
	\begin{array}{ll}
		\nabla\times\mathbf{E}=-\frac{i\omega}{c}\frac{\partial\mathbf{H}}{\partial t},  & \nabla\cdot\mathbf{H}=0 \\\\
		\nabla\times\mathbf{H}=\frac{i\omega}{c}\epsilon(x)\frac{\partial\mathbf{E}}{\partial t},  & \nabla\cdot\epsilon\mathbf{E}=0 
	\end{array}
\right.,
\end{equation*}
\noindent which, after eliminating the electric field $\mathbf{E}$, reduces to:
\vspace{-3pt}\begin{equation} 
\nabla\times\frac{1}{\epsilon(x)}\nabla\times\mathbf{H}=\xi\mathbf{H}, \hspace{1cm} \nabla\cdot\mathbf{H}=0, \quad \text{where $\xi=(\omega/c)^2$.}
\label{eigprov1}
\end{equation}

In a two-dimensional periodic medium (where $\epsilon(x)$ is periodic with respect to $x$ and $y$ and homogeneous with respect to $z$, for example), problem (\ref{eigprov1}) reduces to scalar equations $-\Delta E=\lambda\epsilon(x)E$ and:
\vspace{-3pt}\begin{equation} 
-\nabla\cdot\frac{1}{\epsilon(x)}\nabla H=\xi H, \qquad \text{where $\xi=(\omega/c)^2$. }
\label{scalareigprob}
\end{equation}

One of the main goals of the photonic crystals theory is to choose $\epsilon(x)\geq1$ such that the spectrum of the corresponding problem, scalar (\ref{scalareigprob}) or vectorial (\ref{eigprov1}), has a gap.  Existence of a gap delivers a frequency interval (band) over which electromagnetic waves cannot propagate in the material.  A complete band gap is a range of frequencies for which no  Bloch wave of any wavelength or direction can  propagate through the crystal.  Band gaps have many interesting and useful applications ranging from efficient photovoltaic cells to power electronics and optical computers, see \cite{Joannopoulos2008,Johnson2002}.

Most of the state-of-the-art developments \cite{AmmariKang1,AmmariKangLee,Bouchitte2017,FigKuch2,FigKuch3,FigKuch1,HempelLienau} have been restricted to the asymptotic theory of band gaps at infinite contrast.  For the scalar case (\ref{scalareigprob}), the authors exploited structural resonances associated with the Neumann-Poincar\'e operator to develop new techniques for complex operator valued functions, which delivered explicit formulas for band gaps at finite contrast.  This provides mathematically rigorous and explicit formulas for the size of band gaps and pass bands, given in terms of the contrast, shape and configuration of scatters, and lattice parameters, see \cite{RobertRobert1,LV2017b}.  

In this paper, we lay the foundation for the analytical methods to obtain the corresponding results to the ones obtained in \cite{RobertRobert1} for the fully three-dimensional electromagnetic photonic crystals lattices, via the vector wave equation (\ref{eigprov1}).  In particular, we establish an analytic representation for the periodic and quasiperiodic spectra of (\ref{eigprov1}) in terms of the contrast between the dielectric constants of the two material components, together with a radius of convergence described in terms of the crystal geometry by way of the associated Neumann-Poincar\'e spectrum.

We consider a Bloch wave $\mathbf{h}(x)$, with Bloch eigenvalue $\xi=(\omega/c)^2$, propagating through a three-dimensional photonic crystal, characterized by the periodic relative dielectric constant $a^{-1}(x)=\epsilon(x)=\epsilon(x+p)$, $p\in\mathbb{Z}^3$, with unit cell $Y=(0,1]^3$, defined by:
\begin{equation*}
%\label{eps} 
\epsilon(x) = 
     \begin{cases}
      1  &\quad\text{inside the inclusion $D$}\\
      \epsilon=1/k &\quad\text{in the host material $H:=Y\setminus D$.}
     \end{cases}
\end{equation*}
The magnetic field $\mathbf{h}(x)$ inside ``non-magnetic media'' solves the vector Helmholtz equation: 
\begin{equation}
	\label{Helm}
	\nabla\times\left(a(x)\nabla\times \mathbf{h}(x)\right)=\xi\mathbf{h}(x), \hspace{2mm}x\in\mathbb{R}^3,
\end{equation}
together with the $\alpha$-quasiperiodicity condition $\mathbf{h}(x+p)=\mathbf{h}(x)e^{i\alpha\cdot\mathbf{p}}$.  Here, $\alpha$ lies in the first Brillouin zone of the reciprocal lattice given by $Y^*=(-\pi,\pi]^3$.  Equation (\ref{Helm}) describes time harmonic wave propagation for the magnetic field in non-magnetic media, i.e., for heterogeneous media with relative magnetic permeability $\mu=1$ everywhere.  

We examine Bloch wave propagation through high contrast crystals made from periodic configurations of two dielectric materials.  The inclusion $D$ contained within the interior of the period cell $Y$ and surrounded by the second ``host'' material, $H:=Y\setminus D$, see Figure~\ref{plane}.   
\begin{figure}[ht]
  \centering
  \includegraphics[width=0.25\linewidth]{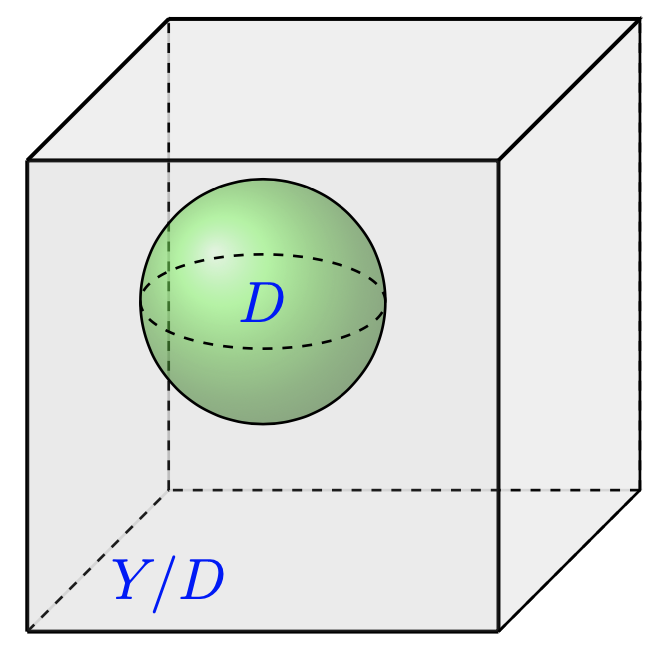}
  \caption{Period cell}
  \label{plane}
\end{figure}

The coefficient $a(x)$ is then specified on the unit cell by:
\begin{equation*}
%	\label{a}
	a(x)=k\chi_{H}(x)+\chi_{D}(x),
\end{equation*}
where $\chi_{H}$ and $\chi_{D}$ are the indicator functions for the sets $H$ and $D$, and are extended by periodicity to $\mathbb{R}^3$.  In this paper, we  consider periodic crystals made from finite collections of separated inclusions, each with $C^{1,\gamma}$ boundary, where $\gamma >0$.

For each $\alpha\in Y^\star$, the Bloch eigenvalues $\xi$  are of finite multiplicity and denoted by $\lambda_j(k,\alpha)$, $j\in \mathbb{N}$. We develop power series expansions for each branch of the dispersion relation:
\begin{equation}
\lambda_j(k,\alpha)=\xi,\,\,\hbox{ $j\in\mathbb{N}$}
\label{DispersionRelations}
\end{equation}
that are valid for $k$ in a neighborhood of infinity.

To proceed, we complexify the problem and consider $k\in\mathbb{C}$.  Now $a(x)$ takes on complex values inside $H$ and the operator $-\nabla\times (k \chi_{H} + \chi_{D})\nabla\times$
is no longer uniformly elliptic.  Our approach develops an explicit representation formula for $-\nabla\times (k \chi_{H} + \chi_{D})\nabla\times$ that holds for complex values of $k$.  We identify the subset $z={1}/{k}\in\Omega_0$ of $\mathbb{C}$ where this operator is invertible. The explicit formula shows that the solution  operator $(-\nabla\times(k \chi_{H} + \chi_{D})\nabla\times)^{-1}$ may be regarded more generally as a meromorphic operator valued function of $z$, for $z\in\Omega_0=\mathbb{C}\setminus S$, see Section~\ref{asymptotic}  and Lemma~\ref{inverseoperator}. Here, the set $S$ is discrete and consists of poles lying on the negative real axis with only one accumulation point at $z=-1$. For the problem treated here, we expand about $z=0$, and the distance between $z=0$ and the set $S$ is used to bound the radius of convergence for the power series. The spectral representation for $-\nabla\times (k \chi_{H} + \chi_{D})\nabla\times$ follows from the existence of a complete orthonormal set of $\alpha$-quasiperiodic functions associated with the {\em $\alpha$-quasiperiodic resonances of the crystal}, i.e., $\alpha$-quasiperiodic functions $\mathbf{v}$ and real eigenvalues 
$\lambda_i(\alpha)$, $i\in\mathbb{N}$, for which:
\begin{equation*}
-\nabla\times(\chi_{D})\nabla\times \mathbf{v}=-\lambda_i(\alpha) \Delta \mathbf{v}.
\end{equation*}
The collection of these eigenvalues, for $\alpha\in Y^\ast$, comprises the \textit{structural spectrum} of the crystal. 
The structural spectrum encodes the geometry of the crystal and inclusions independently of dielectric properties. 
These resonances are shown to be connected to the spectra of Neumann-Poincar\'e operators associated with $\alpha$-quasiperiodic double layer potentials. The formal definition of the structural spectrum given in terms of the Neumann-Poincar\'e eigenvalues, for $\alpha\in Y^\ast$, is provided in Definition \ref{DefStruct}.

For $\alpha=\mathbf{0}$, these eigenvalues are the well known electrostatic resonances identified in \cite{BergmanES}, \cite{BergmanC}, \cite{MiltonES}, and \cite{Milton}. Other electrostatic resonances for a vectorial Helmholtz equation are introduced and explored in \cite{Bergman19}. Both Neumann-Poincar\'e operators and the associated electrostatic resonances have been the focus of  theoretical investigations \cite{Kang}, \cite{Shapero} and applied in analysis of plasmonic excitations for suspensions of noble metal particles \cite{Mayorgoz} and electrostatic breakdown \cite{JagerMosko}.  The explicit spectral representation for the operator  $-\nabla\times (k \chi_{H} + \chi_{D})\nabla\times$ is crucial for elucidating the interaction between the contrast $k$ and the quasiperiodic resonances of the crystal, see Theorem~\ref{representation}.  

The spectral representation is applied to analytically continue the band structure $\lambda_j(k,\alpha)=\xi$, $j\in\mathbb{N}$, $\alpha\in Y^\star$ for $k \in \mathbb{N}$ to $\mathbb{C}$, see Theorem~\ref{extension}. 
On setting $z=1/k$, the spectral representation for the inverse operator written as $A_\alpha(z)=(-\nabla\times (k \chi_{H} + \chi_{D})\nabla\times)^{-1}$  shows it to be a meromorphic operator valued function of $z=1/k$, see Section~\ref{asymptotic}  and Lemma~\ref{inverseoperator}. 
Application of  the contour integral formula for spectral projections \cite{SzNagy}, \cite{TKato1}, \cite{TKato2} delivers an analytic representation formula for the band structure, see Section~\ref{asymptotic}.  We apply perturbation theory in Section~\ref{asymptotic}, together with a calculation provided in Section~\ref{derivation}, to find an explicit formula for the radii of convergence for the power series  $\lambda_j(k,\alpha)$  about $1/k=0$. The formula shows that the radius of convergence and the separation between different branches of the dispersion relation for any fixed $\alpha\in Y^\ast$ are determined by: 1) the distance of the origin to the nearest pole $z^\ast$ of $(-\nabla\times (k \chi_{H} + \chi_{D})\nabla\times)^{-1}$, and 2) the separation between distinct eigenvalues in the $z=1/k\rightarrow 0$ limit, see Theorem~\ref{separationandraduus-alphanotzero} and Theorem~\ref{separationandraduus-alphazero}. These theorems provide conditions on the contrast guaranteeing the separation of the $j$-th and $j+1$-th eigenvalue groups that depend explicitly upon  $z^\ast$, $j \in \mathbb{N}$ and $\alpha\in Y^\star$. Error estimates for series truncated after $N$ terms follow directly from the formulation.

The paper is organized as follows: In the next section, we introduce the Hilbert space formulation of the problem and the variational formulation of the quasi-static resonance problem. The completeness of the eigenfunctions associated with the quasi-static spectrum is established and a spectral representation for the operator $-\nabla\times (k \chi_{H} + \chi_{D})\nabla\times$ is obtained. These results are collected and used to continue the frequency band structure into the complex plane, see  Theorem~\ref{extension} of Section~\ref{bandstructure}. Spectral perturbation theory \cite{KatoPerturb} is applied to recover the power series expansion for Bloch spectra in Section~\ref{asymptotic}. The leading order spectral theory is developed for quasiperiodic $\alpha\not=0$ and periodic $\alpha=\mathbf{0}$ problems in Section~\ref{limitspectra} and Section~\ref{limitspeczero}, respectively.  The main theorems on radius of convergence and separation of spectra, given by Theorem~\ref{separationandraduus-alphanotzero} and Theorem~\ref{separationandraduus-alphazero}, are presented in Section~\ref{radius}. A large class of geometries for which an $\alpha$-independent lower bound on the quasi-static resonances is introduced in Section~\ref{radiusgeneralshape}. Explicit formulas for each term of the power series expansion is recovered and expressed in terms of layer potentials in Section~\ref{leading-order}.  The explicit formulas for the convergence radii are derived in Section~\ref{derivation} as well as hands-on proofs of Theorem~\ref{separationandraduus-alphanotzero}, Theorem~\ref{separationandraduus-alphazero} and the explicit error estimates for $N$-th order truncations.

\section{Hilbert space setting, quasiperiodic resonances and representation formulas}
\label{layers}

The space of all $\alpha$-quasiperiodic complex vector valued functions belonging to $L^2_{loc}(\mathbb{R}^3,\mathbb{C}^3)$ is denoted by $L^2_\#(\alpha,Y,\mathbb{C}^3)$ and the $L^2$-inner product is defined by:
\begin{equation}
\label{L2innerproduct}
(\mathbf{u},\mathbf{v})=\int_Y\,\mathbf{u}\cdot\overline{\mathbf{v}}\,dx.
\end{equation}

For $\mathbf{h}\in L^2_\#(\alpha,Y,\mathbb{C}^3)$, its Helmholtz decomposition is given by:
\begin{equation}
\mathbf{h}=\nabla h_{\rm pot}+\nabla\times\mathbf{h}_{\rm curl},
\label{Helmoltz1}
\end{equation}
where $h_{\rm pot}$ is an $\alpha$-quasiperiodic scalar field belonging to  $H^1_{loc}(\mathbb{R}^3,\mathbb{C})$ and  $\mathbf{h}_{\rm curl}\in L^2_\#(\alpha,Y,\mathbb{C}^3)$, with $\nabla\times\mathbf{h}_{\rm curl}\in L^2_\#(\alpha,Y,\mathbb{C}^3)$. The subspaces of gradients and curls are orthogonal with respect to the $L^2$-inner product (\ref{L2innerproduct}). The Helmholtz decomposition (\ref{Helmoltz1}) is shown in Appendix~\ref{app:helm}. 

For $\alpha\neq\mathbf{0}$, the eigenfunctions $\mathbf{h}$ of \eqref{Helm} belong to the space $J_{\#}(\alpha,Y,\mathbb{C}^3)\subset L^2_\#(\alpha,Y,\mathbb{C}^3)$ given by:
\begin{equation}
J_{\#}(\alpha,Y,\mathbb{C}^3) = \{ \mathbf{h} \in  H_{loc}^1(\mathbb{R}^3,\mathbb{C}^3): \,\mathbf{h}\text{ is $\alpha$-quasiperiodic on   $Y$},\,\,{\rm div}\,\mathbf{h}=0 \text{ in }Y\}.
\label{H1}
\end{equation}

A simple calculation, found in Appendix~\ref{app:nhp=0}, shows that, for $\mathbf{h}\in J_{\#}(\alpha,Y,\mathbb{C}^3)$, we have $\nabla h_{pot}=0$ in (\ref{Helmoltz1}).  Hence, $\mathbf{h}=\nabla\times\mathbf{h}_{\rm curl}$ for $\mathbf{h}\in J_{\#}(\alpha,Y,\mathbb{C}^3)$.
Another straightforward calculation, given in  Appendix~\ref{app:h=0inY}, delivers the following result:
\begin{theorem}
For $\mathbf{u}\in J_{\#}(\alpha,Y,\mathbb{C}^3)$, the null space of $\nabla \times \mathbf{u}$, for $\alpha\not=0$, is $\{0\}$ and the bilinear form given by: 
\begin{equation}
\label{innerproduct}
\begin{aligned}
\langle \mathbf{u},\mathbf{v}\rangle=\int_Y\,\nabla\times\mathbf{u}\cdot\nabla\times\overline{\mathbf{v}}\,dx
\end{aligned}
\end{equation}
is an inner product on $J_{\#}(\alpha,Y,\mathbb{C}^3)$, with norm defined by $\Vert\mathbf{u}\Vert^2=\langle \mathbf{u},\mathbf{u}\rangle$. The space  $J_{\#}(\alpha,Y,\mathbb{C}^3)$ is a Hilbert space under the inner product \eqref{innerproduct}, with $J_{\#}(\alpha,Y,\mathbb{C}^3)\subset W_{\#}^1(\alpha,Y,\mathbb{C}^3)$ and: $$\int_Y\,\nabla\times\mathbf{u}\cdot\nabla\times\overline{\mathbf{v}}\,dx=\int_Y\,\nabla\mathbf{u}:\nabla\overline{\mathbf{v}}\,dx$$ for $\mathbf{u},\,\mathbf{v}\in J_{\#}(\alpha,Y,\mathbb{C}^3)$, where `` $:$" represents the Frobenius inner product (see Appendix~\ref{app:lemmas}). Moreover, the null space corresponding to the operator on the left hand side of \eqref{Helm} is identically zero.
\end{theorem} 

For $\alpha=\mathbf{0}$, one has that $L^2_\#(0,Y,\mathbb{C}^3)$ is the space of periodic $L^2$- vector fields on $Y$.  For this case, $\mathbf{h}\in L^2_\#(0,Y,\mathbb{C}^3)$ has the Helmholtz decomposition into $L^2$- orthogonal components given by:
\begin{equation}
\mathbf{h}=\nabla h_{\rm pot}+\nabla\times\mathbf{h}_{\rm curl}+\mathbf{c},
\label{Helmoltz2}
\end{equation}
where $h_{\rm pot}$ is a periodic scalar field belonging to $H^1_{loc}(\mathbb{R}^3,\mathbb{C})$,  $\mathbf{h}_{\rm curl}\in L^2_\#(0,Y,\mathbb{C}^3)$, with $\nabla\times\mathbf{h}_{\rm curl}\in L^2_\#(0,Y,\mathbb{C}^3)$, and $\mathbf{c}$ is a constant vector in $\mathbb{C}^3$, see Appendix~\ref{app:helm}. For $\alpha=\mathbf{0}$, the eigenfunctions $\mathbf{h}$ for \eqref{Helm} belong to the space:
\begin{equation*}
\{ \mathbf{h} \in  H_{loc}^1(\mathbb{R}^3,\mathbb{C}^3): \text{$\mathbf{h}$ periodic on $Y$},\,\,\nabla\cdot\mathbf{h}=0\text{ in }Y\}.
%\label{H0}
\end{equation*}
A simple calculation, given in Appendix~\ref{app:nhp=0}, shows that $\nabla h_{pot}=0$ and $\mathbf{h}=\nabla\times\mathbf{h}_{curl}+\boldsymbol{c}$.  We introduce the  space $J_{\#}(0,Y,\mathbb{C}^3) \subset L^2_{\#}(0,Y,\mathbb{C}^3)$ given by:
\begin{equation}
J_{\#}(0,Y,\mathbb{C}^3) =\{ \mathbf{h} \in  H_{loc}^1(\mathbb{R}^3,\mathbb{C}^3): \text{$\mathbf{h}$ is periodic},\,\,\nabla\cdot\mathbf{h}=0\text{ in }Y, \text{ and } \int_Y\,\mathbf{h}\,dx=0\}.
\label{H1-0}
\end{equation}

\begin{theorem}
For $\mathbf{u}\in J_{\#}(0,Y,\mathbb{C}^3)$, the null space of $\nabla \times \mathbf{u}$ is $\{0\}$ and the bilinear form: 
\begin{equation}\label{innerproduct2}
\begin{aligned}
\langle \mathbf{u},\mathbf{v}\rangle=\int_Y\,\nabla\times\mathbf{u}\cdot\nabla\times\overline{\mathbf{v}}\,dx,
\end{aligned}
\end{equation}
is an inner product on $J_{\#}(0,Y,\mathbb{C}^3)$, with norm defined by $\Vert\mathbf{u}\Vert^2=\langle \mathbf{u},\mathbf{u}\rangle$. The space $J_{\#}(0,Y,\mathbb{C}^3)\subset W_{\#}^1(0,Y,\mathbb{C}^3)$ with inner product \eqref{innerproduct2} is a Hilbert space and: $$\int_Y\,\nabla\times\mathbf{u}\cdot\nabla\times\overline{\mathbf{v}}\,dx=\int_Y\,\nabla\mathbf{u}:\nabla\overline{\mathbf{v}}\,dx$$ for $\mathbf{u},\,\mathbf{v}\in J_{\#}(0,Y,\mathbb{C}^3)$. Moreover, the null space corresponding to the operator on the left hand side of \eqref{Helm}, for $\mathbf{h}\in J_{\#}(0,Y,\mathbb{C}^3)$, is $\{0\}$.
\end{theorem}
This theorem follows from a calculation given in Appendix~\ref{app:h=0inY}.
From now on, we will refer to $J_{\#}(\alpha,Y,\mathbb{C}^3)$ for all $\alpha \in Y^\ast$, with the special choice of $J_{\#}(\alpha,Y,\mathbb{C}^3)$ for $\alpha=\mathbf{0}$ defined as in \eqref{H1-0}.

The  weak form of equation \eqref{Helm} is given by:
\begin{equation}\label{weak}
\epsilon^{-1}\int_{H}(\nabla\times \mathbf{h})\cdot (\nabla\times \overline{\mathbf{w}})\ dx +  \int_{D}(\nabla\times \mathbf{h})\cdot (\nabla\times \overline{\mathbf{w}})\ dx=\xi \int_{Y}\mathbf{h}\cdot \overline{\mathbf{w}}\ dx,
\end{equation}
for all $\mathbf{w}\in J_{\#}(\alpha,Y,\mathbb{C}^3)$.  We set $k=\epsilon^{-1}$, and the left hand side of \eqref{weak} is given by the sesquilinear form $B_{k}: J_{\#}(\alpha,Y,\mathbb{C}^3)\times J_{\#}(\alpha,Y,\mathbb{C}^3) \rightarrow \mathbb{C}$, defined as:
 \begin{equation}\label{ses}
 B_{k}(\mathbf{u,w}):= k\int_{H}(\nabla\times \mathbf{u})\cdot (\nabla\times \overline{\mathbf{w}})\ dx +  \int_{D}(\nabla\times \mathbf{u})\cdot (\nabla\times \overline{\mathbf{w}})\ dx.
 \end{equation}
The linear  operator $T_{k}^\alpha$, associated with the sesquilinear form $B_{k}$, is defined by: \begin{equation}
    \label{tkassocbk}
\langle T_{k}^\alpha\mathbf{u},\mathbf{w}\rangle:=B_{k}(\mathbf{u},\mathbf{w}),
\end{equation}
for all $\mathbf{u}$ and $\mathbf{w}$ in $J_{\#}(\alpha,Y,\mathbb{C}^3)$.

Our goal is to  rewrite \eqref{Helm} in terms of a spectral representation formula for the differential operator $\nabla\times\left( k\nabla \times\cdot\right)$.  We will do this by developing the spectral representation of $T_k^\alpha$, which can be directly linked to the following eigenvalue problem:
\begin{equation}\label{ypproblem}
\lambda\langle\mathbf{u},\mathbf{w}\rangle=\lambda\int_{Y}(\nabla \times \mathbf{u})\cdot (\nabla \times \overline{\mathbf{w}})\,dx=\int_{D}(\nabla \times \mathbf{u})\cdot (\nabla \times \overline{\mathbf{w}})\,dx, 
\end{equation}
for all $\mathbf{u,w}\in J_{\#}(\alpha,Y,\mathbb{C}^3)$; which will be shown to possess countably many real eigenvalues $\lambda_{n}$, with  corresponding eigenfunctions $\psi_{n}\in J_{\#}(\alpha,Y,\mathbb{C}^3)$, that satisfy:
 \begin{equation*}
 \lambda_{n} \int_{Y}(\nabla\times \psi_{n})\cdot (\nabla\times \overline{\mathbf{w}})\,dx= \int_{D}(\nabla\times \psi_{n})\cdot (\nabla\times \overline{\mathbf{w}})\,dx,\qquad  \forall\mathbf{w}\in J_{\#}(\alpha,Y,\mathbb{C}^3).
 \end{equation*} 
The eigenspaces associated with different eigenvalues are easily seen to be orthogonal in the inner product \eqref{innerproduct}.
 We apply these eigenfunctions to introduce a different decomposition of $J_\#(\alpha, Y, \mathbb{C}^3)$ that is orthogonal in the inner product \eqref{innerproduct}. We introduce the three subspaces denoted by $W_1^\alpha$, $W_2^\alpha$, $W_3^\alpha$ that are mutually orthogonal with respect to the inner product \eqref{innerproduct} and defined as:
\begin{equation}\label{w1}
W_{1}^\alpha =  \left\{\mathbf{u}\in J_{\#}(\alpha,Y,\mathbb{C}^3),\,\,\nabla\times \mathbf{u}=0 \text{ in }D\right\}, 
\end{equation}
\begin{equation}\label{w2}
W_{2}^\alpha = \left\{\mathbf{u}\in J_{\#}(\alpha,Y,\mathbb{C}^3),\,\,\nabla\times \mathbf{u}=0 \text{ in }H\right\},
\end{equation}
and  $W_{3}^\alpha\subset J_{\#}(\alpha,Y,\mathbb{C}^3)$ is the subspace perpendicular to  the direct sum $(W_{1}^\alpha\oplus W_{2}^\alpha)$. 

The decomposition of $J_{\#}(\alpha,Y,\mathbb{C}^3)$ is recorded in the following lemma.

 \begin{lemma}\label{ortho} The space $J_{\#}(\alpha,Y,\mathbb{C}^3)$ can be decomposed into orthogonal  invariant subspaces  spanned by eigenfunctions of the eigenvalues of problem   \eqref{ypproblem} and: $$J_{\#}(\alpha,Y,\mathbb{C}^3)=W_1^\alpha\oplus W_2^\alpha\oplus W_3^\alpha.$$ \end{lemma}

It follows from the definitions of $W_{1}^\alpha$ and $W_{2}^\alpha$ that they are subspaces of the eigenspaces of \eqref{ypproblem} associated with the eigenvalues  $0$ and $1$, respectively. From \eqref{ypproblem}, we easily deduce that the eigenvalues $\lambda$ belong to $[0,1]$.
To proceed, we must provide the explicit characterization of functions in $W_3$ in terms of eigenspaces. To do this, we introduce the appropriate differential operators defined on the surface of the dielectric inclusion $\partial D$. We begin by defining the surface differential operators for smooth functions. 
The surface divergence $Div_S$ for smooth complex-valued tangential vector fields $\mathbf{v}$ is defined over the surface $\partial D$ by: 
\begin{equation*}
%\label{surfaceDiv}
Div_S\mathbf{v}:=\sum_{j,i}n_i(n_i\partial_j-n_j\partial_i)v_j,
\end{equation*}
where $n_i,\,i=1,2,3$, are the components of the unit outward normal vector $\mathbf{n}$ to the surface.
The operator:
\begin{equation*}
%\label{divcurl}
\begin{aligned}
\mathbf{n}\cdot\nabla\times\mathbf{v}&:=&(n_2\partial_3-n_3\partial_2,n_3\partial_1-n_1\partial_3,n_1\partial_2-n_2\partial_1)\cdot\mathbf{v}
\end{aligned}
\end{equation*}
is only composed of tangential derivatives and can be viewed as an operator defined on $\partial D$. 
For every vector field $\mathbf{v}$ in $L^2(\partial D)^3$, we have the relation between $Div_S$ and $\mathbf{n}\cdot\nabla\times$ given by:
\begin{equation*}
%\label{divcurlDivs}
\begin{aligned}
Div_S(\mathbf{n}\times\mathbf{v})=-\mathbf{n}\cdot\nabla\times\mathbf{v},
\end{aligned}
\end{equation*}
see \cite{Mitrea1996}. Also, see \cite{Mitrea1996}, for a scalar function $f\in W^{s,2}(\partial D)$ and a vector function $\mathbf{g}\in W^{1-s,2}(\partial D)^3$, for $0\leq s \leq 1$, we have the identity: 
\begin{equation}
\label{sid}
\begin{split}
\int_{\partial D}\mathbf{g}\cdot \mathbf{n}\times \nabla f\,ds=-\int_{\partial D} f(\mathbf{n}\cdot \nabla\times \mathbf{g})\,ds.
\end{split}
\end{equation}

To complete the set up, we introduce the spaces: 
\begin{align*}%\label{ltan}
 L^{2}_{t}(\partial D)^{3}&\,=\,\left\{\boldsymbol{\rho} \in L^{2}(\partial D)^{3}\middle|  \ \  \mathbf{n}\cdot \boldsymbol{\rho}=0 \ \ \text{on} \ \ \partial D\right\},\\
L^{2}_{t,0}(\partial D)^{3}&\,=\,\left\{\boldsymbol{\rho} \in L^{2}_t(\partial D)^{3}\middle|  \ \  Div_S \boldsymbol{\rho}=0 \ \ \text{on} \ \ \partial D\right\},\\
L^{2}_{0}(\partial D)&\,=\,\left\{{\rho} \in L^{2}(\partial D)\middle|  \ \  (\rho,1)_{\partial D}=0  \right\},\\
 H^{-1/2}_{0}(\partial D)&\,=\,\left\{{\rho} \in H^{-1/2}(\partial D)\middle|  \ \  (\rho,1)_{\partial D}=0  \right\},
\end{align*}
where $\displaystyle(\rho,1)_{\partial D}:=\int_{\partial D}\rho\,ds$.
%and
% \begin{equation}\label{v-half}
 %V^{^{-\frac{1}{2}}}_{t}(\partial D)^{3}=\left\{(\mathbf{n}\times \nabla)f\ ; \ f\in H^{\frac{1}{2}}(\partial D) \right\}
% \end{equation}
 %with $L^{2}_{t,0}(\partial D)^{3}\subset V^{^{-\frac{1}{2}}}_{t}(\partial D)^{3}\subset H^{-\frac{1}{2}}(\partial D)^{3}$ see \cite{Mitrea1996}.
 
%\textcolor{red}{Need to recheck this paragraph.} 
%\textcolor{blue}{For a vector field $\mathbf{u}$ defined in $H\cup D$, the boundary trace $\mathbf{u}\big|^\pm_{\partial D}$ ($+$ from outside of $D$ and $-$ from inside of $D$) is assumed to be taken as the nontangential limit almost everywhere with respect to surface measure on the boundary, whenever this exists.}

In order to relate $W_3^\alpha$ to the invariant subspaces of the eigenvalue problem \eqref{ypproblem}, we will introduce a representation of $W_3^\alpha$ given by single layer potentials parameterized by densities on $\partial D$. This is done in the next section.
%Using Lemma~\ref{lem:bbf}, we have for \textcolor{red}{ $\mathbf{u}\in W_3$:$$\|u\|_{W_3}^2=\int_Y\nabla\times\mathbf{u}\cdot\nabla\times\overline{\mathbf{u}}\,dx=\int_Y|\nabla\mathbf{u}|^2\,dx.$$}
\subsection{Mapping Properties of the Single Layer Potential Operator}
 
We start by introducing the $\alpha$-quasiperiodic Green's function:
\begin{equation}
\label{Green-quasi}
{G^\alpha}(x,y)=-\sum_{\mathbf{n}\in \mathbf{Z}^{3}}\dfrac{e^{i(2\pi \mathbf{n}+\alpha)\cdot (x-y)}}{|{\alpha}+2\pi \mathbf{n}|^{2}}I_{3\times3},\,\,\,\,\hbox{for $\alpha\not=0$},
\end{equation}
and the periodic Green's function:
\begin{equation}
\label{Green-periodic}
{G^0}(x,y)=-\sum_{\mathbf{n}\in \mathbf{Z}^{3}\setminus\{0\}}\dfrac{e^{i(2\pi \mathbf{n})\cdot (x-y)}}{|2\pi \mathbf{n}|^{2}}I_{3\times3},\,\,\,\,\hbox{for $\alpha=\mathbf{0}$},
\end{equation}
where $|\cdot|$ is the usual norm of a vector in $\mathbb{R}^3$.
For $\alpha\in Y^\ast$ and $\boldsymbol{\rho}\in  L^{2}_{t,0}(\partial D)^{3}$, we define the $\alpha$-quasiperiodic single layer potential as: 
\begin{equation}\label{S}
\begin{split}
 {S^\alpha}(\boldsymbol{\rho})(x)=\int_{\partial D}{G^\alpha}(x,y)\boldsymbol{\rho}(y)ds_{ y}, \ \ \ x\not \in \partial D.
 \end{split}
 \end{equation}
 %$$ \displaystyle L^{2}_{t,0}( \partial D)^{3}:=\left\{\boldsymbol{\rho} \in L^{2}_{t}(\partial D)^{3}\middle| \ \  \text{Div}_S\boldsymbol\rho=0, \ \ \  \int_{\partial D} \boldsymbol\rho \ ds_{y}=0\right\}.$$
 
 The single layer potential operator satisfies the continuity condition at $x\in\partial D$:
 \begin{equation}\label{jump1}
\displaystyle S^\alpha(\boldsymbol\rho)\big|_{\partial D}^+=S^\alpha(\boldsymbol\rho)\big|_{\partial D}^-,
\end{equation}
\begin{equation}\label{laplace}
-\Delta S^\alpha(\boldsymbol\rho)=0 \text{ for $x\in H\cup D$,}
\end{equation}
 and $S^\alpha(\boldsymbol\rho)\in W_{\#}(\alpha,Y,\mathbb{C}^3)$ with $ S^\alpha(\boldsymbol\rho)\big|_{\partial D}$ in $W^{1/2,2}(\partial D)^3$.  Let $\Gamma^-(x)$ be a truncated circular cone in the interior of $D$ with vertex $x$ and let $\Gamma^+(x)$ be a truncated circular cone in the interior of $H$ with vertex $x$.  Now consider these cones with common vertex $\mathbf{p}$ on $\partial D$.  The boundary trace of a function $f$ at $\mathbf{p}$, $f(\mathbf{p})\big|_{\partial D}^{\pm}$, is given by:
\begin{equation*}
%\label{surface}
\lim_{\stackrel{x\rightarrow\mathbf{p}}{x\,\in\,\Gamma^+(\mathbf{p})}}f(x)=f(\mathbf{p})\big|_{\partial D}^{+}\,\,,\qquad \lim_{\stackrel{x\rightarrow\mathbf{p}}{x\,\in\,\Gamma^-(\mathbf{p})}}f(x)=f(\mathbf{p})\big|_{\partial D}^{-}\,.
\end{equation*}

 We introduce the magnetic dipole operator $M^\alpha: L^{2}_{t,0}( \partial D)^{3}\rightarrow L^{2}_{t,0}( \partial D)^{3}$ given by:
\begin{equation}
    \label{magdiepole}
 M^\alpha(\boldsymbol\rho)=\mathbf{n}\times \left({\rm p.v.}\int_{\partial D} \nabla_{x}\times \left({G}^\alpha (x,y)\  \boldsymbol\rho(y)\right) ds_{y}\right), \qquad x\in \partial D\text{ and }\alpha\in Y^*.
 \end{equation}  We have the following jump conditions for $x\in\partial D$:
 \begin{equation}\label{jump16}
\displaystyle \mathbf{n}\times \nabla_{x}\times S^\alpha(\boldsymbol\rho)\big|_{\partial D}^{\pm}=\pm\frac{1}{2}\boldsymbol\rho+M^\alpha(\boldsymbol\rho).
\end{equation}
For scalar densities $\rho\in L^2(\partial D)$, we recall the jump conditions for $x\in\partial D$:
\begin{equation*}
%\label{jump2}
\displaystyle \mathbf{n}\cdot \nabla_{x} S^\alpha(\rho)\big|_{\partial D}^{\pm}=\mp\frac{1}{2}\rho+(K^{-\alpha})^\ast(\rho),
\end{equation*}
where the Neumann–Poincar\'e operator $(K^{-\alpha})^\ast$ is defined by:
\begin{equation*}
%\label{adjoint}
\displaystyle (K^{-\alpha})^\ast(\rho)={\rm p.v.}\int_{\partial D}\frac{\partial G^\alpha(x,y)}{\partial\mathbf{n}(x)}\rho(y)\,ds_y.
\end{equation*}
Applying Lemma~4.2 of \cite{Mitrea1996} we obtain:
   \begin{equation*}
   %\label{divergence}
    {\rm div}\, S^\alpha(\boldsymbol\rho)(x)=\int_{\partial D}{G^\alpha}(x,y)(\text{Div}_S^\alpha\boldsymbol\rho(y))ds_y,
  \end{equation*}
and:  
 \begin{equation}\label{divergencezero}
    {\rm div}\, S^\alpha(\boldsymbol\rho)(x)=0,
  \end{equation}
since $\boldsymbol\rho$ in $L^{2}_{t,0}(\partial D)^3$. We may extend
Lemma~4.4 of \cite{Mitrea1996} to the periodic and $\alpha$-quasiperiodic cases, see Appendix~\ref{app:extension4.4}, to deliver a commutation relation between the surface divergence, the magnetic dipole and the Neumann–Poincar\'e operator given by:
\begin{equation}\label{dividentity}
{Div}_S M^\alpha(\boldsymbol{\rho})=(K^{-\alpha})^\ast(Div_S\boldsymbol{\rho}),
\end{equation}
where equality holds as elements of $W^{-1}(\partial D)$. It is noted, for future reference, that:
\begin{equation}\label{mapl2}
\boldsymbol{n}\cdot\nabla\times S^\alpha(\boldsymbol{\rho})\,:\,L^2_{t,0}(\partial D)^3\rightarrow L_0^2(\partial D),
\end{equation}
is an isomorphism, see \cite{Mitrea1996}.

The following two lemmas are crucial for the parametrization of $W_{3}^\alpha$ by single layer potentials.
\begin{lemma}\label{mapfroml2t0}
Let the single layer potential operator $S^\alpha$ be defined as in \eqref{S}.  For every $\boldsymbol\rho\in  L^{2}_{t,0}( \partial D)^{ 3}$, we have that   $S^\alpha(\boldsymbol\rho)\in W_{3}^\alpha$.
\end{lemma}
\begin{proof}
First, recall that $[S^\alpha(\boldsymbol \rho)] |_{\partial D}^{\pm}=0$ from \eqref{jump1}, ${\rm div} S^\alpha(\boldsymbol \rho)=0$ in $Y$ from \eqref{divergencezero}, and from \eqref{laplace} it follows that: 
\begin{equation}\label{vectorharmonic}
\nabla\times\nabla\times S^\alpha(\boldsymbol \rho)=\nabla(\nabla\cdot S^\alpha(\boldsymbol \rho))-\Delta S^\alpha(\boldsymbol\rho)=-\Delta S^\alpha(\boldsymbol\rho)=0 \text{,\,\,\, for $x\in H\cup D$.}
\end{equation}
Choosing a smooth $\boldsymbol{w}_2$  in $W_2^\alpha$, we get:
\begin{equation}\label{w2perp}
\int_Y\,\nabla \times S^\alpha(\boldsymbol\rho)\cdot\nabla\times\overline{\boldsymbol{w}}_2\,dx=\int_D\,\nabla \times S^\alpha(\boldsymbol\rho) \cdot\nabla\times\overline{\boldsymbol{w}}_2\,dx.
\end{equation}
Since $\boldsymbol{w}_2 \in W_2^\alpha$, we have that $\nabla\times\boldsymbol{w}_2=0$ in $H$ and, since $H$ is connected, we have $\boldsymbol{w}_2=\nabla \phi$ in $H$, for some scalar potential $\phi$, with $\boldsymbol{w}_2\vert_{\partial D}^-=\boldsymbol{w}_2\vert_{\partial D}^+=\nabla \phi\vert_{\partial D}^+$.  
Integration by parts in \eqref{w2perp}, the application of \eqref{vectorharmonic}, and the fact that            $\boldsymbol{w}_2\vert_{\partial D}^-=\nabla \phi\vert_{\partial D}^+$ give:
\begin{align}
\int_D\nabla \times S^\alpha(\boldsymbol{\rho})\cdot\nabla\times\overline{\mathbf{w}}_2\,dx&=\int_D\nabla\times\nabla\times S^\alpha(\boldsymbol{\rho})\cdot\overline{\mathbf{w}}_2\,dx-
\int_{\partial D^-}\mathbf{n}\times\nabla\times S^\alpha(\boldsymbol{\rho})\cdot\overline{\mathbf{w}}_2\,d s_{x}\notag\\
&=-\int_{\partial D^-}\mathbf{n}\times\nabla\times S^\alpha(\boldsymbol{\rho})\cdot\nabla\phi\,d s_{x}
\label{w2parts}
\end{align}
and, from \eqref{jump16}, we see that:
\begin{align}
\label{BoundaryofD}
\int_{\partial D^-}\,\boldsymbol{n}\times\nabla \times S^\alpha(\boldsymbol\rho)\cdot\nabla\phi\,ds_{x}&=\int_{\partial D^-}\,\left(-\frac{1}2\boldsymbol{\rho}+M^\alpha(\boldsymbol{\rho})\right)\cdot\nabla\phi\,ds_{x}\notag\\
&=\int_{\partial D^-}\,\left(\frac{1}2Div_S\boldsymbol{\rho}-Div_S M^\alpha(\boldsymbol{\rho})\right)\phi\,ds_{x}.
\end{align}
Since $\boldsymbol{\rho}\in  L^{2}_{t,0}(\partial D)^{3}$, from \eqref{dividentity}
we obtain:
\begin{equation}\label{dividentity0}
{Div}_S M^\alpha(\boldsymbol{\rho})=(K^{-\alpha})^\ast(Div_S^\alpha(\boldsymbol{\rho})=0.
\end{equation}
It now follows immediately, from \eqref{w2perp}, \eqref{w2parts}, \eqref{BoundaryofD} and \eqref{dividentity0}, that:
\begin{equation}\label{w2perp0}
\int_Y\,\nabla \times S^\alpha(\boldsymbol\rho)\cdot\nabla\times\overline{\boldsymbol{w}}_2\,dx=0,
\end{equation}
for a dense set of test fields $\boldsymbol{w}_2$ in $W_2^\alpha$, and we conclude that $S^\alpha(\boldsymbol{\rho}) \perp W_2^\alpha$.
Identical arguments can be made for $\boldsymbol{w}_1 \in W_1^\alpha$, to find that:
\begin{equation*}
%\label{OrthogonaltoW1}
\int_Y\,\nabla \times S^\alpha(\boldsymbol\rho)\cdot\nabla\times\overline{\boldsymbol{w}}_1\,dx=0,
\end{equation*}
and the lemma follows.
\end{proof}

Define the Sobolev space:
  \begin{equation*}
  %\label{l2smallA}
  V^{^{-\frac{1}{2}}}_{t}(\partial D)^{3}:=\left\{(\mathbf{n}\times \nabla)f\ : \ f\in W^{{1}/{2},2}(\partial D) \right\},
 \end{equation*}
 with the norm $\|A\|_{V^{^{-\frac{1}{2}}}_{t}(\partial D)^{3}}$ given by:
  \begin{equation*}
  %\label{vnorm}
  \|A\|_{V^{^{-\frac{1}{2}}}_{t}(\partial D)^{3}}=\text{inf}\left\{\| \sigma+ f\|_{W^{{1}/{2},2}(\partial D)}\ : \  \sigma \in \mathbb{C}, \ \ f\in W^{\frac{1}{2},2}(\partial D),\ \  (\mathbf{n}\times \nabla)f=A \right\}.
  \end{equation*}
 Moreover, from \cite{Mitrea1996}, we have:
   \begin{equation*}
   %\label{l2smallL2}
  L^{2}_{t,0}( \partial D)^{3} = V^{0}_{t}(\partial D)^{3}=\left\{(\mathbf{n}\times \nabla)f\ : \ f\in W^{1,2}(\partial D) \right\},
 \end{equation*}
with:
\begin{align}
\mathbf{n}\times \nabla&: W^{1,2}(\partial D)\setminus \mathbb{C}\rightarrow L^{2}_{t,0}( \partial D)^{3},\label{isomorphL20tan}\\
\mathbf{n}\times \nabla&: W^{1/2,2}(\partial D)\setminus \mathbb{C}\rightarrow V^{^{-\frac{1}{2}}}_{t}(\partial D)^3,\label{isomorphv-1half}
\end{align}
isomorphisms, and:
  \begin{equation*}
  %\label{l2smallB}
 L^{2}_{t,0}( \partial D)^{3} \subset  V^{^{-\frac{1}{2}}}_{t}(\partial D)^{3}\subset W^{-{1}/{2},2}(\partial D)^{3}.
 \end{equation*}

We now present the  mapping property of the single layer potential operator necessary for   characterizing the spectrum of the sesquilinear operator $T^\alpha=S^\alpha M^\alpha(S^\alpha)^{-1}$.

\begin{theorem}\label{thm:extW3}
The single layer potential operator can be extended as a bounded linear map from $V^{^{-\frac{1}{2}}}_{t}(\partial D)^3$ to $W_3^\alpha$.
\end{theorem}

\begin{proof}
To prove this theorem, we first show the following lemma.
\begin{lemma}\label{dense1}
The space of tangential vector fields  $L^{2}_{t,0}( \partial D)^{3}$ is a dense subspace of $V^{^{-\frac{1}{2}}}_{t}(\partial D)^{3}$.
 \end{lemma}
\begin{proof}  
Note that, from \eqref{isomorphv-1half}, for ${\boldsymbol{g}}\in V^{^{-\frac{1}{2}}}_{t}(\partial D)^3$ we can write ${\boldsymbol{g}}=\mathbf{n}\times \nabla f$, for some $f\in W^{1/2,2}(\partial D)\setminus \mathbb{C}$.  From the density of $W^{1,2}(\partial D)$ in $W^{1/2,2}(\partial D)$, there exists a sequence $\left\{f_{j}\right\}_{j=1}^{\infty}\in W^{1,2}(\partial D)^2\setminus \mathbb{C}\ \subset\  W^{1/2,2}(\partial D)\setminus \mathbb{C}$ converging to $f$ in $W^{1/2,2}(\partial D)\setminus \mathbb{C}$.  From \eqref{isomorphL20tan}, there are associated functions ${\boldsymbol{g}}_j$ in $L^2_{t,0}(\partial D)^3$ such that ${\boldsymbol{g}}_j={\boldsymbol{n}}\times\nabla f_j$. By the continuity of  the map $\mathbf{n}\times \nabla: W^{1/2,2}(\partial D)\rightarrow V^{^{-\frac{1}{2}}}_{t}(\partial D)^{3}$, we have the existence of a positive constant $C$ such that:
\begin{equation*}
%\label{dense}
\|{\boldsymbol{g}}-{\boldsymbol{g}}_{j}\|_{V^{-\frac{1}{2}}_{t}(\partial D)}=\|\mathbf{n}\times \nabla f-\mathbf{n}\times \nabla f_{j}\|_{V^{-\frac{1}{2}}_{t}(\partial D)}\leq C\|f-f_{j}\|_{W^{1/2,2}(\partial D)\setminus \mathbb{C}}\,,
\end{equation*}
and it follows that  $L^{2}_{t,0}( \partial D)^{3}$ is dense in $V^{^{-\frac{1}{2}}}_{t}(\partial D)^{3}.$ 
\end{proof}
%Here  the space  $H^{\frac{1}{2}}(\partial D)$ of complex scalar valued functions defined on $\partial D$ is given by:
%% $$H^{\frac{1}{2}}(\partial D)=\left\{u\in L^{2}(\partial D),\,\,\frac{|u(x)-u(y)|}{|x-y|^{2}}\in L^{2}(\partial D\times \partial D) \right\}  $$
%with:
%$$\|u\|_{H^{\frac{1}{2}}(\partial D)}=\left(\int_{\partial D}|u|^{2}dx+\int_{\partial D}\int_{\partial D}\frac{|u(x)-u(y)|^{2}}{|x-y|^{2}}dxdy\right)^{\frac{1}{2}},$$

With Lemma~\ref{dense1} in hand, we prove Theorem~\ref{thm:extW3}.  Given ${\boldsymbol{\rho}} \in  L^{2}_{t,0}( \partial D)^{3}$ and   $S^\alpha({\boldsymbol{\rho}})\in W_{3}^\alpha$, we have:
\begin{equation}\label{srhonorm}
\begin{aligned}
\|S^\alpha({\boldsymbol{\rho}})\|^2
%&= \int_{Y}\nabla\times S({\boldsymbol{\rho}})\cdot \nabla\times \overline{S({\boldsymbol{\rho}})}\,d\mathbf{x}\\
&=\int_{H}\nabla\times S^\alpha({\boldsymbol{\rho}})\cdot \nabla\times \overline{S^\alpha({\boldsymbol{\rho}})}\,d\mathbf{x}+\int_{D}\nabla\times S^\alpha({\boldsymbol{\rho}})\cdot \nabla\times \overline{S^\alpha({\boldsymbol{\rho}})}\,d\mathbf{x}\\
&=\int_{\partial D}[\mathbf{n}\times\nabla \times S^\alpha({\boldsymbol{\rho}})]^{-}_{+}\cdot\ \overline{S^\alpha({\boldsymbol{\rho}})}\,ds_\mathbf{x} \\
&=-\int_{\partial D} {\boldsymbol{\rho}}\cdot \ \overline{S^\alpha({\boldsymbol{\rho}})}\,ds_\mathbf{x}.
\end{aligned}
\end{equation}

Writing ${\boldsymbol{\rho}}=\mathbf{n}\times \nabla f$, for $f\in W^{1,2}(\partial D)\setminus \mathbb{C}$, and using \eqref{sid} in \eqref{srhonorm}, we get:
\begin{equation*}
-\int_{\partial D} {\boldsymbol{\rho}} \ \cdot \overline{ S^\alpha({\boldsymbol{\rho}})}\,ds_\mathbf{x} =-\int_{\partial D}  \mathbf{n}\times \nabla f \cdot \ \overline{S^\alpha({\boldsymbol{\rho}})}\,ds_\mathbf{x}=\int_{\partial D} f \  \mathbf{n}\cdot \nabla \times \overline{S^\alpha({\boldsymbol{\rho}})}\,ds_\mathbf{x}.
\end{equation*}
%\begin{equation*}
%\begin{split}
%\|S(\rho)\|^{2}
%&=-\int_{\partial D} {\boldsymbol{\rho}} \ \cdot \overline{S({\boldsymbol{\rho}})}\,ds_\mathbf{x} \\
%&=-\int_{\partial D}  \mathbf{n}\times \nabla f \cdot \ \overline{S({\boldsymbol{\rho}})}\,ds_\mathbf{x}\\
%&=\int_{\partial D} f \  \mathbf{n}\cdot \nabla \times \overline{S({\boldsymbol{\rho}})}\,ds_\mathbf{x}.
%\end{split}
%\end{equation*}

From \eqref{mapl2}, $\mathbf{n}\cdot\nabla\times S^\alpha({\boldsymbol{\rho}})\in  L^2_{0}(\partial D)$, so it also belongs to $W^{-\frac{1}{2},2}_0(\partial D)=(W^{\frac{1}{2},2}(\partial D)\setminus \mathbb{C})'$, where the notation `` $'$ " is used to indicate the dual space.
From \eqref{srhonorm} and the last equation above, for $f \in W^{{1,2}}(\partial D)\setminus \mathbb{C}$, we have:
\begin{equation*}
\|S^\alpha({\boldsymbol{\rho}})\|^{2} 
=\int_{\partial D} f \  \mathbf{n}\cdot \nabla \times \overline{S^\alpha({\boldsymbol{\rho}})}\,ds_\mathbf{x}\leq \inf_{\sigma\in\mathbb{C}}\|f+\sigma\|_{W^{\frac{1}{2},2}(\partial D)}\| \mathbf{n}\cdot\nabla\times S^\alpha({\boldsymbol{\rho}})\|_{W^{-\frac{1}{2},2}_0(\partial D)}\,,
\end{equation*}
%\begin{equation*}
%\label{adid3}
%\begin{split}
%\|S({\boldsymbol{\rho}})\|^{2} 
%&=\int_{\partial D} f \  \mathbf{n}\cdot \nabla \times \overline{S({\boldsymbol{\rho}})}\,ds_\mathbf{x}\\
%&\leq \inf_{\sigma\in\mathbb{C}}\|f+\sigma\|_{W^{\frac{1}{2},2}(\partial D)}\| \mathbf{n}\cdot\nabla\times S({\boldsymbol{\rho}})\|_{W^{-\frac{1}{2},2}_0(\partial D)}\,,
%\end{split}
%\end{equation*}
where $\inf_{\sigma\in \mathbb{C}}\|f+\sigma\|_{W^{\frac{1}{2},2}(\partial D)}$ is the norm for  $W^{{1,2}}(\partial D)\setminus \mathbb{C}$. 
Since the map $ \mathbf{n}\cdot \nabla\times S^\alpha: V^{^{-\frac{1}{2}}}_{t}(\partial D)^{3}\rightarrow W^{-\frac{1}{2},2}_0(\partial D)$ is bounded (see \cite{Mitrea1996}), we have that $\|\mathbf{n}\cdot \nabla\times S^\alpha(\boldsymbol{\rho})\|_{W^{-\frac{1}{2},2}(\partial D)}\leq C \|\boldsymbol{\rho}\|_{V^{^{-\frac{1}{2}}}_{t}(\partial D)^{3}}$ and also $\inf_{\sigma\in\mathbb{C}}\|f+\sigma\|_{W^{\frac{1}{2},2}(\partial D)}=\|\boldsymbol{\rho}\|_{V^{-\frac{1}{2}}(\partial D)^3}$, so it follows that: 
\begin{equation*}
%\label{adid4}
\|S^\alpha({\boldsymbol{\rho}})\|^{2} 
\leq C\inf_{\sigma\in\mathbb{C}}\|f+\sigma\|_{W^{\frac{1}{2},2}(\partial D)}\|{\boldsymbol{\rho}}\|_{V^{-\frac{1}{2}}_{t}(\partial D)^{3}}
\leq C\|{\boldsymbol{\rho}}\|_{V^{-\frac{1}{2}}_{t}(\partial D)^{3}}^{2}
\end{equation*}
%\begin{equation*}
%\label{adid4}
%\begin{split}
%\|S({\boldsymbol{\rho}})\|^{2} 
%&\leq C\inf_{\sigma\in\mathbb{C}}\|f+\sigma\|_{W^{\frac{1}{2},2}(\partial D)}\|{\boldsymbol{\rho}}\|_{V^{-\frac{1}{2}}_{t}(\partial D)^{3}}\\
%&\leq C\|{\boldsymbol{\rho}}\|_{V^{-\frac{1}{2}}_{t}(\partial D)^{3}}^{2},
%\end{split}
%\end{equation*}
and, therefore:
\begin{equation}\label{sbdd}
\|S^\alpha({\boldsymbol{\rho}})\|\leq C \|{\boldsymbol{\rho}}\|_{V^{-\frac{1}{2}}_{t}(\partial D)^{3}}.
\end{equation}

The inequality \eqref{sbdd} implies that $S^\alpha(\boldsymbol{\rho})$ is a bounded operator mapping into $W_3^\alpha$  for the densely defined subspace $L^2_{t,0}(\partial D)^3$ of $ V^{^{-\frac{1}{2}}}_{t}(\partial D)^3$.
Then, we extend the densely defined map $S^\alpha$ to $V_t^{\frac{1}{2}}(\partial D)^3$, using the BLT theorem, to deduce that its extension $S^\alpha: V^{^{-\frac{1}{2}}}_{t}(\partial D)^{3}\rightarrow W_{3}^\alpha$ is bounded.
\end{proof}
%In the sequil we continue to refer to the extension of the single layer potential as the single layer potential $S$.

\begin{theorem}\label{thm:iso}
The single layer potential operator  $S^\alpha: V^{^{-\frac{1}{2}}}_{t}(\partial D)^{3}\rightarrow W_{3}^\alpha$ is a bijection. 
\end{theorem}
\begin{proof}
We first show  that $S^\alpha$ is one-to-one. For a given   $\boldsymbol{\rho}\in V^{^{-\frac{1}{2}}}_{t}(\partial D)^{3}$, we have  $ \mathbf{u}=S^\alpha(\boldsymbol{\rho})\in W_{3}^\alpha$.  Furthermore:
\begin{equation*}
%\label{difference}
\begin{split}
\boldsymbol{\rho}
%&= \mathbf{n}\times \nabla\times  \mathbf{u}\big|_{\partial D^{+}}-  \mathbf{n}\times \nabla\times  \mathbf{u}\big|_{\partial D^{-}}\\
& =  \mathbf{n}\times \nabla\times  \mathbf{u}\big|_{\partial D^{+}}-  \mathbf{n}\times \nabla\times  \mathbf{u}\big|_{\partial D^{-}}
+ \mathbf{n}\times \nabla\times  \mathbf{u}\big|_{\partial Y}-  \mathbf{n}\times \nabla\times  \mathbf{u}\big|_{\partial Y}\\
&= \mathbf{n}\times \nabla\times  \mathbf{u}\big|_{\partial H}- \mathbf{n}\times \nabla\times  \mathbf{u}\big|_{\partial D^{-}}- \mathbf{n}\times \nabla\times  \mathbf{u}\big|_{\partial Y}.
\end{split}
 \end{equation*}
Given a bounded Lipschitz domain $\Omega\in \mathbb{R}^3$, if $ \mathbf{f} \in L^{2}(\Omega)^3$ and $\nabla\times  \mathbf{f} \in L^{2}(\Omega)^3$, then $ \mathbf{n}\times  \mathbf{f}\in W^{-\frac{1}{2},2}(\partial \Omega)^3$.  As a consequence, there is a $C>0$, depending only on $\partial \Omega$, such that:
 \begin{equation*}
 %\label{tracecurlid}
 \| \mathbf{n}\times  \mathbf{f}\|_{W^{-\frac{1}{2},2}(\partial \Omega)^3}\leq C(\| \mathbf{f}\|_{L^{2}(\Omega)^3}+\|\nabla\times  \mathbf{f}\|_{L^{2}(\Omega)^3}).
 \end{equation*}
 Set $\mathbf{f}=\nabla\times \mathbf{u}$ and, since $\nabla\times \nabla\times  \mathbf{u}=0 $ in $H\cup D$, $\boldsymbol{\rho} \in V^{^{-\frac{1}{2}}}_{t}(\partial D)^{3}\subset W^{-\frac{1}{2},2}(\partial D)^3$, one has:
\begin{equation*}
%\label{sbounded}
 \begin{split}
 & \|\boldsymbol{\rho}\|_{W^{-\frac{1}{2},2}(\partial D)^{3}} \\
 & = \| \mathbf{n}\times \nabla\times  \mathbf{u}\big|_{\partial D^{+}}-  \mathbf{n}\times \nabla\times \mathbf{ u}\big|_{\partial D^{-}}\|_{W^{-\frac{1}{2},2}(\partial D)^{3}}\\
 &\leq \| \mathbf{n}\times \nabla\times  \mathbf{u}\|_{W^{-\frac{1}{2},2}(\partial H)^{3}}+ \| \mathbf{n}\times \nabla\times  \mathbf{u}\|_{W^{-\frac{1}{2},2}(\partial D)^{3}}+\| \mathbf{n}\times \nabla\times  \mathbf{u}\|_{W^{-\frac{1}{2},2}(\partial (Y))^{3}}\\
 &\leq C(\| \nabla\times  \mathbf{u}\|_{L^{2}( H)^{3}}+\|\nabla\times  \mathbf{u}\|_{L^{2}( D)^{3}}+\|\nabla\times  \mathbf{u}\|_{L^{2}( Y)^{3}})\\
& \leq C\|  \mathbf{u}\|=C\| S^\alpha(\boldsymbol{\rho})\|.
\end{split}
 \end{equation*}
 
 Now, for  $\boldsymbol{\rho}_{1}, \boldsymbol{\rho}_{2} \in V^{^{-\frac{1}{2}}}_{t}(\partial D)^{3}\subset W^{-\frac{1}{2},2}(\partial D)^3$, we obtain:  
 \begin{equation*}
 %\label{sq}
 0\leq  \|\boldsymbol{\rho}_{1}-\boldsymbol{\rho}_{2}\|_{W^{-\frac{1}{2},2}(\partial D)^{3}}\leq C_{2}\| S^\alpha(\boldsymbol{\rho}_{1})-S^\alpha(\boldsymbol{\rho}_{2})\|,
 \end{equation*}
to conclude that $S^\alpha: V^{^{-\frac{1}{2}}}_{t}(\partial D)^{3}\rightarrow W_{3}^\alpha$ is one-to-one. 

To show the surjectivity of $S^\alpha$, assume that $\mathbf{u}\in W_{3}^\alpha$ is given.  From the definition of $W_3^\alpha$ and integration by parts, we have: 
\begin{equation*}
%\label{curlcurlnull}
\nabla\cdot  \mathbf{u} =0, \quad \quad \nabla\times \nabla\times  \mathbf{u}=0, \text{ on $H\cup D$}.
\end{equation*}
Writing ${\mathbf{w}}=\nabla\times\mathbf{u}$, we see that $\nabla\times\mathbf{w}=0$ in $H\cup D$ so $\mathbf{w}=\nabla q_1$, for $q_1\in W^{1,2}(H)$, and $\mathbf{w}=\nabla q_2$, for $q_2\in W^{1,2}(D)$.  Let $\Gamma^-(x)$ be a truncated circular cone in the interior of $D$ with vertex $x$ and let $\Gamma^+(x)$ be a truncated circular cone in the interior of $H$ with vertex $x$.  Now consider these cones with common vertex $\mathbf{p}$ on $\partial D$.  Taking the cross product of $\mathbf{w}=\nabla\times \mathbf{u}$ with the normal to the surface $\partial D$ given by $\mathbf{n}(\mathbf{p})$, we get:
\begin{equation*}
%\label{surface}
\lim_{\stackrel{x\rightarrow\mathbf{p}}{x\,\in\,\Gamma^+(\mathbf{p})}}\mathbf{n}(\mathbf{p})\times\nabla\times  \mathbf{u}(x)=\mathbf{n}(\mathbf{p})\times   \nabla q_1(\mathbf{p}),\qquad \lim_{\stackrel{x\rightarrow\mathbf{p}}{x\,\in\,\Gamma^-(\mathbf{p})}}\mathbf{n}(\mathbf{p})\times \nabla\times  \mathbf{u}(x)=\mathbf{n}(\mathbf{p})\times   \nabla q_2(\mathbf{p}).
\end{equation*}
%$$Div_S( \mathbf{n}\times\nabla\times  \mathbf{u}) =0, \text{ on $\partial D^\pm$        and       }  [ \mathbf{u}]^-_+=0.$$
From \eqref{isomorphv-1half}, we have that $\mathbf{n}\times\nabla:W^{\frac{1}{2},2}(\partial D)/\mathbb{C}\rightarrow V^{-\frac{1}{2}}_t(\partial D)^3$ is an isomorphism, and we choose:
\begin{equation*}
%\label{rhoregularity}
{\boldsymbol{\rho}}_{ \mathbf{u}}= \mathbf{n}\times   \nabla q_1\big|_{\partial D^{+}}- \mathbf{n}\times   \nabla q_2\big|_{\partial D^{-}} \,\,\, \in V^{-\frac{1}{2}}(\partial D)^3.
\end{equation*}
Setting $ \mathbf{v}=S^\alpha({\boldsymbol{\rho}}_{ \mathbf{u}})$ gives:
\begin{equation}\label{identies}
\begin{aligned}
&\nabla\times \nabla\times  \mathbf{v}=0  \ \ \text{in} \ \  D\cup H,  \qquad \nabla\cdot  \mathbf{v}=0\ \ \text{ in} \ \  Y, \qquad[ \mathbf{n}\times\nabla\times  \mathbf{v}]^{+}_{-}={\boldsymbol{\rho}}_ \mathbf{u},
\end{aligned}
\end{equation}
and:
\begin{equation}
\label{boundary}
\begin{aligned}
\int_{\partial Y}\,\mathbf{n}\times \nabla\times  (\mathbf{v}-\mathbf{u})\cdot (\overline{\mathbf{v}}-\overline{\mathbf{u}})\,ds=0,\qquad \hbox{for $\mathbf{v},\mathbf{u}\in W_3^\alpha$}.
\end{aligned}
\end{equation}
Using integration by parts and applying \eqref{identies} and \eqref{boundary}, we discover:
%\begin{equation*}
%\begin{split}
%( \mathbf{w}- \mathbf{u}, \mathbf{w}- \mathbf{u})_{Y}
%&=-\int_{\partial Y} \mathbf{n}\times\nabla\times ( \mathbf{w}- \mathbf{u})\cdot (\overline{ \mathbf{w}- \mathbf{u}}) \ ds_{y}\\
%& \quad +\int_{\partial D}[ \mathbf{n}\times\nabla\times ( \mathbf{w}- \mathbf{u})]^{+}_{-}\cdot(\overline{ \mathbf{w}- \mathbf{u}})\ %ds_{y}\\
%&\quad+\int_{H\cup D}\nabla\times\nabla\times ( \mathbf{w}- \mathbf{u})\cdot(\overline{ \mathbf{w}- \mathbf{u}})\ dy.
%\end{split}
%\end{equation*}
%Using the antiperiodicity, and the fact that: 
%$$[( \mathbf{n}\times\nabla\times  \mathbf{w})]_{\partial D}=[( \mathbf{n}\times\nabla\times  \mathbf{u})]_{\partial D}=\rho_{ \mathbf{u}}, \  \ \ \nabla\times\nabla\times ( \mathbf{w}- \mathbf{u})=0 \ \ \text{ in} \ \ \  H\cup D,$$
\begin{equation*}
%\label{innerzero}
\Vert\mathbf{v}- \mathbf{u}\Vert=0.
\end{equation*}
%\qquad \text{implies} \qquad  \mathbf{w}-\mathbf{u}=C.$$
For $\alpha\not=0$, this implies $\mathbf{v}=\mathbf{u}$ and, for $\alpha=\mathbf{0}$, we have $\mathbf{u}-\mathbf{v}=\mathbf{c}$, where $\mathbf{c}$ is a constant vector. But, for $\alpha=\mathbf{0}$, we have
 $\displaystyle 0= \int_{Y}\mathbf{w}\ dx=\int_{Y} \mathbf{u}\ dx$,  to conclude $\mathbf{c}=0$ and $\mathbf{v}=\mathbf{u}$.  This  shows that $S^\alpha$ is surjective.
\end{proof}
From Theorem~\ref{thm:iso}, we see that the inverse map $(S^\alpha)^{-1}\,:W_3^\alpha\rightarrow {V^{\frac{1}{2}}_t}(\partial D)^3$ exists. Finally, we apply the open mapping theorem to derive the following theorem.
\begin{theorem}
The inverse $(S^\alpha)^{-1}\,:W_3^\alpha\rightarrow V_t^{\frac{1}{2}}(\partial D)^3$ is  bounded.
\end{theorem}

\subsection{Compactness of Magnetic Dipole Operator}
In this section, we show that the magnetic dipole operator $M^\alpha$ is compact.  

\begin{theorem}\label{mcompact}
The operator $M^\alpha: V^{^{-\frac{1}{2}}}_{t}(\partial D)^{3}\rightarrow V^{^{-\frac{1}{2}}}_{t}(\partial D)^{3}$ is compact and satisfies:
 \begin{equation}\label{specmk}
\sigma(M^\alpha;\ V^{^{-\frac{1}{2}}}_{t}(\partial D)^{3})=\sigma((K^{-\alpha})^\ast;\ H^{-\frac{1}{2}}_{0}(\partial D)),
\end{equation}
 where  $(K^{-\alpha})^\ast$ is the scalar valued  Neumann–Poincar\'e operator defined on $H^{-\frac{1}{2}}_{0}(\partial D)$ and where $\sigma(M^\alpha;\ V^{^{-\frac{1}{2}}}_{t}(\partial D)^{3})$ and $\sigma((K^{-\alpha})^\ast;\ H^{-\frac{1}{2}}_{0}(\partial D))$ are the spectra of  $M^\alpha$ and $(K^{-\alpha})^\ast$, respectively.
  \end{theorem}
  
\begin{proof} We first establish that the magnetic dipole operator $M^\alpha$ is  a bounded map of $V^{^{-\frac{1}{2}}}_{t}(\partial D)^{3}$.  To do this, we start with the following Plemelj-like identity, that can be derived as in\cite{Mitrea1996}:
\begin{equation}\label{plid}
(K^{-\alpha})^\ast(\mathbf{n}\cdot \nabla\times S^\alpha)=\mathbf{n}\cdot \nabla\times S^\alpha M^\alpha, \qquad \text{for} \ \ \boldsymbol{\rho} \in L^{2}_{t,0}( \partial D)^{3},
\end{equation}
The scalar valued Neumann–Poincar\'e  operator is bounded and compact on $H^{-\frac{1}{2}}_{0}(\partial D)$, see \cite{Shapero}.  The map $\mathbf{n}\cdot \nabla\times S^\alpha: V^{^{-\frac{1}{2}}}_{t}(\partial D)^{3}\rightarrow H^{-\frac{1}{2}}_{0}(\partial D)$ can be shown to be an isomorphism, as in \cite{Mitrea1996}. The boundedness of $(K^{-\alpha})^\ast$ and the boundedness of the operator $\mathbf{n}\cdot \nabla\times S^\alpha$ imply that:
\begin{equation}\label{kstar}
\begin{split}
\|(K^{-\alpha})^\ast(\mathbf{n}\cdot \nabla\times S^\alpha(\boldsymbol{\rho}))\|_{H^{-\frac{1}{2}}_{0}(\partial D)}
&\leq C \|\mathbf{n}\cdot \nabla\times S^\alpha(\boldsymbol{\rho})\|_{H^{-\frac{1}{2}}_{0}(\partial D)}\leq C\|\boldsymbol{\rho}\|_{V^{^{-\frac{1}{2}}}_{t}(\partial D)^{3}}.
\end{split}
\end{equation}
On the other hand, the boundedness  of $\mathbf{n}\cdot \nabla\times S^\alpha$ also implies the following: 
\begin{equation}\label{kstar1}
\|M^\alpha(\boldsymbol{\rho})\|_{V^{^{-\frac{1}{2}}}_{t}(\partial D)^{3}}\leq C\|\mathbf{n}\cdot \nabla\times S^\alpha M^\alpha(\boldsymbol{\rho})\|_{H^{-\frac{1}{2}}_{0}(\partial D)}.
\end{equation}
In view of \eqref{plid}, \eqref{kstar}, and \eqref{kstar1}, we have: 
\begin{equation*}
%\label{mbdd}
\|M^\alpha(\boldsymbol{\rho})\|_{V^{^{-\frac{1}{2}}}_{t}(\partial D)^{3}}\leq C \|\boldsymbol{\rho}\|_{V^{^{-\frac{1}{2}}}_{t}(\partial D)^{3}},
\end{equation*}
and we conclude that $M^\alpha(\boldsymbol{\rho})$ is bounded, for $\boldsymbol{\rho} \in L^{2}_{t,0}( \partial D)^{3}\subset V^{^{-\frac{1}{2}}}_{t}(\partial D)^{3}$. Since $L^{2}_{t,0}( \partial D)^{3}$ is dense in $V^{^{-\frac{1}{2}}}_{t}(\partial D)^{3}$, we can extend $M^\alpha$ as a bounded linear map of $V^{^{-\frac{1}{2}}}_{t}(\partial D)^{3}$. 
%\[
%  M(\rho) =
%  \begin{cases}
%                                   M(\rho)& \rho \in L^{2}_{t,0}( \partial D)^{3}   \\ 
 %                                &   \\ 
  %                                 
   %                               \displaystyle \lim_{n \rightarrow \infty} M(\rho_{n}) &  \rho \in \overline{L^{2}_{t,0}( \partial D)^{3}}, \  \ \ \  \{\rho_{n}\}\in L^{2}_{t,0}( \partial D)^{3}, \ \ \rho_{n}\rightarrow \rho \in V^{^{-\frac{1}{2}}}_{t}(\partial D)^{3}.
%                                     \end{cases}
%\]

Next, observe that $\mathbf{n}\cdot\nabla\times S^\alpha: V^{^{-\frac{1}{2}}}_{t}(\partial D)^{3}\rightarrow H^{-\frac{1}{2}}_{0}(\partial D)$ is an isomorphism, so for a bounded sequence $\left\{\boldsymbol{\rho}_{n}\right\}\in V^{^{-\frac{1}{2}}}_{t}(\partial D)^{3}$, we have:
 \begin{equation*}
 %\label{m1}
 \|\mathbf{n}\cdot \nabla\times S^\alpha(\boldsymbol{\rho}_{n})\|_{H^{-\frac{1}{2}}_{0}(\partial D)}\leq C\|\boldsymbol{\rho}_{n}\|_{V^{^{-\frac{1}{2}}}_{t}(\partial D)^{3}},
 \end{equation*}
which shows that $\left\{\mathbf{n}\cdot \nabla\times S^\alpha(\boldsymbol{\rho}_{n})\right\}_{n=1}^{\infty}\in H^{-\frac{1}{2}}_{0}(\partial D)$ is bounded. 
By the compactness of $(K^{-\alpha})^*$, we have that the subsequence $\left\{(K^{-\alpha})^{*}(\mathbf{n}\cdot \nabla\times S^\alpha(\boldsymbol{\rho}_{n_{k}}))\right\}_{k=1}^{\infty}\in H^{-\frac{1}{2}}_{0}(\partial D)$ is Cauchy, 
%\begin{equation*}
%\label{m2}
%\begin{split}
% \left\{K^{*}(\mathbf{n}\cdot \nabla\times S(\boldsymbol{\rho}_{n_{k}}))\right\}_{k=1}^{\infty}\in H^{-\frac{1}{2}}_{0}(\partial D),
% \end{split}
%\end{equation*}
which in turn, by \eqref{plid}, implies that $\left\{\mathbf{n}\cdot \nabla\times S^\alpha(M^\alpha(\boldsymbol{\rho}_{n_{k}}))\right\}_{k=1}^{\infty}\in H^{-\frac{1}{2}}_{0}(\partial D)$ is also Cauchy. 
 Because $\mathbf{n}\cdot\nabla\times S^\alpha: V^{^{-\frac{1}{2}}}_{t}(\partial D)^{3}\rightarrow H^{-\frac{1}{2}}_{0}(\partial D)$ is an isomorphism and $(K^\alpha)^{*}$ is a continuous map, we have for $\{\boldsymbol{\rho}_{n_{k}}\}_{k=1}^\infty$:
   \begin{equation*}
  %\label{m3}
 \|M^\alpha(\boldsymbol{\rho}_{n_{k}})-M^\alpha(\boldsymbol{\rho}_{n_{l}})\|_{V^{^{-\frac{1}{2}}}_{t}(\partial D)^{3}}\\
 \leq C\|\mathbf{n}\cdot \nabla\times S^\alpha(M^\alpha(\boldsymbol{\rho}_{n_{k}}))-\mathbf{n}\cdot \nabla\times S^\alpha(M^\alpha(\boldsymbol{\rho}_{n_{l}}))\|_{H^{-\frac{1}{2}}_{0}(\partial D)}\,,
 \end{equation*}
%   \begin{equation*}
  %\label{m3}
% \begin{split}
%& \|M(\boldsymbol{\rho}_{n_{k}})-M(\boldsymbol{\rho}_{n_{l}})\|_{V^{^{-\frac{1}{2}}}_{t}(\partial D)^{3}}\\
% &\quad\leq C\|\mathbf{n}\cdot \nabla\times S(M(\boldsymbol{\rho}_{n_{k}}))-\mathbf{n}\cdot \nabla\times S(M(\boldsymbol{\rho}_{n_{l}}))\|_{H^{-\frac{1}{2}}_{0}(\partial D)},
%  \end{split}
% \end{equation*}
 and we conclude that the sequence  $\left\{M^\alpha(\boldsymbol{\rho}_{n_{k}})\right\}_{n=1}^{\infty}\in V^{^{-\frac{1}{2}}}_{t}(\partial D)^{3}$ is Cauchy, and thus, $M^\alpha$ is a compact operator on $V^{^{-\frac{1}{2}}}_{t}(\partial D)^{3}$.
Finally, the identity \eqref{specmk} is the direct consequence of \eqref{plid}, and the isomorphic  map $\mathbf{n}\cdot \nabla\times S^\alpha: V^{^{-\frac{1}{2}}}_{t}(\partial D)^{3}\rightarrow H^{-\frac{1}{2}}_{0}(\partial D)$.
\end{proof}
It is noted that the spectrum of $(K^{-\alpha})^*$ lies in $[-1/2,1/2]$ (see e.g., \cite{Shapero}) and, by the previous theorem, %\ref{mcompact}, 
we see that:
\begin{equation}\label{specbd}
 \sigma(M^\alpha;\ V^{^{-\frac{1}{2}}}_{t}(\partial D)^{3})\subset[-1/2,1/2].
 \end{equation}

\subsection{Spectral Property of the operator $T^\alpha=S^\alpha M^\alpha(S^\alpha)^{-1}$}

\begin{theorem} \label{tcompact} The operator $T^\alpha=S^\alpha M^\alpha(S^\alpha)^{-1}: W_{3}^\alpha\rightarrow W_{3}^\alpha$  is Hermitian, compact, and satisfies:
\begin{equation}\label{tm}
\sigma\left(T^\alpha;\ W_{3}^\alpha\right)=\sigma(M^\alpha;\ V^{^{-\frac{1}{2}}}_{t}(\partial D)^{3}).
\end{equation}
\end{theorem}

\begin{proof} First, we show that $T^\alpha:W_{3}^\alpha\rightarrow W_3^\alpha$ is Hermitian. For $\mathbf{u},\,\mathbf{w} \in W_{3}^\alpha$, we have:
\begin{equation*}
\begin{split}
\langle T^\alpha\mathbf{u},\mathbf{ w}\rangle
%&=\int_{Y}(\nabla\times T\mathbf{u})\cdot (\nabla\times \overline{\mathbf{w}})\ dx\\
&=\int_{Y}(\nabla\times S^\alpha M^\alpha(S^\alpha)^{-1}\mathbf{u})\cdot (\nabla\times \overline{\mathbf{w}})\ dx\\
&=\int_{H}(\nabla\times S^\alpha M^\alpha (S^\alpha)^{-1}\mathbf{u})\cdot (\nabla\times \overline{\mathbf{w}})\ dx+\int_{D}(\nabla\times S^\alpha M^\alpha(S^\alpha)^{-1}\mathbf{u})\cdot (\nabla\times \overline{\mathbf{w}})\ dx.
\end{split}
\end{equation*}
Using integration by parts and since $\nabla\times\nabla\times S^\alpha M^\alpha (S^\alpha)^{-1}\mathbf{u})=0$ in $H\cup D$, we see that:
\begin{equation*}
\int_{Y}(\nabla\times S^\alpha M^\alpha(S^\alpha)^{-1}\mathbf{u})\cdot (\nabla\times \overline{\mathbf{w}})\ dx=\int_{\partial D}\left[\mathbf{n}\times \nabla\times  S^\alpha M^\alpha(S^\alpha)^{-1}\mathbf{u}\right]^{+}_{-}\cdot \overline{\mathbf{w}}\ ds_\mathbf{x}.
\end{equation*}
Then, using the jump condition \eqref{jump16}, 
% \begin{equation*}
%\displaystyle \mathbf{n}\times \nabla_{x}\times S(\boldsymbol{\rho})\big|_{\partial D}^{\pm}=\pm\frac{1}{2}\boldsymbol{\rho}+M(\boldsymbol{\rho}),
%\end{equation*}
we obtain $\langle T^\alpha\mathbf{u},\mathbf{ w}\rangle=\int_{\partial D} M^\alpha(S^\alpha)^{-1}\mathbf{u}\cdot \overline{\mathbf{w}}\ ds_\mathbf{x}$.  We can write $\mathbf{u}=S^\alpha\boldsymbol{\beta}$, for some $ \boldsymbol{\beta}\in V^{^{-\frac{1}{2}}}_{t}(\partial D)^{3} $, to get:
\begin{equation*}
\begin{split}
\langle T^\alpha\mathbf{u}, \mathbf{w}\rangle
%&=\int_{\partial D} MS^{-1}S\boldsymbol{\beta}\cdot \overline{\mathbf{w}}\ ds_{y}\\
&=\int_{\partial D} M^\alpha\boldsymbol{\beta}\cdot \overline{\mathbf{w}}\ ds_{x}\\
&=\dfrac{1}{2}\int_{\partial D} [\mathbf{n}\times\nabla \times S^\alpha\boldsymbol{\beta}\big|_{+}-\mathbf{n}\times\nabla \times  S^\alpha\boldsymbol{\beta}\big|_{-}]\cdot \overline{\mathbf{w}}\ ds_{x}.
\end{split}
\end{equation*}
Integration by parts gives:
\begin{equation*}
\begin{split}
&\dfrac{1}{2}\int_{\partial D} [\mathbf{n}\times\nabla\times S^\alpha\boldsymbol{\beta}\big|_{+}-\mathbf{n}\times\nabla\times S^\alpha\boldsymbol{\beta}\big|_{-}]\cdot  \overline{\mathbf{w}}\ ds_{x}\\
&=\dfrac{1}{2}\int_{H}(\nabla\times S^\alpha\boldsymbol{\beta})\cdot (\nabla\times \overline{\mathbf{w}})\ dx- \dfrac{1}{2}\int_{D}(\nabla\times S^\alpha\boldsymbol{\beta})\cdot (\nabla\times \overline{\mathbf{w}})\ dx.
\end{split}
\end{equation*}
Therefore:
\begin{equation}\label{T}
\langle T^\alpha\mathbf{u}, \mathbf{w}\rangle=\dfrac{1}{2}\int_{H}(\nabla\times\mathbf{ u})\cdot (\nabla \times  \overline{\mathbf{w}})\ dx- \dfrac{1}{2}\int_{D}(\nabla\times \mathbf{u})\cdot (\nabla\times \overline{\mathbf{w}})\ dx,
\end{equation}
and $T^\alpha$ is seen to be Hermitian.

Now, the identity given by \eqref{tm} is established.  Consider the eigenvalue eigenvector pair $(\mu,\boldsymbol{\rho})\in \sigma\left(M^\alpha;\ V^{^{-\frac{1}{2}}}_{t}(\partial D)^{3}\right)\times V^{^{-\frac{1}{2}}}_{t}(\partial D)^{3}$ of $M^\alpha\boldsymbol{\rho}=\mu \boldsymbol{\rho}$.  There exists $\mathbf{u}\in W_{3}^\alpha$ such that $\mathbf{u}=S^\alpha\boldsymbol{\rho}$,  and $\boldsymbol{\rho}=(S^\alpha)^{-1}\mathbf{u}$. Therefore, we have $M^\alpha(S^\alpha)^{-1}\mathbf{u}=\mu S^{-1}\mathbf{u}$.  This implies that: 
\begin{equation*}
S^\alpha M^\alpha(S^\alpha)^{-1}\mathbf{u}=\mu S^\alpha(S^\alpha)^{-1}\mathbf{u} \ \ \Rightarrow \ \ \ T^\alpha\mathbf{u}=\mu \mathbf{u},
\end{equation*}
which shows that $\sigma\left(M^\alpha;\ V^{^{-\frac{1}{2}}}_{t}(\partial D)^{3}\right)\subset  \sigma\left(T^\alpha;\ W_{3}^\alpha\right)$.

On the other hand, if we have $T^\alpha\mathbf{u}=\mu \mathbf{u}$, then $S^\alpha M^\alpha(S^\alpha)^{-1}\mathbf{u}=\mu \mathbf{u}$; therefore, multiplying both sides by $(S^\alpha)^{-1}$ gives $M^\alpha(S^\alpha)^{-1}\mathbf{u}=\mu (S^\alpha)^{-1}\mathbf{u}$, and we obtain: $$\sigma\left(T^\alpha;\ W_{3}^\alpha\right)\subset \sigma(M^\alpha;\ V^{^{-\frac{1}{2}}}_{t}(\partial D)^{3}).$$
Finally, the compactness of $T^\alpha=S^\alpha M^\alpha(S^\alpha)^{-1}$ easily follows from the compactness of $M^\alpha$.
\end{proof}
\noindent It now follows from \eqref{T} that the eigenvalue problem $T^\alpha\boldsymbol{u}=\mu\boldsymbol{u}$ is equivalent to \eqref{ypproblem}, so the eigenfunctions form a complete orthonormal system that span $W_3^\alpha$. 

It is clear from Theorems~\ref{mcompact} and \ref{tcompact} that:
\begin{equation*}
%\label{specmkT}
\sigma\left(T^\alpha;\ W_{3}^\alpha\right)=\sigma((K^{-\alpha})^{*};\ H^{-\frac{1}{2}}_{0}(\partial D)),
\end{equation*}
and we denote dependence on $\alpha$ explicitly and write $\mu_i(\alpha)$, $i\in \mathbb{N}$, $\alpha\in Y^\ast$ and make the following definition.
\begin{definition}
\label{DefStruct}
The \textit{structural spectra} for the crystal is given by $\cup_{\alpha\in Y^*}\{\mu_i(\alpha)\}_{i\in \mathbb{N}}$,
where the pairs $\mu_i(\alpha)$, $\boldsymbol{u}_i\in W_2^\alpha$ satisfy:
\begin{equation*}
%\label{structspec}
T^\alpha\boldsymbol{u}_i=\mu_i(\alpha)\boldsymbol{u}_i.
\end{equation*}
\end{definition}

\subsection{Spectral Representation Theorem}

We present a spectral representation of the differential operator appearing in \eqref{Helm}. With this in mind, by Theorem~\ref{tcompact} and \eqref{specbd}, the invariant subspace associated with each eigenvalue $\mu_n(\alpha)$ of $T^\alpha$ is denoted by $E_n=\{u\in W^\alpha_3\,:
\,T^\alpha \boldsymbol{u}=\mu_n(\alpha) \boldsymbol{u}\}$
and the orthogonal projection onto this subspace is denoted by $P_{\mu_n}^\alpha$; here, orthogonality is with respect to the $\langle\cdot,\cdot\rangle$ inner product. We write the projections onto $W_1^\alpha$ and $W_2^\alpha$ as $P^\alpha_1$ and $P^\alpha_2$, respectively. The differential operator appearing in \eqref{Helm} can be factored into the form given by the following theorem.

\begin{theorem}\label{representation}
The vector Laplacian in a photonic crystal admits the representation:
$$\nabla\times\left(a(x)\nabla\times \mathbf{u}(x)\right) =-\Delta_\alpha T_{k}^\alpha\mathbf{u}(x),$$
where $\Delta_\alpha$ is the $\alpha$-quasiperiodic Laplace operator defined on $Y$ and $T_k^\alpha$ is the linear transform associated with the bilinear form $B_{ k}$ defined for $\mathbf{u}(x)\in J_{\#}(\alpha,Y,\mathbb{C}^3)$, see (\ref{ses}). The linear operator $T_{k}^\alpha$ (\ref{tkassocbk}) has the spectral representation, which separates the effect of the contrast $k$ from the underlying geometry of the photonic crystal, given by:
\begin{equation*}
%\label{Trepresentation}
T_{k}^\alpha\mathbf{u}=k\,P^\alpha_{1}\mathbf{u}+P^\alpha_{2}\mathbf{u}+\sum_{\frac{-1}{2}<\mu_{n}(\alpha)<\frac{1}{2}}\left[k\,\left(\frac{1}{2}+\mu_{n}(\alpha)\right)+\left(\dfrac{1}{2}-\mu_{n}(\alpha)\right)\right]P^\alpha_{\mu_{n}}\mathbf{u},
\end{equation*}
where $\{\mu_n(\alpha)\}=\sigma\left(T^\alpha;\ W_{3}^\alpha\right)$, with $W_3^\alpha\subset J_{\#}(\alpha,Y,\mathbb{C}^3)$.
If $k\in\mathbb{C}\setminus Z$, where: \begin{equation}
\label{Zeta}
    Z=\left\{\frac{\mu_n(\alpha)-1/2}{\mu_n(\alpha)+1/2}\right\}_{-1/2\leq\mu_n(\alpha)\leq1/2},
\end{equation}
then $T_{k}^\alpha$ has an inverse and, for $z=k^{-1}$, it is given by:
\begin{equation}
  (T_{k}^\alpha)^{-1}\mathbf{u}=z\,P^\alpha_{1}\mathbf{u}+P^\alpha_{2}\mathbf{u}+\sum_{\frac{-1}{2}<\mu_{n}(\alpha)<\frac{1}{2}}z\left[\left(\frac{1}{2}+\mu_{n}(\alpha)\right)+z\left(\dfrac{1}{2}-\mu_{n}(\alpha)\right)\right]^{-1}P^\alpha_{\mu_{n}} \mathbf{u}. 
\label{Takinv}
\end{equation} 
\end{theorem}

\begin{proof} Let $\mathbf{u}\in J_{\#}(\alpha,Y,\mathbb{C}^3)$.  Note that: 
\begin{equation}
\langle\mathbf{u},\mathbf{v}\rangle=\langle\sum_{i=1}^{\infty}P_{\mu_{i}}^\alpha\mathbf{u},\mathbf{v}\rangle=\langle P_{1}^\alpha\mathbf{u}+P_{2}^\alpha\mathbf{u}+\sum_{-\frac{1}{2}<\mu_{n}(\alpha)<\frac{1}{2}}P_{\mu_{n}}^\alpha\mathbf{u},\mathbf{v}\rangle,
\label{expansionh1}
    \end{equation}
for all $\mathbf{v}\in J_{\#}(\alpha,Y,\mathbb{C}^3)$, from where:
$$\langle T^\alpha\mathbf{u},\mathbf{v}\rangle=\langle\sum_{i=1}^{\infty}\mu_{i}(\alpha)P_{\mu_{i}}^\alpha\mathbf{u},\mathbf{v}\rangle,\qquad \qquad \forall\mathbf{v}\in J_{\#}(\alpha,Y,\mathbb{C}^3).$$
Also, by \eqref{w1}, \eqref{w2} and \eqref{T}, for all $\mathbf{v}\in J_{\#}(\alpha,Y,\mathbb{C}^3)$, we have:
\begin{equation*}
\begin{array}{lll}
\langle T^\alpha\mathbf{u_1},\mathbf{v}\rangle&=&\dfrac{1}{2}\langle\mathbf{u_1},\mathbf{v}\rangle, \qquad \qquad\forall\mathbf{u_{1}}\in W_{1}^\alpha,\\
\langle T^\alpha\mathbf{u_2},\mathbf{v}\rangle&=&-\dfrac{1}{2}\langle \mathbf{u_2},\mathbf{v}\rangle, \qquad \qquad\forall\mathbf{u_{2}}\in W_2^\alpha.
\end{array}
\end{equation*}

%If we let $P_{\mu_i}^\alpha$ be the projection operator acting on the solution $\mathbf{u}$ of \eqref{weak}, and taking it onto the subspace spanned by the eigenfunctions corresponding to the eigenvalues $\mu_{i}(\alpha)$, then  we have:
%\begin{equation}\label{expansionh1}
%\begin{split}
%&\mathbf{u}=P_{1}^\alpha\mathbf{u}+P_{2}^\alpha\mathbf{u}+\sum_{-\frac{1}{2}<\mu_{n}(\alpha)<\frac{1}{2}}P_{\mu_{n}}^\alpha\mathbf{u}.
%=a_{1}\psi_{1}+a_{2}\psi_{2}+\sum_{-\frac{1}{2}<\mu_{n}<\frac{1}{2}}a_{\!_{\mu_{n}}}\psi_{\mu_{n}} ,
%\end{split}
%\end{equation}

By \eqref{expansionh1}, for $\mathbf{u},\mathbf{v}\in J_{\#}(\alpha,Y,\mathbb{C}^3)$, we have:
\begin{equation}\label{bz}
B_{k}(P_{\mu_{n}}^\alpha\mathbf{u},\mathbf{v})=k\int_{H}(\nabla\times P_{\mu_{n}}^\alpha\mathbf{u} )\cdot(\nabla\times \overline{\mathbf{v}})\ dx +\int_{D}(\nabla\times P_{\mu_{n}}^\alpha\mathbf{u})\cdot(\nabla\times \overline{\mathbf{v}})\ dx.
\end{equation}
On the other hand, by \eqref{T}, we know that:
\begin{equation*}
\begin{split}
\langle T^\alpha\,P_{\mu_{n}}^\alpha\mathbf{u},\mathbf{v}\rangle&=\dfrac{1}{2}\int_{H}(\nabla\times P_{\mu_{n}}^\alpha\mathbf{u})\cdot (\nabla \times  \overline{\mathbf{v}})\ dx- \dfrac{1}{2}\int_{D}(\nabla\times P_{\mu_{n}}^\alpha\mathbf{u})\cdot (\nabla\times \overline{\mathbf{v}})\ dx\\
&=\mu_n(\alpha)\int_{H}(\nabla\times P_{\mu_{n}}^\alpha\mathbf{u})\cdot (\nabla \times  \overline{\mathbf{v}})\ dx+\mu_n(\alpha)\int_{D}(\nabla\times P_{\mu_{n}}^\alpha\mathbf{u})\cdot (\nabla\times \overline{\mathbf{v}})\ dx,
\end{split}
\end{equation*}
which implies that:
\begin{equation}\label{hppass}
\int_{H}(\nabla\times P_{\mu_{n}}^\alpha\mathbf{u})\cdot (\nabla \times  \overline{\mathbf{v}})\ dx=\dfrac{\frac{1}{2}+\mu_{n}(
\alpha)}{\frac{1}{2}-\mu_{n}(\alpha)}\int_{D}(\nabla\times P_{\mu_{n}}^\alpha\mathbf{u})\cdot \nabla\times \overline{\mathbf{v}})\, dx.
\end{equation}
We also have:
\begin{equation}\label{ypint}
\int_{D}(\nabla\times P_{\mu_{n}}^\alpha\mathbf{u})\cdot(\nabla\times \overline{\mathbf{v}})\ dx=\left(\dfrac{1}{2}-\mu_{n}(\alpha)\right)\int_{Y}(\nabla\times P_{\mu_{n}}^\alpha\mathbf{u})\cdot(\nabla\times \overline{\mathbf{v}})\ dx,
\end{equation}
from where \eqref{hppass} becomes:
\begin{equation}\label{hpint}
\int_{H}(\nabla\times P_{\mu_{n}}^\alpha\mathbf{u})\cdot (\nabla \times  \overline{\mathbf{v}})\ dx=\left(\dfrac{1}{2}+\mu_{n}(\alpha)\right)\int_{Y}(\nabla\times P_{\mu_{n}}^\alpha\mathbf{u})\cdot(\nabla\times \overline{\mathbf{v}})\ dx.
\end{equation}
Substituting \eqref{ypint} and \eqref{hpint} into \eqref{bz}, we get:
\begin{equation}\label{one}
B_{k}(P_{\mu_{n}}^\alpha\mathbf{u},\mathbf{v})=\left[k\left(\frac{1}{2}+\mu_{n}(\alpha)\right)+\left(\dfrac{1}{2}-\mu_{n}(\alpha)\right)\right]\int_{Y}(\nabla\times P_{\mu_{n}}^\alpha\mathbf{u})\cdot(\nabla\times \overline{\mathbf{v}})\ dx.
\end{equation}
Noting that:
\begin{align}
B_{k}(P_1^\alpha\mathbf{u},\mathbf{v})&=k\int_{H}(\nabla\times P_1^\alpha\mathbf{u})\cdot (\nabla\times \overline{\mathbf{v}})\ dx,\label{two}\\
B_{k}(P_2^\alpha\mathbf{u},\mathbf{v})&=\int_{D}(\nabla\times P_2^\alpha\mathbf{u})\cdot (\nabla\times \overline{\mathbf{v}})\ dx\,,\label{three}
\end{align}
one concludes that:
\begin{equation*}
B_{k}(\mathbf{u},\mathbf{v})=\langle T_{k}^\alpha\mathbf{u},\mathbf{v}\rangle=\langle kP_{1}^\alpha\mathbf{u}+P_{2}^\alpha\mathbf{u}+\sum_{-1/2<\mu_{n}(\alpha)<1/2}\left[k\left(\frac{1}{2}+\mu_{n}(\alpha)\right)+\left(\dfrac{1}{2}-\mu_{n}(\alpha)\right)\right]P_{\mu_{n}}^\alpha\mathbf{u}, \ \mathbf{v}\rangle,
\end{equation*}
and Theorem~\ref{representation} easily follows since $-\Delta_\alpha$ is the operator related to the bilinear form $\langle\mathbf{u},\mathbf{v}\rangle$.
\end{proof}

\section{Band Structure for Complex Coupling Constant }
\label{bandstructure}
We recall that $a(x)=(\epsilon(x))^{-1}$ and the operator representation is applied to write the Bloch eigenvalue problem as:
\begin{equation}
 \nabla\times (( \epsilon(x))^{-1}\nabla \times\mathbf{h}) =-\Delta_\alpha \,T_{k}^\alpha\mathbf{h}=\xi\mathbf{h}.
\label{representationform} 
\end{equation}
We characterize the Bloch spectra by analyzing the operator: 
\begin{equation}
B^\alpha(k):=(T_k^\alpha)^{-1}(-\Delta_{\alpha})^{-1},
\label{inverseoperatorquadform}
\end{equation}
where the operator $(-\Delta_{\alpha})^{-1}:L_{\#}^2(\alpha, Y, \mathbb{C}^3)\rightarrow J_{\#}(\alpha, Y, \mathbb{C}^3)$, defined for all $\alpha\in Y^\ast$, is given by:
\begin{equation}
(-\Delta_{\alpha})^{-1}\mathbf{ u}(x)=-\int_Y G^\alpha(x,y) \mathbf{u}(y)\,dy.
\label{inverselaplacian}
\end{equation}

Let us suppose $\alpha\neq\mathbf{0}$.
The operator $B^\alpha(k): L^2_{\#}(\alpha,Y,\mathbb{C}^3) \longrightarrow J_{\#}(\alpha,Y,\mathbb{C}^3)$ is easily seen to be bounded for $k\notin Z$ ($\ref{Zeta}$), see Theorem~\ref{bounded}. Since $H^1_{\#}(\alpha,Y,\mathbb{C}^3)$ (and hence $J_{\#}(\alpha,Y,\mathbb{C}^3)$) embeds compactly into $L^2_{\#}(\alpha,Y, \mathbb{C}^3)$, we find that $B^\alpha(k)$ is a bounded compact linear operator on $L^2_\#(\alpha,Y,\mathbb{C}^3)$ (see Theorem~\ref{compact2}) and, therefore, it has a discrete spectrum $\{ \gamma_i(k,\alpha) \}_{i \in \mathbb{N}}$, with a possible accumulation point at $0$. The corresponding  eigenspaces are finite-dimensional and the eigenfunctions $\mathbf{p}_i\in L^2_{\#}(\alpha,Y,\mathbb{C}^3)$ satisfy:
\begin{eqnarray}
B^\alpha(k)\mathbf{p}_i(x)=\gamma_i(k,\alpha)\,\mathbf{p}_i(x),\,\hbox{ for $x\in Y$,}
\label{compactproblem}
\end{eqnarray}
and also belong to $J_\#(\alpha,Y,\mathbb{C}^3)$. Observe that, for $\gamma_i \neq 0$, \eqref{compactproblem} holds if and only if \eqref{representationform} holds with $\xi= \lambda_i(k,\alpha) =\gamma_i^{-1}(k,\alpha)$, and $-\Delta_\alpha T_k^\alpha \mathbf{p}_i=\lambda_i(k,\alpha) \mathbf{p}_i$.  Collecting results, we have the following theorem.

\begin{theorem}
\label{extension}
The Bloch eigenvalue problem \eqref{Helm} for the operator $-\nabla \times (k\chi_H+\chi_D)\nabla \times$, associated with the sesquilinear form \eqref{ses}, can be extended for values of the coupling constant $k$ off the positive real axis into $\mathbb{C}\setminus Z$ ($Z$ given by (\ref{Zeta})), i.e., for each $\alpha\in Y^\star$, the Bloch eigenvalues are of finite multiplicity and denoted by $\lambda_j(k,\alpha)=\gamma_j^{-1}(k,\alpha)$, $j\in \mathbb{N}$, and the band structure (\ref{DispersionRelations}):
\begin{equation*}
\lambda_j(k,\alpha)=\xi,\hbox{ $j\in\mathbb{N}$}
\end{equation*}
extends to complex coupling constants $k \in \mathbb{C}\setminus Z$.
\end{theorem}

%\textcolor{red}{Sections 2 \& 3 edited  (RL) July 2, 2021 --------}

\section{Power Series Representation of Bloch Eigenvalues for High Contrast Periodic Media }
\label{asymptotic}
In what follows, we set $\gamma=\lambda^{-1}(k,\alpha)$ and analyze the spectral problem:
\begin{equation}
B^\alpha(k) \mathbf{u} = \gamma (k,\alpha) \mathbf{u}.
\label{forpointtwo}
\end{equation}
Henceforth, we will analyze the high contrast limit by developing a power series in $z=1/k$, about $z=0$, for the spectrum of the family of operators \eqref{inverseoperatorquadform} associated with \eqref{forpointtwo}:  
$$\begin{array}{lcl}
B^{\alpha}(k) & = & (T_k^\alpha)^{-1}(-\Delta_{\alpha})^{-1}\\
& = &  (zP_1^\alpha + P_2^\alpha + z\sum \limits_{-1/2<\mu_i(\alpha)<1/2}[(1/2 + \mu_i(\alpha)) + z(1/2-\mu_i(\alpha))]^{-1} P_{\mu_i}^\alpha)(-\Delta_{\alpha})^{-1}\\
&=:& A^\alpha(z).
\end{array}$$
Here, we define the operator $A^\alpha(z)$ such that $A^\alpha(1/k)=B^\alpha(k)$, and the associated eigenvalues $\beta(1/k,\alpha)=\gamma(k,\alpha)$.  Then, the spectral problem becomes $A^\alpha(z)\mathbf{u}=\beta(z,\alpha)\mathbf{u}$, for $\mathbf{u}\in L^2_{\#}(\alpha,Y,\mathbb{C}^3)$.

It is easily seen, from the above representation, that  $A^{\alpha}(z)$ is self-adjoint for $k\in\mathbb{R}$ and is a family of bounded operators taking $L^2_{\#}(\alpha,Y,\mathbb{C}^3)$ into itself.  Also, we have the following lemma.

\begin{lemma}
\label{inverseoperator}
$A^{\alpha}(z)$ is holomorphic on $\Omega_0 := \mathbb{C} \setminus \mathcal{S}$, where $\mathcal{S}=\displaystyle\cup_{i\in\mathbb{N}} z_i(\alpha)$ is the collection of points $z_i(\alpha)=(\mu_i(\alpha)+1/2)/(\mu_i(\alpha)-1/2)$ on the negative real axis associated with
the eigenvalues $\{\mu_i(\alpha)\}_{i\in\mathbb{N}}$. The set $\mathcal{S}$ consists of poles  of $A^\alpha(z)$ with only one accumulation point at $z=-1$.
\end{lemma}
The upper bound $z^*(\alpha)$ on $\mathcal{S}$ for fixed $\alpha\in Y^\ast$ is written:
\begin{eqnarray}
 \max_i\{z_i(\alpha)\}=z^\ast(\alpha)<0.
\label{bdsonS}
\end{eqnarray}
In Section \ref{radiusgeneralshape}, we develop explicit lower bounds on the structural spectrum, i.e.: 
$$-1/2<\mu^-\leq \mu_i(\alpha)\in \cup_{\alpha\in Y^*}\{\mu_i(\alpha)\}_{i\in \mathbb{N}}$$ that holds for a generic class of inclusion domains $D$. The corresponding upper bound $z^+$ is written:
\begin{eqnarray}
 \max\{z^*(\alpha); \alpha\in Y^{\ast}\}=\frac{\mu^-+1/2}{\mu^--1/2}=z^+<0,
\label{bdsonSall}
\end{eqnarray}
and $z^*(\alpha)\leq z^+$.

Let $\beta^\alpha_0 \in \sigma(A^{\alpha}(0))$ with spectral projection $P(0)$, and let $\Gamma$ be a closed contour in $\mathbb{C}$ enclosing $\beta^\alpha_0$, but no other $\beta \in \sigma(A^{\alpha}(0))$.
The spectral projection associated with $\beta^\alpha (z) \in \sigma(A^{\alpha}(z))$, for $\beta^\alpha(z) \in \text{int}(\Gamma)$, is denoted by $P(z)$. We write $\mathcal{M}(z) = P(z)L^2_{\#}(\alpha,Y,\mathbb{C}^3)$ and suppose, for the moment, that $\Gamma$ lies in the resolvent of $A^\alpha(z)$ and $\text{dim}(\mathcal{M}(0)) =\text{dim}(\mathcal{M}(z))= m$, realizing that Theorems \ref{separationandraduus-alphanotzero} and \ref{separationandraduus-alphazero}  provide explicit conditions for when this holds true.  Now define $\hat{\beta}^\alpha(z) := \frac{1}{m} \text{tr}(A^{\alpha}(z)P(z))$, the weighted mean of the eigenvalue group $\{ \beta^\alpha_1(z), \ldots ,\beta^\alpha_m(z) \}$ corresponding to $\beta^\alpha_1(0) = \ldots = \beta^\alpha_m(0) = \beta^\alpha_0$.   We  write the weighted mean as:
\begin{equation*}
%\label{analyticforband}
\hat{\beta}^\alpha(z) = \beta^\alpha_0 +  \frac{1}{m} \text{tr}[(A^{\alpha}(z)-\beta^\alpha_0)P(z)].
\end{equation*}
Since $A^{\alpha}(z)$ is analytic in a neighborhood of the origin, we write:
\begin{equation*}
A^{\alpha}(z) = A^{\alpha}(0) + \sum_{n=1}^{\infty} z^n A^{\alpha}_n.
\end{equation*}
The explicit form of the sequence $\{A^{\alpha}_n\}_{n\in \mathbb{N}}$ is given in Section \ref{radius}. Define the resolvent of $A^{\alpha}(z)$ by:
$$R(\zeta, z) = (A^{\alpha}(z) - \zeta )^{-1}\text{;}$$
and expanding successively in Neumann series and power series, we have the identity:
\begin{align}
\label{foursix}
R(\zeta,z) & =  R(\zeta,0)[I + (A^{\alpha}(z) - A^{\alpha}(0))R(\zeta,0)]^{-1} \notag\\
& =  R(\zeta,0)+\sum_{p=1}^\infty [-(A^{\alpha}(z) - A^{\alpha}(0))R(\zeta,0)]^p\\
& =  R(\zeta,0) + \sum_{n=1}^{\infty} z^n R_n(\zeta),\notag
\end{align}
where:
$$R_n(\zeta) = \sum_{k_1 + \ldots k_p = n, k_j \geq 1} (-1)^pR(\zeta,0)A^{\alpha}_{k_1}R(\zeta,0)A^{\alpha}_{k_2}\ldots R(\zeta,0)A^{\alpha}_{k_p},\hspace{2mm} \text{for $n\geq 1$.}$$

Application of  the contour integral formula for spectral projections \cite{SzNagy}, \cite{TKato1}, \cite{TKato2}, delivers the expansion for the spectral projection:
\begin{equation}
\label{Project1}
P(z)  = \displaystyle -\frac{1}{2\pi i}\displaystyle \oint_{\Gamma} R(\zeta, z)d\zeta = P(0) + \sum_{n=1}^\infty z^n P_n,
\end{equation}
where $P_n =  -\frac{1}{2\pi i} \oint_{\Gamma} R_n(\zeta)d\zeta$.  Now, we develop the series for the  weighted mean of the eigenvalue group. Start with:
\begin{eqnarray*}
(A^{\alpha}(z) - \beta^\alpha_0)R(\zeta,z) = I + (\zeta - \beta^\alpha_0)R(\zeta,z),
%\label{project2}
\end{eqnarray*}
and we have:
\begin{equation*}
(A^{\alpha}(z) - \beta^\alpha_0)P(z) = - \frac{1}{2 \pi i}\oint_{\Gamma} (\zeta - \beta^\alpha_0)R(\zeta,z)d\zeta \text{,}
\end{equation*}
so:
\begin{equation}
\hat{\beta}(z) - \beta^\alpha_0 = - \frac{1}{2m \pi i} \text{tr}\oint_{\Gamma} (\zeta - \beta^\alpha_0)R(\zeta,z)d\zeta.
\label{fourten}
\end{equation}
Equation \eqref{fourten} delivers an analytic representation formula for a Bloch eigenvalue or, more generally, the eigenvalue group when $\beta^\alpha_0 $ is not a simple eigenvalue.
Substituting the third line of \eqref{foursix} into \eqref{fourten} yields:
\begin{equation}
\hat{\beta}^\alpha(z) = \beta^\alpha_0 + \sum_{n=1}^\infty z^n\beta^\alpha_n,
\label{foureleven}
\end{equation}
where:
\begin{equation}
{\beta}^\alpha_n = - \frac{1}{2m \pi i}\, \text{tr}\sum_{k_1+\cdots+k_p=n} \frac{(-1)^p}{p}\oint_{\Gamma}A^{\alpha}_{k_1}R(\zeta,0)A^{\alpha}_{k_2}\ldots R(\zeta,0)A^{\alpha}_{k_p}R(\zeta,0)d\zeta;\,\hbox{ $n\geq 1$}.
\label{fourtwelve}
\end{equation}

\section{Spectrum in the High Contrast Limit, $\alpha\not=0$}
\label{limitspectra}
We investigate the spectrum of the limiting operator $A^{\alpha}(0)$, for $\alpha\not=0$.  Using the representation:
\begin{equation*}
A^{\alpha}(z) = (zP_1^\alpha + P_2^\alpha + z\sum \limits_{-\frac{1}{2} < \mu_i(\alpha) < \frac{1}{2}} [(1/2 + \mu_i(\alpha)) + z(1/2-\mu_i(\alpha))]^{-1}P_{\mu_i}^\alpha)(-\Delta_{\alpha})^{-1} \text{,}
\end{equation*}
we see that $A^{\alpha}(0) = P_2^\alpha(-\Delta_{\alpha})^{-1}$;
%\begin{equation*}
%A^{\alpha}(0) = P_2^\alpha(-\Delta_{\alpha})^{-1};
%\end{equation*}
and, from Theorem~\ref{compact2}, we get that $P_2^\alpha(-\Delta_{\alpha})^{-1}$ is a bounded compact operator and has a discrete spectrum.
Denote the spectrum of $A^{\alpha}(0)$ by $\sigma(A^{\alpha}(0))$.  Since $A^\alpha(0)$ is clearly self-adjoint and compact, it follows that $\sigma(A^{\alpha}(0)) \subset \mathbb{R}$ is discrete, with only one possible cluster point at zero. Next, we show that it is strictly positive as well.

We now consider the eigenvalue problem:
\begin{equation}
    \label{A0eigprob}
P_2^\alpha(-\Delta_{\alpha})^{-1}\mathbf{u}=\beta \mathbf{u},
\end{equation}
with $\beta \in \sigma(A^{\alpha}(0))$ and eigenfunction $\mathbf{u}\in L^2_{\#}(\alpha,Y,\mathbb{C}^3)$.
This eigenvalue problem is  equivalent to finding $\beta$ and $\mathbf{u}\in W_2^\alpha$ for which:
\begin{equation}
(\mathbf{u}, \mathbf{v})_{L^2(Y, \mathbb{C}^3)} = \beta \langle \mathbf{u}, \mathbf{v} \rangle, \hbox{ for all $\mathbf{v}\in J_{\#}(\alpha,Y,\mathbb{C}^3)$}.
\label{eqiveigennot0}
\end{equation}
Indeed, to see the equivalence, note that we have $P_2^\alpha(-\Delta_{\alpha})^{-1}:L^2_{\#}(\alpha,Y,\mathbb{C}^3)\rightarrow W_2^\alpha$ and, for $\mathbf{v} \in J_{\#}(\alpha,Y,\mathbb{C}^3)$, it holds:
%To conclude we  will then show that the set of eigenvalues for \eqref{eqiveigenH01} is given by $\sigma(-\Delta^{-1}_D)$. 
\begin{equation*}
%\label{fivefour}
\begin{array}{lcl}
\langle P_2^\alpha(-\Delta_{\alpha})^{-1}\mathbf{u}, \mathbf{v} \rangle =  \beta \langle \mathbf{u}, \mathbf{v} \rangle=\beta \langle P_2^\alpha\mathbf{u}, \mathbf{v} \rangle; \end{array}
\end{equation*}
hence:
\begin{equation}
\label{fivefour2}
\begin{array}{lcl}
\langle (-\Delta_{\alpha})^{-1}\mathbf{u}, P_2^\alpha\mathbf{v}\rangle= \beta \langle \mathbf{u}, P_2^\alpha\mathbf{v} \rangle.
\end{array}
\end{equation}
Since $\langle (-\Delta_{\alpha})^{-1}\mathbf{u}, \mathbf{v}\rangle = \int_Y \mathbf{u} \cdot \overline{\mathbf{v}}\,dx=(\mathbf{u}, \mathbf{v})_{L^2(Y,\mathbb{C}^3)}$, for any $\mathbf{u}\in L^2_\#(\alpha,Y,\mathbb{C}^3)$ and $\mathbf{v} \in J_{\#}(\alpha,Y,\mathbb{C}^3)$, equation \eqref{fivefour2} becomes:
\begin{equation*}
%\label{projectedalpnanotzero}
(\mathbf{u}, P_2^\alpha\mathbf{v})_{L^2(Y,\mathbb{C}^3)}= \beta \langle \mathbf{u}, P_2^\alpha\mathbf{v} \rangle,
\end{equation*}
and the equivalence follows by noticing that $P_2^\alpha$ is the  projection of $J_{\#}(\alpha,Y,\mathbb{C}^3)$ onto $W_2^\alpha$. 

Rewriting \eqref{eqiveigennot0} as:
\begin{equation*}
\int_D \nabla \times \mathbf{u} \cdot \nabla \times \overline{\mathbf{v}}\,dx = \beta^{-1} \int_Y \mathbf{u} \cdot \overline{\mathbf{v}}\,dx,
\end{equation*}
we define the sesquilinear form $b_0(\cdot , \cdot): W_2^\alpha \times W_2^\alpha \rightarrow \mathbb{C}$ by:
\begin{equation*}
%    \label{b0}
    b_0(\mathbf{u},\mathbf{v}) := \int_D \nabla \times \mathbf{u} \cdot \nabla \times \overline{\mathbf{v}}\,dx.
\end{equation*}
Clearly $b_0$ is bounded and we wish to show that the spectrum is positive. To this end we introduce the following lemma.

\begin{lemma}
\label{b0coerciveW2}
For all $\mathbf{u} \in W_2^\alpha$, there exists $C>0$ such that: 
  \begin{equation}
  \label{b0lowerbound}
      b_0(\mathbf{u},\mathbf{u}) \geq C \int_Y |\mathbf{u}|^2\,dx.
  \end{equation}
\end{lemma}
\begin{proof}
  Suppose \eqref{b0lowerbound} does not hold.  Note that, for each $n=1,2,\ldots$, there exists $\mathbf{u}_n \in W_2^\alpha$, for which:
  \begin{equation*}
     n\,\,\int_D \nabla \mathbf{u}_n : \nabla  \overline{\mathbf{u}_n}\,dx \,=\, n\,\,\int_D \nabla \times \mathbf{u}_n \cdot \nabla \times \overline{\mathbf{u}_n}\,dx\,\, <\,\, \int_Y |\mathbf{u}_n|^2\,dx.
  \end{equation*}
  Then, on normalizing $\mathbf{u}_n$ with respect to the $L^2$-norm, there exists a sequence $\{\mathbf{v}_n \} \subset W_2^\alpha$, with $\| \mathbf{v}_n \|_{L^2(Y,\mathbb{C}^3)} = 1$ and $\nabla \mathbf{v}_{n} \rightarrow 0$ strongly in $L^2_{\#}(\alpha, Y, \mathbb{C}^3)$.  After possibly passing to a subsequence, we apply standard arguments to conclude that $\mathbf{v}_n \rightarrow \mathbf{v}$ strongly in $J_{\#}(\alpha, Y, \mathbb{C}^3)$, such that $\mathbf{v}$ is constant in $Y$ and $\| \mathbf{v} \|_{L^2(Y,\mathbb{C}^3)} = 1$.
  But the only constant function in $J_{\#}(\alpha, Y, \mathbb{C}^3)$, for $\alpha\not=0$, is the zero function; which leads to a contradiction.
\end{proof}
In light of Lemma~\ref{b0coerciveW2}, we conclude that the problem \eqref{A0eigprob} has a positive, decreasing sequence of eigenvalues, with a possible cluster point only at zero.

%In the periodic case, i.e. $\alpha = 0$,

%To conclude we  show that the set of eigenvalues for \eqref{eqiveigenH01} is given by %$\sigma(-\Delta^{-1}_D)$. Note that $P_2^\alpha v$ is supported in $D$, so
%\begin{equation}
%\lambda^{-1} \int \limits_{D} u\overline{P_2^\alpha v} = \int \limits_{D} \nabla u \cdot \nabla %\overline{P_2^\alpha v}.
%\end{equation}
%Now since $P_2^\alpha: H^1_{\#}(\alpha,Y) \rightarrow W_2=\tilde{H}_0^1(D)$ is onto, it follows that %$\lambda^{-1}$ is a Dirichlet eigenvalue of the negative Laplacian acting on $D$ and the proof of %Theorem \ref{equiv1} is complete.

\section{Spectrum in the High Contrast Limit: Periodic Case, $\alpha=\mathbf{0}$}
\label{limitspeczero}

We describe the spectrum of the limiting operator $A^{0}(0)$, which is written as $A^0(0) = P_2^0(-\Delta_0)^{-1}$,
%\begin{equation}
%\label{perdeflimit}
%A^0(0) = P_2^0(-\Delta_0)^{-1},
%\end{equation}
where $P_2^0$ is the projection onto  $W _2^0$.
Here, the operator $(-\Delta_0)^{-1}$ is compact and self-adjoint on $L^2_{\#}(0,Y,\mathbb{C}^3)$, and given by \eqref{inverselaplacian} for $\alpha=\mathbf{0}$.  Denote the spectrum of $A^{0}(0)$ by $\sigma(A^{0}(0))$.  In this case we see, as in the case $\alpha\neq\mathbf{0}$ of the previous section, that $\sigma(A^{0}(0))\subset \mathbb{R}_+$ is discrete, with only one possible cluster point at zero. 
%To characterize this spectrum, we introduce the sequence of numbers $\{\nu_j\}_{j\in\mathbb{N}}$, given by the positive roots $\nu$ of the spectral function $S(\nu)$ defined by:
%\begin{equation}
%S(\nu)=\nu\sum \limits_{i\in \mathbb{N}} \frac{ a^2_i}{\nu-\delta^*_i}-1,
%\label{roots}
%\end{equation}
%where $\{\delta^*_j\}_{j\in\mathbb{N}}$ are the Dirichlet eigenvalues for $-\Delta_D$, associated with eigenfunctions $\psi_j$, for which $\int_D\psi_j\,dx\not=0$ and $a_j=|\int_D\psi_j\,dx|$. 
%The following theorem provides the explicit characterization of $\sigma(A^{\alpha}(0))$. 
%\begin{theorem}
%\label{equiv2}
%Let $\{\delta'_j\}_{j\in\mathbb{N}}$ denote the collection of Dirichlet eigenvalues for $-\Delta_D$, associated with eigenfunctions $\psi_j$, for which $\int_D\psi_j\,dx=0$. 
%Then, $\sigma(A^{0}(0)) = \{{\delta'_j}^{-1}\}_{j\in\mathbb{N}}\cup\{{\nu_j}^{-1}\}_{j\in\mathbb{N}}$.
%\end{theorem}

As in \cite{Bouchitte2017}, one can define: 
\begin{definition} The geometric average is a path integral with components defined by:
$$(\oint \mathbf{u})\cdot\mathbf{e}^i:=\int_{\Gamma_i}\,\mathbf{u}\cdot\mathbf{e}^i\,d\ell,$$
where $\Gamma_i$ is any curve in $H$ connecting two opposite points on the faces of $\partial Y$ orthogonal to $\mathbf{e}^i$ and $d\ell$ is an element of arc-length.
\end{definition}
The goal is to precisely identify $\sigma(A^{0}(0))\subset \mathbb{R}_+$.  With that in mind, we introduce the spaces:
\begin{align*}
%\label{FY}
F(Y)&=\left\{\mathbf{u}\in H^1_{loc}(\mathbb{R}^3,\mathbb{C}^3)\,:\,\mathbf{u} \text{ periodic on } Y\,,\,\nabla\cdot\mathbf{u}=0\,\text{in}\,Y,\,\,\nabla\times\mathbf{u}=0\,\text{in}\,\,H\right\}\\
%\label{Chi0div}
\chi_0^{div}&=\left\{\mathbf{u}\in F(Y):\,\oint\mathbf{u}=0\right\}.
\end{align*}
A characterization of the space $W_2^0$ is given by the following lemma.
\begin{lemma}\label{lem:W2} Let $\chi_Y$ be the characteristic function of $Y$.  We have: 
\begin{equation}
\label{w2rep}
W_2^0=\tilde{W}_2=\left\{\mathbf{u}=\tilde{\mathbf{u}}-\left(\int_Y\,\tilde{\mathbf{u}}\,dx\right)\chi_{Y}\,:\,\tilde{\mathbf{u}}\in\chi^{div}_0\right\}.
\end{equation}
\end{lemma}
\begin{proof}
Consider the space $F(Y)$.  The curl-free condition in $H$, together with the $Y$-periodicity condition, implies that $\mathbf{u}=\nabla\varphi+\mathbf{c}$ in $H$, where $\varphi\in W^{1,2}_{\#}(H)$ and $\oint\mathbf{u}=\mathbf{c}\in\mathbb{C}^3$. From this, we can conclude that $\chi^{div}_0\oplus\mathbb{C}^3= F(Y)$ and that:
\begin{equation*}
%\label{altW2}
 W_2^0=\left\{F(Y)\,:\,\,\int_Y\,\mathbf{u}\,dx=0\right\} = \left\{\mathbf{u}\in\chi^{div}_0\oplus\mathbb{C}^3\,:\,\,\int_Y\,\mathbf{u}\,dx=0\right\}.
 \end{equation*}
To see that $W_2^0=\tilde{W}_2$, we introduce the orthonormal system $\{\mathbf{u}_j\}_{j\in\mathbb{N}}$ in $L_{\#}^2(0,Y,\mathbb{C}^3)$ that is dense in $\chi_0^{div}$ with respect to the $W^{1,2}(Y,\mathbb{C}^3)$-norm, and is given by the eigenvectors of \eqref{eqiveigen}, see Theorem~\ref{eigenperiod} below.  Then:
\begin{equation*}
%\label{specF(y)}
F(Y)=\left\{\mathbf{u}\in \rm{span}\left\{\mathbf{u}_j\right\}_{j\in \mathbb{N}}\oplus\,\rm{span}\left\{\mathbf{e}^1,\mathbf{e}^2,\mathbf{e}^3\,\right\}\right\},
\end{equation*}
and an element $\mathbf{u}$ of $F(Y)$ is written:
\begin{equation*}
%\label{span}
\mathbf{u}=\sum_{j=1}^\infty{c}_j\mathbf{u}_j+a_1\mathbf{e}^1+a_2\mathbf{e}^2+a_3\mathbf{e}^3.
\end{equation*}
%\textcolor{red}{where $c_j, a_i\in\mathbb{C}$, for $j\in\mathbb{N}$ and $i=1,2,3$.}  
From this, we see that the condition $\int_Y\,\mathbf{u}\,dx=0$ is equivalent to:
\begin{equation*}
%\label{mean}
a_k=-\mathbf{e}^k\cdot\sum_{j=1}^\infty\,\int_Y\,{c}_j\mathbf{u}_j\,dx, \,\,\text{ for $k=1,2,3$}.
\end{equation*}
We define:
\begin{equation*}
%\label{spanchi}
\tilde{\mathbf{u}}=\sum_{j=1}^\infty{c}_j\mathbf{u}_j\in\,\chi_0^{div},
\end{equation*}
to discover $\mathbf{u}=\tilde{\mathbf{u}}-\,\int_Y\,\tilde{\mathbf{u}}\,dx$,
%\begin{equation*}
%\label{spanchiident}
%\mathbf{u}=\tilde{\mathbf{u}}-\,\int_Y\,\tilde{\mathbf{u}}\,dx,
%\end{equation*}
so $W_2^0=\tilde{W}_2$ and the lemma follows.
\end{proof}

Next, we identify all the eigenfunctions and eigenvalues of the following auxiliary eigenvalue problem.
%\begin{equation}
 %   \label{A00eigprob}
%P_2^0(-\Delta_{0})^{-1}\mathbf{u}=\beta \mathbf{u},
%\end{equation}
%with $\beta \in \sigma(A^{0}(0))$ and eigenfunction $\mathbf{u}\in L^2_{\#}(0,Y,\mathbb{C}^3)$, and notice that this is  equivalent to finding $\beta\in\mathbb{R}$ and $\mathbf{u}\in W_2$ for which:
%\begin{equation}
%(\mathbf{u}, \mathbf{v})_{L^2(Y, \mathbb{C}^3)} = \beta \langle \mathbf{u}, \mathbf{v} \rangle, \hbox{ for all $\mathbf{v}\in %J_{\#}(0,Y,\mathbb{C}^3)$}.
%\label{eqiveigen}
%\end{equation}
Find all eigen-pairs $(\mathbf{u},~\beta)$ in $\chi_0^{div}\times \mathbb{R}_+$ for which:
\begin{equation}
(\mathbf{u}, \mathbf{v})_{L^2(Y, \mathbb{C}^3)} = \beta \langle \mathbf{u}, \mathbf{v} \rangle, \hbox{ for all $\mathbf{v}\in \chi_0^{div}$}.
\label{eqiveigen}
\end{equation}
This eigenvalue problem is analyzed in \cite{Bouchitte2017}.  Following the results in \cite{Bouchitte2017}, we get the following theorem.
\begin{theorem}
\label{eigenperiod}
    The eigenvalues $\beta$ of \eqref{eqiveigen} are positive and form a sequence $\{\beta_n\}_{n=1}^\infty$ converging to 0. The eigenvectors of \eqref{eqiveigen} deliver a orthonormal system in $L_{\#}^2(0,Y,\mathbb{C}^3)$ that is dense in $\chi_0^{div}$ with respect to the $W^{1,2}(Y,\mathbb{C}^3)$-norm.
\end{theorem}

We now provide a precise characterization of the spectrum $\sigma(A^{0}(0))$ of the limit operator $A^0(0)$.  In preparation, we consider the countably dense in $L_{\#}^2(0,Y,\mathbb{C}^3)$, subset of $\chi^{div}_0$, orthonormal family of eigenfunctions $\{\mathbf{u}_n\}_{n=1}^\infty$ associated with the eigenvalues $\beta_n\searrow0$  of \eqref{eqiveigen}. Here, orthonormality is considered with respect to the $L^2(Y,\mathbb{C}^3)$-inner product. 

We have that $\sigma(A^{0}(0))$ consists of all $\nu^{-1}$ such that there exists a pair $\mathbf{u}$ and $\nu$, with $\mathbf{u}\in W_2^0$ and $\nu>0$, such that:
\begin{equation}
\langle \mathbf{u}, \mathbf{v} \rangle_D = \nu\,(\mathbf{u}, \mathbf{v})_{L^2(Y, \mathbb{C}^3)},\,\, \hbox{ for all $\mathbf{v}\in W_2^0$},
\label{eqiveigenchi-sig0}
\end{equation}
where $\langle \mathbf{u}, \mathbf{v} \rangle_D=\int_{D}\nabla\times\mathbf{u}
\cdot\nabla\times\overline{\mathbf{v}}\,dx$.  By (\ref{w2rep}), $\mathbf{u}=\tilde{\mathbf{u}}-\int_Y\,\tilde{\mathbf{u}}\,dx$, with $\tilde{\mathbf{u}}\in\chi^{div}_0.$  Hence, there exists a sequence $\left\{c_n\right\}_{n=1}^{\infty}\subset\mathbb{C}$ such that: 
\begin{equation}
    \tilde{\mathbf{u}}=\sum_{n=1}^\infty c_n\mathbf{u}_n,\hspace{0.3cm}\text{ and  }  \hspace{0.5cm} \mathbf{u}=\sum_{n=1}^\infty c_n\mathbf{u}_n+\mathbf{c}.
    \label{utilde-u-1}
\end{equation} 
where $\mathbf{c}=-\int_Y\,\tilde{\mathbf{u}}\,dx$.

First, suppose $\mathbf{u}\in\chi^{div}_0$ and $\mathbf{c}=-\int_Y\,\mathbf{u}\,dx=0$.  By (\ref{eqiveigenchi-sig0}), for $\mathbf{v}=\tilde{\mathbf{v}}-\int_Y\,\tilde{\mathbf{v}}\,dx$, with $\tilde{\mathbf{v}}\in\chi^{div}_0$, we obtain:$$\langle \mathbf{u}, \tilde{\mathbf{v}} \rangle_D = \nu\,(\mathbf{u}, \tilde{\mathbf{v}}-\int_Y\,\tilde{\mathbf{v}}\,dx)_{L^2(Y, \mathbb{C}^3)} = \nu\,(\mathbf{u}, \tilde{\mathbf{v}})_{L^2(Y, \mathbb{C}^3)},$$ since: $$(\mathbf{u}, \int_Y\,\tilde{\mathbf{v}}\,dx)_{L^2(Y, \mathbb{C}^3)} =\int_{Y}\mathbf{u}\cdot\overline{\int_Y\,\tilde{\mathbf{v}}\,dx}\,dy=0.$$
So $\mathbf{u}$ solves $\langle \mathbf{u}, \tilde{\mathbf{v}} \rangle_D = \nu\,(\mathbf{u}, \tilde{\mathbf{v}})_{L^2(Y, \mathbb{C}^3)}$, for all $\tilde{\mathbf{v}}\in\chi^{div}_0$,
%\begin{equation}
%\langle \mathbf{u}, \tilde{\mathbf{v}} \rangle_D = \nu\,(\mathbf{u}, \tilde{\mathbf{v}})_{L^2(Y, \mathbb{C}^3)}, \hbox{ for all $\tilde{\mathbf{v}}\in\chi^{div}_0$},
%\label{eqiveigenchi-sig02}
%\end{equation}
and is, therefore, an eigenfunction of \eqref{eqiveigenchi-sig0} belonging $\chi^{div}_0$ with $\int_Y \mathbf{u}\,dx=0$.  So all eigenvalues $\nu$ are eigenvalues $\left\{{\beta_n^{-1}}'\right\}_{n=1}^\infty\subset\left\{\beta_n^{-1}\right\}_{n=1}^\infty$ corresponding to mean zero eigenfunctions.
To summarize, a component of the spectrum $\sigma(A^{0}(0))$ of the limit operator $A^0(0)$ is given by  $\left\{{\beta_n^{-1}}'\right\}_{n=1}^\infty$. 

Next we identify the remaining component of $\sigma(A^{0}(0))$.
Now, suppose that $\mathbf{c}=-\displaystyle\int_Y\,\tilde{\mathbf{u}}\,dx\neq0$, and that $\mathbf{u}$ is an eigenfunction of \eqref{eqiveigenchi-sig0} with eigenvalue $\nu$. We normalize so that $|\mathbf{c}|=1$.  We have $\mathbf{u}=\tilde{\mathbf{u}}-\int_Y\,\tilde{\mathbf{u}}\,dx$ and for all $\mathbf{v}=\tilde{\mathbf{v}}-\int_Y\,\tilde{\mathbf{v}}\,dx$, we get: 
\begin{equation}
\langle \tilde{\mathbf{u}}, \tilde{\mathbf{v}} \rangle_D = \nu\,(\mathbf{u}, \tilde{\mathbf{v}})_{L^2(Y, \mathbb{C}^3)}, \hbox{ for all $\tilde{\mathbf{v}}\in\chi^{div}_0$}.
\label{eqiveigenchi-sig03}
\end{equation}

Using \eqref{utilde-u-1} in \eqref{eqiveigenchi-sig03}, we have:
\begin{equation}
\langle \sum_{n=1}^\infty c_n\mathbf{u}_n, \tilde{\mathbf{v}} \rangle_D = \nu\,(\sum_{n=1}^\infty c_n\mathbf{u}_n+\mathbf{c}, \tilde{\mathbf{v}})_{L^2(Y, \mathbb{C}^3)}, \hbox{ for all $\tilde{\mathbf{v}}\in\chi^{div}_0$}.
\label{eqiveigenchi-sig04}
\end{equation}
Now, pick $\tilde{\mathbf{v}}=\mathbf{u}_m$, $m\in\mathbb{N}^+$, in \eqref{eqiveigenchi-sig04}, to get:
\begin{align*}
    &c_m\beta_m^{-1}=\nu c_m+\nu\,(\mathbf{c}, \mathbf{u}_m)_{L^2(Y, \mathbb{C}^3)}\\
    &\implies c_m\beta_m^{-1}=\nu c_m+\nu\,\mathbf{c}\cdot \int_Y\overline{\mathbf{u}_m}\,dx\\
%    &\implies c_m\,(\beta_m^{-1}-\nu)=\nu\,\mathbf{c}\cdot \int_Y\overline{\mathbf{u}_m}\,dx\\
    &\implies c_m=\frac{\nu\,\mathbf{c}\cdot \int_Y\overline{\mathbf{u}_m}\,dx}{(\beta_m^{-1}-\nu)}.
\end{align*}

Then \eqref{utilde-u-1} becomes:
\begin{equation*}
    \tilde{\mathbf{u}}=\sum_{n=1}^\infty \frac{\nu\,\mathbf{c}\cdot \int_Y\overline{\mathbf{u}_n}\,dx}{(\beta_n^{-1}-\nu)}\mathbf{u}_n(\mathbf{x}),\hspace{0.3cm}\text{ and  }  \hspace{0.5cm} \mathbf{u}=\sum_{n=1}^\infty \frac{\nu\,\mathbf{c}\cdot \int_Y\overline{\mathbf{u}_n}\,dx}{(\beta_n^{-1}-\nu)}\mathbf{u}_n(\mathbf{x})+\mathbf{c}.
%    \label{utilde-u}
\end{equation*} 
Since we require $\int_Y\mathbf{u}\,dx=0$, we obtain:
\begin{equation}
\label{zeroaverage}
    \mathbf{c}=-\nu\,\sum_{n=1}^\infty \frac{\int_Y\mathbf{u}_n\,dx\otimes \int_Y\overline{\mathbf{u}_n}\,dx}{(\beta_n^{-1}-\nu)}\mathbf{c}.
\end{equation}
We introduce the \textit{effective magnetic permeability tensor}:
\begin{equation*}
%\label{matrixBF2}
    \boldsymbol{\mu}(\nu)=\left(I_{3\times3}+\nu\,\sum_{n=1}^\infty \frac{\int_Y\mathbf{u}_n\,dx\otimes \int_Y\overline{\mathbf{u}_n}\,dx}{(\beta_n^{-1}-\nu)}\right),
\end{equation*}
and \eqref{zeroaverage} gives the homogeneous system for the vector $\mathbf{c}$ in $\mathbb{C}^3$ given by:
\begin{equation}
\label{matrixBF}
    \boldsymbol{\mu}(\nu)\mathbf{c}=0.
\end{equation}
The effective permeability tensor agrees with the one given by the high contrast homogenization of Maxwell's equations in \cite{Bouchitte2017}.
We form the spectral function given by:
\begin{align}
    \label{specfunction}
    S(\nu)={\rm det}[\boldsymbol{\mu}(\nu)],
\end{align}
and, clearly, we have a nontrivial solution of \eqref{matrixBF} when $S(\nu)=0$.
The roots of the spectral function form a countable non-decreasing sequence of positive numbers $\{\nu_n\}_{n=1}^\infty$ tending to infinity. We set $\tilde{\beta}_n=\nu^{-1}_n$ and the complete characterization of $\sigma(A^{0}(0))$ given by:
\begin{theorem}
 $$\sigma(A^{0}(0))=\{\beta_n^{'}\}_{n=1}^\infty\cup\{\tilde{\beta}_n\}_{n=1}^\infty.$$
\end{theorem}

When the inclusion shape is invariant under the cubic group of rotations, the effective permeability tensor is a multiple of the identity, i.e., $\boldsymbol{\mu}(\nu)=I_{3\times 3}\lambda(\nu)$, where $\lambda(\nu)$ is a scalar function of $\nu$. Here, $\det\left\{\boldsymbol{\mu}(\nu)\right\}=\lambda^3(\nu)$, so $\nu_j$ are the roots of the equation $\lambda(\nu)=0$. For any constant vector $\mathbf{v}$ in $\mathbb{R}^3$ we have:
\begin{equation}
    \label{merimorphic}
    \lambda({\nu})= \frac{\boldsymbol{\mu}(\nu)\mathbf{v}\cdot\mathbf{v}}{|\mathbf{v}|^2}=1-\nu\sum_{n\in\mathbb{N}}\frac{a^2_n}{\nu-\beta^*_n},
\end{equation}
where $a^2_n={|\int_D\mathbf{u}_n\;dx\cdot\mathbf{v}|^2}/{|\mathbf{v}|^2}>0$ and $\beta_n^\ast$ are only associated with nonzero mean eigenfunctions. For $\beta^\ast_{n-1}<\nu<\beta^\ast_{n}$, calculation shows $-\infty<\lambda(\nu)<\infty$, with $\lambda'(\nu)>0$. From this, we conclude $\beta^\ast_{n}<\nu_j<\beta^\ast_{n+1}$ and we have the interlacing $\nu_{n-1}<\beta^\ast_n<\nu_{n}$.

\section{Radius of Convergence and Separation of Spectra}
\label{radius}

Fix an inclusion geometry specified by the domain $D$. Suppose first $\alpha\in Y^\star$ and $\alpha\not =0$. %Recall from Theorem \ref{equiv1}  that the spectrum of $A^\alpha(0)$ is $\sigma(-\Delta_D^{-1})$. 
Take $\Gamma_j$ to be a closed contour in $\mathbb{C}$ containing an eigenvalue  $\beta^\alpha_j(0)\in\sigma(A^\alpha(0))$, but no other element of $\sigma(A^\alpha(0))$, i.e, for $\alpha\neq\mathbf{0}\in Y^*$ fixed, $\beta^\alpha_j(0)$ is separated from other components of the spectrum, see Figure~\ref{spectrum}. Define $d$ to be the distance between $\Gamma_j$ and $ \sigma(A^\alpha(0))$, i.e.:
\begin{equation}
d={\rm{dist}}(\Gamma_j,\sigma(A^\alpha(0))=\inf_{\zeta\in\Gamma_j}\{{\rm{dist}}(\zeta,\sigma(A^\alpha(0))\}.
\label{dist}
\end{equation}
The component of the spectrum of $A^\alpha(0)$ inside $\Gamma_j$  is precisely $\beta^\alpha_j(0)$, and we denote this by $\Sigma'(0)$. The part of the spectrum of $A^\alpha(0)$ in the  domain exterior to $\Gamma_j$ is denoted by $\Sigma''(0)$, and $\Sigma''(0)=\sigma(A^\alpha(0))\setminus \beta^\alpha_j(0)$. The invariant subspace of $A^\alpha(0)$ associated with $\Sigma'(0)$ is denoted by $\mathcal{M}'(0)$ with $\mathcal{M}'(0)=P(0)L^2_{\#}(\alpha,Y,\mathbb{C}^3)$.

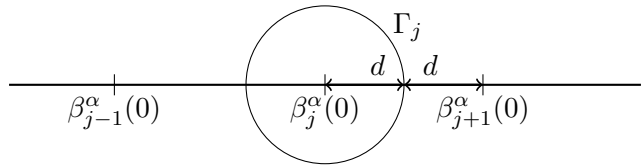
\begin{figure}[ht] 
\centering
\begin{tikzpicture}[xscale=0.70,yscale=0.70]
%\draw [-,thick] (0,2.25) -- (0,-2.25);
\draw [-,thick] (-6,0) -- (6,0);
\draw [<->,thick] (0,0) -- (1.5,0);
\node [above] at (1,0) {$d$};
\draw [<->,thick] (1.5,0) -- (3,0);
\node [above] at (2,0) {$d$};
%\node [right] at (1,0.0) {$\nu$};
%\node [above] at (1.6,-0.60) {$y-x$};
\draw (-4,0.2) -- (-4.0, -0.2);
\node [below] at (-4,0) {$\beta^\alpha_{j-1}(0)$};
\draw (0,0.2) -- (0, -0.2);
\node [below] at (0,0) {$\beta^\alpha_j(0)$};
\draw (3,0.2) -- (3, -0.2);
\node [below] at (3,0) {$\beta^\alpha_{j+1}(0)$};
\draw (0,0) circle [radius=1.5];
\node [right] at (1.1,1.1) {$\Gamma_j$};
%\node [right] at (0.0,0.65) {$\delta=\overline{u}s(t)$};
%\node [left] at (-0.4,0.65) {$\delta=0$};
%\node [above] at (1.2,-1.0) {$\epsilon$};
%\node [right] at (-0.95,1.7) {$E^-_{\nu}$};
\end{tikzpicture} 
\caption{Schematic of $\Gamma_j$, $d$, $\Sigma'(0)$, and $\Sigma''(0)$. }
 \label{spectrum}
\end{figure}

Suppose the lowest $\alpha$-quasiperiodic resonance eigenvalue for the domain $D$ lies inside $-1/2<\mu^-(\alpha)<0$. It is noted that, in the sequel, a large and generic class of domains are identified for which $-1/2<\mu^-(\alpha)$.  The corresponding upper bound on the set $z\in \mathcal{S}$, for which $A^\alpha(z)$ is not invertible, is given by: 
\begin{eqnarray}
z^\ast(\alpha)=\frac{\mu^-(\alpha)+1/2}{\mu^-(\alpha)-1/2}<0,
\label{upperonS}
\end{eqnarray}
see \eqref{bdsonS}.
Now set:
\begin{equation}
r^*=\frac{|\alpha|^2d|z^\ast(\alpha)|}{\frac{1}{1/2-\mu^-(\alpha)}+|\alpha|^2d}.
\label{radiusalphanotzero}
\end{equation}

\begin{theorem}{\rm Separation of spectra and radius of convergence for $\alpha\in Y^\star$, $\alpha\not=0$.}\\
\label{separationandraduus-alphanotzero}
The following properties  hold for inclusions with domains $D$ that satisfy \eqref{upperonS}:
\begin{enumerate}
\item If $|z|<r^*$, then $\Gamma_j$ lies in the resolvent of both $A^\alpha(0)$ and $A^\alpha(z)$ and, thus, separates the spectrum of $A^\alpha(z)$ into two parts given by the component of spectrum of $A^\alpha(z)$ inside $\Gamma_j$, denoted by $\Sigma'(z)$, and the component exterior to $\Gamma_j$,  denoted by $\Sigma''(z)$. The invariant subspace of $A^\alpha(z)$ associated with $\Sigma'(z)$ is denoted by $\mathcal{M}'(z)$, with $\mathcal{M}'(z)=P(z)L^2_{\#}(\alpha,Y,\mathbb{C}^3)$.

\item The projection $P(z)$ is holomorphic for $|z|<r^*$ and $P(z)$ is given by:
\begin{equation*}
P(z)=\frac{-1}{2\pi i}\oint_{\Gamma_j} R(\zeta,z)\,d\zeta.
%\label{formula-alphanotzero}
\end{equation*}
\item The spaces $\mathcal{M}'(z)$ and $\mathcal{M}'(0)$ are isomorphic for $|z|<r^*$.
\item The power series \eqref{foureleven} converges uniformly for $z\in\mathbb{C}$ inside any disk centered at the origin contained within $|z|<r^*$.
\end{enumerate}
\end{theorem}

Suppose now $\alpha=\mathbf{0}$. 
%Recall from Theorem \ref{equiv2} that the limit spectrum for $A^0(0)$ is
%$\sigma(A^0(0)) =  \{{\delta'_j}^{-1}\}_{j\in\mathbb{N}}\cup\{{\nu_j}^{-1}\}_{j\in\mathbb{N}}$. 
For this case, take $\Gamma_j$ to be the closed contour in $\mathbb{C}$ containing an eigenvalue $\beta_j^0(0)\in\sigma(A^0(0))$, but no other element of
$\sigma(A^0(0))$, i.e., $\Gamma_j$ separates $\beta^\alpha_j(0)$ from other components of the spectrum, and define: 
\begin{equation*}
%\label{dforperiodic}
d=\inf_{\zeta\in\Gamma_j}\{\rm{dist}(\zeta,\sigma(A^0(0)))\}.
\end{equation*}
Suppose that the lowest $\alpha$-quasiperiodic resonance eigenvalue for the domain $D$ lies inside $-1/2<\mu^-(0)<0$ and the corresponding upper bound on $\mathcal{S}$ is given by: 
\begin{equation}
z^\ast(0)=\frac{\mu^-(0)+1/2}{\mu^-(0)-1/2}<0.
\label{upperonSzero}
\end{equation}
Set:
\begin{equation}
r^*=\frac{4\pi^2d|z^\ast(0)|}{\frac{1}{1/2-\mu^-(0)}+4\pi^2d}.
\label{radiusalphazero}
\end{equation}
\begin{theorem}{\rm Separation of spectra and radius of convergence for $\alpha=\mathbf{0}$.}
\label{separationandraduus-alphazero}
\\
The following properties  hold for inclusions with domains $D$ that satisfy \eqref{upperonSzero}: 
\begin{enumerate}
\item If $|z|<r^*$, then $\Gamma_j$ lies in the resolvent of both $A^0(0)$ and $A^0(z)$ and, thus, separates the spectrum of $A^0(z)$ into two parts given by the component of spectrum of $A^0(z)$ inside $\Gamma_j$, denoted by $\Sigma'(z)$, and the component exterior to $\Gamma_j$, denoted by $\Sigma''(z)$. The invariant subspace of $A^0(z)$ associated with $\Sigma'(z)$ is denoted by $\mathcal{M}'(z)$, with $\mathcal{M}'(z)=P(z)L^2_{\#}(\alpha,Y,\mathbb{C}^3)$.
\item The projection $P(z)$ is holomorphic for $|z|<r^*$ and $P(z)$ is given by:
\begin{equation*}
P(z)=\frac{-1}{2\pi i}\oint_{\Gamma_j} R(\zeta,z)\,d\zeta.
%\label{formula-alphazero}
\end{equation*}
\item The spaces $\mathcal{M}'(z)$ and $\mathcal{M}'(0)$ are isomorphic for $|z|<r^*$.
\item The power series \eqref{foureleven} converges uniformly for $z\in\mathbb{C}$ inside any disk centered at the origin contained within $|z|<r^*$.
\end{enumerate}
\end{theorem}

Next, we provide an explicit representation of the integral operators appearing in the series expansion for the eigenvalue group.
\begin{theorem}{\rm Representation of integral operators in the series expansion for eigenvalues}\\
\label{reptheorem1}
Let $P_3^\alpha$ be the projection onto the orthogonal complement of $W_1^\alpha\oplus W_2^\alpha$, and let $\tilde{I}$ denote the identity on $L^2(\partial D)^3$,  then the explicit representation for
for the operators $A_n^\alpha$ in the expansion  \eqref{foureleven}, \eqref{fourtwelve} is given by:
\begin{align*}
A_1^\alpha&=[S^\alpha(M^\alpha+\frac{1}{2}\tilde{I})^{-1}(S^\alpha)^{-1}P_3^\alpha+P_1^\alpha](-\Delta_\alpha)^{-1}\hbox{ \rm and}\notag\\
A_n^\alpha&=S^\alpha(M^\alpha+\frac{1}{2}\tilde{I})^{-1}(S^\alpha)^{-1}[S^\alpha(M^\alpha-\frac{1}{2}\tilde{I})(M^\alpha+\frac{1}{2}\tilde{I})^{-1}(S^\alpha)^{-1}]^{n-1}P_3^\alpha(-\Delta_\alpha)^{-1}.
%\label{terms}
\end{align*}
\end{theorem}

We have a corollary to Theorems \ref{separationandraduus-alphanotzero} and \ref{separationandraduus-alphazero} regarding the error incurred when only finitely many terms of the series (\ref{foureleven}) are calculated.
\begin{theorem}{\rm Error estimates for the eigenvalue expansion}.
\begin{enumerate}
\label{errorestm}
\item Let $\alpha \neq 0$, and suppose $D$, $z^*(\alpha)$, and $r^*$ are as in Theorem \ref{separationandraduus-alphanotzero}. Then, the following error estimate for the series \eqref{foureleven} holds for $|z|<r^*$:
\begin{equation*}
\left |\hat{\beta}^{\alpha}(z) - \sum \limits_{n = 0}^{p} z^n \beta^{\alpha}_n \right | \leq \frac{d|z|^{p+1}}{(r^*)^p(r^* - |z|)}.
\end{equation*}
\item Let $\alpha = 0$, and suppose $D$, $z^*(0)$, and $r^*$ are as in Theorem \ref{separationandraduus-alphazero}. Then, the following error estimate for the series \eqref{foureleven} holds for $|z|<r^*$:
\begin{equation*}
\left |\hat{\beta}^{0}(z) - \sum \limits_{n = 0}^{p} z^n \beta^{0}_n \right | \leq \frac{d|z|^{p+1}}{(r^*)^p(r^* - |z|)}.
\end{equation*}
\end{enumerate}
\end{theorem}

We summarize results in the following theorem.
\begin{theorem}
\label{maintheorem}
The Bloch eigenvalue problem \eqref{Helm} is defined for the coupling constant $k$ extended into
the complex plane and the operator $-\nabla\times(k\chi_{H}+\chi_D)\nabla\times$  with domain $J_{\#}(\alpha,Y,\mathbb{C}^3)$ is holomorphic for $k\in\mathbb{C}\setminus Z$. The associated Bloch spectra is given by the eigenvalues $\lambda_j(k,\alpha)=(\beta_j^\alpha(1/k))^{-1}$, for $j\in\mathbb{N}$. For $\alpha\in Y^\star$ fixed, the eigenvalues are of finite multiplicity. Moreover for each $j$ and $\alpha\in Y^\star$, the eigenvalue group is analytic within any neighborhood of infinity
contained within the disk $|k|>(r^*)^{-1}$ where $r^*$ is given by \eqref{radiusalphanotzero} for $\alpha\not=0$ and by \eqref{radiusalphazero} for $\alpha=\mathbf{0}$.
\end{theorem}

The proofs of Theorems~\ref{separationandraduus-alphanotzero}, \ref{separationandraduus-alphazero}  and \ref{errorestm} are given in Section~\ref{derivation}. The proof of Theorem~\ref{reptheorem1} is given in Section~\ref{leading-order}.

\section{Radius of Convergence and Separation of Spectra for Periodic Scatterers of General Shape}
\label{radiusgeneralshape}

In this section, we identify an explicit condition on the inclusion geometry that guarantees a lower bound $\mu^-$ on the structural spectrum.  

 Let $D$ be a simply connected set, compactly contained in $Y$,  with $C^{1,\gamma}$ boundary, $\gamma>0$. Recall that, by Theorem~\ref{mcompact}, we have that the eigenvalues of the magnetic dipole operator are precisely those of the Neumann-Poincar\'e operator, that is: $$\sigma(M^\alpha;\ V^{^{-\frac{1}{2}}}_{t}(\partial D)^{3})=\sigma((K^{-\alpha})^{*};\ H^{-\frac{1}{2}}_{0}(\partial D)).$$  Moreover, a criteria for an $\alpha$-independent lower bound for $\sigma\left((K^{-\alpha})^{*};\ H^{-\frac{1}{2}}_{0}(\partial D)\right)$ was already established in \cite{RobertRobert1}, in a theorem which we restate below.
%Adding $\frac{1}{2}\displaystyle\int \limits_{Y} (\nabla\times \mathbf{w})\cdot (\nabla\times \overline{\mathbf{v}})\ dy$ to both sides yields:
%\begin{equation}
%\int \limits_{Y\setminus D} (\nabla\times \mathbf{w})\cdot (\nabla\times %\overline{\mathbf{v}})\ dy = (\mu + \frac{1}{2}) \int \limits_{Y} (\nabla\times %\mathbf{w})\cdot (\nabla\times \overline{\mathbf{v}})\ dy.
%\end{equation}

%We will show that there exists a $\rho > 0$ such that $\mu_i + \frac{1}{2} > \rho$ %independent of $i \in \mathbb{N}$ and $\alpha \in Y^*$.  If such a $\rho$ exists, then %cearly $\mu_i > \rho -  %lower binu^- = \rho-\frac{1}{2}$ satisfying the desired ineqty. l 

\begin{theorem}
\label{lowerboundrho}  Let $\mu^-$ be the infimum of the structural spectrum.  Suppose there is a constant $\theta >0$ such that, for all $u \in H^1_{\alpha}(Y)$ that are harmonic in $D$ and $Y\setminus D$, we have:
\begin{equation}
\label{thetaineq}
\| \nabla u\|_{L^2(Y\setminus D)}^2 \geq \theta \| \nabla u \|_{L^2(D)}^2.
\end{equation}

Let $\rho = \min \{ \frac{1}{2}, \frac{\theta}{2} \}$.  Then $\mu^- + \frac{1}{2} > \rho$.
\end{theorem}
%\begin{proof}
%We proceed by contradiction: suppose that $\mu^-(\alpha) + \frac{1}{2} < \frac{1}{2}$ and $\mu^-(\alpha) + \frac{1}{2} < \frac{\theta}{2}$.  Let $u^-$ be the normalized eigenvector of $T$ associated with $\mu^-(\alpha)$.  Then we have:
%\begin{equation}
%\label{sevenfour}
%\int \limits_{Y\setminus D} | \nabla\times\mathbf{u}^-|^2 < \frac{1}{2}
%\end{equation}
%and:
%\begin{equation}
%\frac{\theta}{2} > \int \limits_{Y \setminus D} | \nabla\times\mathbf{u}^-|^2 \geq \theta \int \limits_{D} | \nabla\times\mathbf{u}^-|^2.
%\end{equation}
%Thus, we have:
%\begin{equation}
%\label{sevensix}
%\int \limits_{D} | \nabla\times\mathbf{u}^-|^2 < \frac{1}{2}.
%\end{equation}
%Inequalities \eqref{sevenfour} and \eqref{sevensix} yield:
%$$\| \nabla\times\mathbf{u}^- \|_{L^2(Y)}^2 <1.$$
%But $u^-$ was normalized, so that:
%$$\| \nabla\times\mathbf{u}^- \|_{L^2(Y)}^2 = 1 \text{,}$$
%completing the proof.
%\end{proof}

Clearly, the parameter $\theta$ is a geometric descriptor for $D$. The class of inclusions for which Theorem~\eqref{lowerboundrho} holds, for a fixed positive value of $\theta$, is denoted by $P_\theta$, and we have the following corollary.
\begin{corollary}
\label{theta}
For every inclusion domain $D$ belonging to $P_\theta$, Theorems~\ref{separationandraduus-alphazero} through \ref{maintheorem} hold with $z^*(\alpha)$ replaced with $z^+$ given by:
\begin{equation*}
z^+=\frac{\mu^-+1/2}{\mu^--1/2}<0,
%\label{thetaclass}
\end{equation*}
where $\mu^-=\min\{\frac{1}{2},\frac{\theta}{2}\}-\frac{1}{2}$.
\end{corollary}

In \cite{RobertRobert1}, the authors also introduce a wide class of inclusion shapes with $\theta>0$ that satisfy \eqref{thetaineq}.  
Consider a buffered inclusion geometry, which consists of an inclusion domain $D$ surrounded by a buffer layer $R$, see Figure~\ref{plane2}.  Denote the Dirichlet-to-Neumann map on the boundary of the inclusion by $DN: H^{1/2}(\partial D) \rightarrow H^{-1/2}(\partial D)$, denote its norm by $\| DN \|$, and denote the Poincar\'e constant for the buffer layer by $C_{R}$; we have the following theorem, also from \cite{RobertRobert1}.

\begin{theorem}
The buffered inclusion geometry satisfies \eqref{thetaineq} with: 
\begin{equation*}
\theta^{-1}\geq\sqrt{1+C_{R}^2}\,\Vert DN\Vert\,
%\label{setcriteria}
\end{equation*}
provided this maximum is finite. 
\end{theorem}

We now take $D_i=B_a(x_i)$, a sphere with center $x_i$ and radius $a$, and observe that $D_i'=B_b(x_i)\supset D_i$ if $a<b$.  Following Appendix~A.3 of \cite{Bruno}, we see that $\theta^{-1}$ will satisfy: $$\theta^{-1}=\max_{l\geq1}C_l(a,b),$$ where: $$C_l(a,b)=\frac{lb^{2l+1}+(l+1)a^{2l+1}}{(l+1)(b^{2l+1}-a^{2l+1})}.$$

Adding and subtracting $b^{2l+1}$ in the numerator yields:
\begin{align*}
    C_l(a,b)&=\frac{b^{2l+1}+a^{2l+1}}{b^{2l+1}-a^{2l+1}}-\frac{b^{2l+1}}{(l+1)(b^{2l+1}-a^{2l+1})}\\
    &\leq\frac{b^{2l+1}+a^{2l+1}}{b^{2l+1}-a^{2l+1}}=:C_l^*(a,b).
\end{align*}
Note that $C_l^*(a,b)$ is decreasing in $l$: $$\frac{d}{dl}C_l^*(a,b)=\frac{2(ab)^{2l+1}(\ln(a)-\ln(b))}{(b^{2l+1}-a^{2l+1})^2}<0,$$ for all $l\geq1$.  So: $$\theta^{-1}\leq\max_{l\geq1}C_l^*(a,b)=\frac{b^3+a^3}{b^3-a^3}.$$
Thus:$$\Vert\nabla u \Vert_{L^2(Y\setminus D)}\geq\frac{b^3-a^3}{b^3+a^3}\Vert\nabla u\Vert_{L^2(D)}.$$
Observe that this bound is not sharp.

%\textcolor{red}{Need to update Figure~\ref{plane2}} 
\begin{figure}[h] 
\centering
\begin{tikzpicture}[xscale=0.6,yscale=0.6]
\draw [thick] (-2,-2) rectangle (3,3);
\draw [fill=orange,thick] (0.5,0.5) circle [radius=1.2];
\draw (0.5,0.5) circle [radius=1.6];
\end{tikzpicture} 
\caption{\bf Buffered inclusion.}
 \label{plane2}
\end{figure}
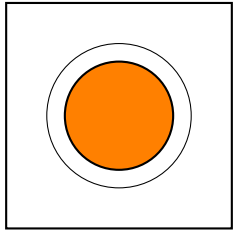
\section{Layer Potential Representation of Operators in Power Series}
\label{leading-order}

In this section, we obtain explicit formulas for the operators $A^\alpha_n$ appearing in the power series  \eqref{fourtwelve}. It is shown that $A^{\alpha}_n$, $n \neq 0$, can be expressed in terms of  integral operators associated with layer potentials, and we establish Theorem \ref{reptheorem1}.  

Recall that $A^{\alpha}(z) - A^{\alpha}(0)$ is given by:
 \begin{equation*}
 \big(z\,P_1^\alpha + \sum_{-\frac{1}{2} < \mu_i(\alpha) < \frac{1}{2}}z\left[(1/2 + \mu_i(\alpha)) + z(1/2-\mu_i(\alpha))\right] ^{-1}P_{\mu_i}^\alpha\big)(-\Delta_{\alpha}^{-1}).
 \end{equation*}
  Factoring $(1/2+\mu_i(\alpha))^{-1}$ and expanding in power series the term: 
 \begin{equation*}
 [(1/2 + \mu_i(\alpha)) + z(1/2-\mu_i(\alpha))]^{-1} = (1/2 + \mu_i(\alpha))^{-1}\sum \limits_{n=0}^{\infty}z^n\left (\frac{\mu_i(\alpha)-1/2}{\mu_i(\alpha) + 1/2} \right )^n,
 \end{equation*}
 we obtain:
 \begin{equation*}
A^{\alpha}(z) - A^{\alpha}(0) = (z P_1^\alpha + \sum \limits_{n=1}^{\infty}z^{n}\sum_{-\frac{1}{2} < \mu_i(\alpha) < \frac{1}{2}}(\mu_i(\alpha) +1/2)^{-1} \left (\frac{\mu_i(\alpha)-1/2}{\mu_i(\alpha) + 1/2} \right)^{n-1}  P_{\mu_i}^\alpha P_3^\alpha)(-\Delta_{\alpha}^{-1}).
\end{equation*}
 
 It follows that:
 \begin{align}
 A_1^{\alpha}  &=  (P_1^\alpha + \sum \limits_{-\frac{1}{2} < \mu_i(\alpha) < \frac{1}{2}} (1/2 + \mu_i(\alpha))^{-1} P_{\mu_i}^\alpha P_3^\alpha)(-\Delta_{\alpha}^{-1}) \label{AnProjections}\\
 A_n^{\alpha} & =  \Big(\sum \limits_{-\frac{1}{2} < \mu_i(\alpha) < \frac{1}{2}}(\mu_i(\alpha) +1/2)^{-1} \left (\frac{\mu_i(\alpha)-1/2}{\mu_i(\alpha) + 1/2} \right)^{n-1}  P_{\mu_i}^\alpha P_3^\alpha\Big)(-\Delta_{\alpha}^{-1}).
 \label{AnProjections2}
 \end{align}
 
 We also that we have the resolution of the identity given by:
 \begin{equation*}
 I=I_{J_\#(\alpha,Y,\mathbb{C}^3)} = P_1^\alpha+P_2^\alpha+P_3^\alpha,
 %\hspace{0.6cm}\hbox{ with } P_3=\sum \limits_{-\frac{1}{2} < \mu_i < \frac{1}{2}} P_{\mu_i},
 %\label{residentity}
 \end{equation*}
 with $P_3^\alpha=\sum \limits_{-\frac{1}{2} < \mu_i(\alpha) < \frac{1}{2}} P_{\mu_i}^\alpha$, and the spectral representation:
 \begin{align*}
  \langle T^\alpha\mathbf{u},\mathbf{v} \rangle & =  \langle (S^\alpha M^\alpha(S^\alpha)^{-1}) P_3^\alpha\mathbf{u} + \frac{1}{2} P_1^\alpha\mathbf{u} - \frac{1}{2} P_2^\alpha\mathbf{u}, \mathbf{v} \rangle \\
 & =  \langle \sum \limits_{-\frac{1}{2} < \mu_i(\alpha) < \frac{1}{2}} \mu_i(\alpha) P_{\mu_i}^\alpha\mathbf{u} + \frac{1}{2} P_1^\alpha\mathbf{u} - \frac{1}{2} P_2^\alpha\mathbf{u}, \mathbf{v} \rangle.
 \end{align*}
 
 Adding $\frac{1}{2} I$ to both sides of the above equation, we obtain:
  \begin{align}
   \label{identproj}
 \langle (T^\alpha+ \frac{1}{2}I)\mathbf{u},\textbf{v} \rangle & =  \langle (\sum \limits_{-\frac{1}{2} < \mu_i(\alpha) < \frac{1}{2}} (\mu_i(\alpha) + \frac{1}{2})P_{\mu_i}^\alpha +P_1^\alpha)\mathbf{u},\mathbf{v}\rangle \notag\\
  & =   \langle ((S^\alpha M^\alpha (S^\alpha)^{-1} +\frac{1}{2}P_3^\alpha)P_3^\alpha +P_1^\alpha)\mathbf{u},\mathbf{v}\rangle \\
  & =  \langle ((S^\alpha(M^\alpha+ \frac{1}{2} \tilde{I})(S^\alpha)^{-1})P_3^\alpha +P_1^\alpha)\mathbf{u},\mathbf{v}\rangle,\notag
  \end{align}
  where $\tilde{I}$ is the identity on  $H^{-1/2}(\partial D)^3$.  Now, from \eqref{identproj}, we see that:
 \begin{equation}
 \label{A1spectrumtolayers}
 \sum \limits_{-\frac{1}{2} < \mu_i(\alpha) < \frac{1}{2}} (\frac{1}{2} + \mu_i(\alpha))^{-1} P_{\mu_i}^\alpha P_3^\alpha = (S^\alpha(M^\alpha + \frac{1}{2}\tilde{I})^{-1}(S^\alpha)^{-1})P_3^\alpha.
 \end{equation}
 
 Combining \eqref{AnProjections} and \eqref{A1spectrumtolayers}, we obtain:
  \begin{equation*}
 %\label{A1Layers}
 A_1^\alpha=[S^\alpha(M^\alpha+\frac{1}{2}\tilde{I})^{-1}(S^\alpha)^{-1}P_3^\alpha+P_1^\alpha](-\Delta_\alpha)^{-1}.
 \end{equation*}
 
 We now turn to the higher-order terms.  By the mutual orthogonality of the projections $P_{\mu_i}^\alpha$, for $n>1$, we have that:
 \begin{align}
\label{separateresonances}
 &\sum \limits_{-\frac{1}{2} < \mu_i(\alpha) < \frac{1}{2}} (\mu_i(\alpha) +1/2)^{-1} \left (\frac{\mu_i(\alpha)-1/2}{\mu_i(\alpha) + 1/2} \right)^{n-1}  P_{\mu_i}^\alpha \\
 %&= \left( \sum \limits_{-\frac{1}{2} < \mu_i < \frac{1}{2}} (1/2 + \mu_i)^{-1}P_{\mu_i} \right ) \left( \sum \limits_{-\frac{1}{2} < \mu_i < \frac{1}{2}} \left( \frac{\mu_i-1/2}{\mu_i +1/2} \right) P_{\mu_i} \right)^{n-1}\\
 &=  \Big( \sum \limits_{-\frac{1}{2} < \mu_i(\alpha) < \frac{1}{2}} (1/2 + \mu_i(\alpha))^{-1}P_{\mu_i}^\alpha \Big) \Big( \sum \limits_{-\frac{1}{2} < \mu_i(\alpha) < \frac{1}{2}} (\mu_i(\alpha)-1/2) P_{\mu_i}^\alpha \Big)^{n-1} \Big( \sum \limits_{-\frac{1}{2} < \mu_i(\alpha) < \frac{1}{2}} (\mu_i(\alpha) +1/2) P_{\mu_i}^\alpha \Big)^{1-n}.\notag
 \end{align}
 
 As above, we have that:
  \begin{align}
 \label{resonancetolayersAn}
 \sum \limits_{-\frac{1}{2} < \mu_i(\alpha) < \frac{1}{2}} (1/2 + \mu_i(\alpha))^{-1}P_{\mu_i}^\alpha P_3^\alpha\,\, & = \,\, S^\alpha(M^\alpha+\frac{1}{2}\tilde{I})^{-1}(S^\alpha)^{-1}P_3^\alpha,\notag\\
  \sum \limits_{-\frac{1}{2} < \mu_i(\alpha) < \frac{1}{2}} (1/2 + \mu_i(\alpha))P_{\mu_i}^\alpha P_3^\alpha \,\, & = \,\,  S^\alpha(M^\alpha+\frac{1}{2}\tilde{I})(S^\alpha)^{-1}P_3^\alpha,\\
   \sum \limits_{-\frac{1}{2} < \mu_i(\alpha) < \frac{1}{2}} (\mu_i(\alpha) - 1/2)P_{\mu_i}^\alpha P_3^\alpha \,\, & = \,\,  S^\alpha(M^\alpha- \frac{1}{2}\tilde{I})(S^\alpha)^{-1}P_3^\alpha.\notag
   \end{align}
%  \begin{equation}
% \begin{array}{lcl}
% \label{resonancetolayersAn}
% \sum \limits_{-\frac{1}{2} < \mu_i < \frac{1}{2}} (1/2 + \mu_i)^{-1}P_{\mu_i} P_3 & = & S(M+\frac{1}{2}\tilde{I})^{-1}S^{-1}P_3,\\
% \\
%  \sum \limits_{-\frac{1}{2} < \mu_i < \frac{1}{2}} (1/2 + \mu_i)P_{\mu_i} P_3 & = & S(M+\frac{1}{2}\tilde{I})S^{-1}P_3,\\
%  \\
%   \sum \limits_{-\frac{1}{2} < \mu_i < \frac{1}{2}} (\mu_i - 1/2)P_{\mu_i} P_3 & = & S(M- \frac{1}{2}\tilde{I})S^{-1}P_3.
%   \end{array}
%   \end{equation}
 
 Combining \eqref{resonancetolayersAn}, \eqref{separateresonances}, and \eqref{AnProjections2}, we obtain the layer-potential representation for $A_n^{\alpha}$, concluding the proof of Theorem~\ref{reptheorem1}:
  \begin{equation*}
 A_n^\alpha=S^\alpha(M^\alpha+\frac{1}{2}\tilde{I})^{-1}(S^\alpha)^{-1}[S^\alpha(M^\alpha-\frac{1}{2}\tilde{I})(M^\alpha+\frac{1}{2}\tilde{I})^{-1}(S^\alpha)^{-1}]^{n-1}P_3^\alpha(-\Delta_\alpha)^{-1}.
\end{equation*}
\\
\section{Derivation of the Convergence Radius and Separation of Spectra}
\label{derivation}

In this section, we present the proof of Theorem~\ref{separationandraduus-alphanotzero} and the proof of Theorem~\ref{separationandraduus-alphazero}.  To begin, we suppose $\alpha\not=0$ and recall that the Neumann series \eqref{foursix}, and consequently \eqref{Project1} and \eqref{foureleven}, converge provided that:
\begin{equation}
\label{tenone}
\| (A^{\alpha}(z) - A^{\alpha}(0))R(\zeta,0) \|_{\mathcal{L}[L^2_{\#} (\alpha,Y,\mathbb{C}^3);L^2_{\#} (\alpha,Y,\mathbb{C}^3)]} <1.
\end{equation}
With this in mind, we will compute an explicit upper bound $B(\alpha,z)$ and identify a neighborhood of the origin on the complex plane for which:
\begin{equation*}
%\label{tenoneb}
\| (A^{\alpha}(z) - A^{\alpha}(0))R(\zeta,0) \|_{\mathcal{L}[L^2_{\#} (\alpha,Y,\mathbb{C}^3);L^2_{\#} (\alpha,Y,\mathbb{C}^3)]} <B(\alpha,z)<1,
\end{equation*}
holds for $\zeta\in\Gamma_j$.  The inequality $B(\alpha,z)<1$ will be used first to derive a lower bound on the radius of convergence of the power series expansion of the eigenvalue group about $z=0$. Then, it will be used to provide a lower bound on the neighborhood of $z=0$ where properties $1$ through $3$ of Theorem~\ref{separationandraduus-alphanotzero} hold.

We have the basic estimate given by:
\begin{align}
\label{tenonedouble}
&\| (A^{\alpha}(z) - A^{\alpha}(0))R(\zeta,0) \|_{\mathcal{L}[L^2_{\#} (\alpha,Y,\mathbb{C}^3);L^2_{\#} (\alpha,Y,\mathbb{C}^3)]} \\
&\quad\leq\| (A^{\alpha}(z) - A^{\alpha}(0))\|_{\mathcal{L}[L^2_{\#} (\alpha,Y,\mathbb{C}^3);L^2_{\#} (\alpha,Y,\mathbb{C}^3)]}\|R(\zeta,0) \|_{\mathcal{L}[L^2_{\#} (\alpha,Y,\mathbb{C}^3);L^2_{\#} (\alpha,Y,\mathbb{C}^3)]}.\nonumber
\end{align}
Here $\zeta\in\Gamma_j$, as defined in Theorem~\ref{separationandraduus-alphanotzero}, and elementary arguments deliver the estimate:
\begin{eqnarray}
\label{tenonedoubleRz}
\|R(\zeta,0) \|_{\mathcal{L}[L^2_{\#} (\alpha,Y,\mathbb{C}^3);L^2_{\#} (\alpha,Y,\mathbb{C}^3)]}\leq d^{-1},
\end{eqnarray}
where $d$ is given by \eqref{dist}.  Next, we estimate $\| (A^{\alpha}(z) - A^{\alpha}(0))\|_{\mathcal{L}[L^2_{\#} (\alpha,Y,\mathbb{C}^3);L^2_{\#} (\alpha,Y,\mathbb{C}^3)]} $.  

Denote the energy seminorm of $\mathbf{u}$ by:
\begin{equation*}
\|\mathbf{u} \|= \| \nabla\times\mathbf{ u} \|_{L^2(Y,\mathbb{C}^3)}.
\end{equation*}
To proceed, we introduce the following Poincar\'e estimate: 
\begin{lemma}	Poincar\'e estimate for functions $\mathbf{u}$ belonging to $J_{\#}(\alpha,Y,\mathbb{C}^3)$, for $\alpha\neq\mathbf{0}$:
\begin{equation}
\label{alpha-poincare}
\|\mathbf{u} \|_{L^2(Y,\mathbb{C}^3)} \leq |\alpha|^{-1}\|\mathbf{u}\|.
\end{equation}
\label{poincarealpha}
\end{lemma}
\begin{proof}
First, we obtain that: 
\begin{align}
\label{poincest1}
(-\Delta_{\alpha}^{-1}\mathbf{u},\mathbf{u})_{L^2(Y,\mathbb{C}^3)}&=\int_Y\int_Y-G^\alpha(x,y)\mathbf{u}(y)\,dy\cdot\overline{\mathbf{u}(x)}\,dx\notag\\
%	&=\int_{Y\times Y}\sum_{n\in\mathbb{Z}^3}\frac{e^{i(2\pi\,n+\alpha)\cdot(x-y)}}{|2\pi\,n+\alpha|^2}\mathbf{u}(y)\cdot\overline{\mathbf{u}(x)}\,dy\,dx\notag\\
%	&=\int_{Y\times Y}\sum_{n\in\mathbb{Z}^3}\frac{e^{-i(2\pi\,n+\alpha)\cdot y}}{|2\pi\,n+\alpha|^2}\mathbf{u}(y)\cdot e^{i(2\pi\,n+\alpha)\cdot x}\overline{\mathbf{u}(x)}\,dy\,dx\notag\\
	&=\sum_{n\in\mathbb{Z}^3}\frac{\left|\int_{Y}e^{-i(2\pi\,n+\alpha)\cdot y}\mathbf{u}(y)\,dy\right|^2}{|2\pi\,n+\alpha|^2}.
\end{align}
Observe that, for $\alpha\in Y^*$, the following holds: 
\begin{equation*}
	|\alpha|^2\leq\left|\left|2\pi n\right|-\left|\alpha\right|\right|^2\leq\left|2\pi n+\alpha\right|^2,
\end{equation*}
and using this in (\ref{poincest1}), we have:
\begin{equation}
\label{poincest2}
(-\Delta_{\alpha}^{-1}\mathbf{u},\mathbf{u})_{L^2(Y,\mathbb{C}^3)}\leq\sum_{n\in\mathbb{Z}^3}\frac{\left|\int_{Y}e^{-i(2\pi\,n+\alpha)\cdot y}\mathbf{u}(y)\,dy\right|^2}{|\alpha|^2}.
\end{equation}
Now, write $\mathbf{u}(y)=\tilde{\mathbf{u}}(y)e^{i\alpha\cdot y}$ and observe that: $$\int_{Y}e^{-i(2\pi\,n+\alpha)\cdot y}\mathbf{u}(y)\,dy=\int_{Y}e^{-i(2\pi\,n)\cdot y}\tilde{\mathbf{u}}(y)\,dy=\hat{\tilde{\mathbf{u}}}(n),$$ where $\hat{\tilde{\mathbf{u}}}$ is the Fourier transform of $\tilde{\mathbf{u}}$, so we can rewrite (\ref{poincest2}) as:
\begin{equation}
\label{spectralbound}
(-\Delta_{\alpha}^{-1}\mathbf{u},\mathbf{u})_{L^2(Y,\mathbb{C}^3)}\leq\frac{1}{|\alpha|^2}\int_{Y}\left|\tilde{\mathbf{u}}(y)\right|^2\,dy=|\alpha|^{-2}\|\mathbf{u}\|_{L^2(Y,\mathbb{C}^3)}^2.
\end{equation}
%\begin{align}
%\label{spectralbound}
%(-\Delta_{\alpha}^{-1}\mathbf{u},\mathbf{u})_{L^2(Y,\mathbb{C}^3)}&\leq\sum_{n\in\mathbb{Z}^3}\frac{\hat{\tilde{\mathbf{u}}}(n)}{|\alpha|^2}=\frac{1}{|\alpha|^2}\int_{Y}\left|\tilde{\mathbf{u}}(y)\right|^2\,dy\notag\\
%	&=\frac{1}{|\alpha|^2}\int_{Y}\left|\mathbf{u}(y)\right|^2\,dy=|\alpha|^{-2}\|\mathbf{u}\|_{{L^2(Y)}^3}^2.
%\end{align}
Also, we have the Cauchy inequality:
\begin{align}
	\label{identl2}
\int_{Y}|\mathbf{u}(y)|^2\,dy&=\int_{Y}\nabla(-\Delta^{-1}_\alpha\mathbf{u}(y)):\overline{\nabla\mathbf{u}(y)}\,dy\notag\\
%	&=\int_{Y}\nabla\times(-\Delta^{-1}_\alpha\mathbf{u}(y))\cdot\overline{\nabla\times\mathbf{u}(y)}\,dy=\langle -\Delta^{-1}_\alpha\mathbf{u},\mathbf{u} \rangle\notag\\
	&\leq\Big(\int_{Y}\left|\nabla\times(-\Delta^{-1}_\alpha\mathbf{u}(y))\right|^2dy\Big)^{1/2}\Big(\int_{Y}\left|\nabla\times\mathbf{u}(y)\right|^2dy\Big)^{1/2}.
\end{align}
Applying (\ref{spectralbound}), we get:
\begin{align}
	\label{energyl2}
\Big(\int_{Y}\left|\nabla\times(-\Delta^{-1}_\alpha\mathbf{u}(y))\right|^2dy\Big)^{1/2}%&=\left(\int_{Y}\nabla\times(-\Delta^{-1}_\alpha\mathbf{u}(y))\cdot\overline{\nabla\times(-\Delta^{-1}_\alpha\mathbf{u}(y))}\,dy\right)^{1/2}\notag\\
&=\Big(\int_{Y}\nabla(-\Delta^{-1}_\alpha\mathbf{u}(y)):\overline{\nabla(-\Delta^{-1}_\alpha\mathbf{u}(y))}\,dy\Big)^{1/2}\notag\\
%&=\left(\int_{Y}-\Delta^{-1}_\alpha\mathbf{u}(y)\cdot\overline{\mathbf{u}(y))}\,dy\right)^{1/2}\notag\\
%&=\left((-\Delta_{\alpha}^{-1}\mathbf{u},\mathbf{u})_{L^2(Y,\mathbb{C}^3)}\right)^{1/2}\notag\\
	&\leq|\alpha|^{-1}\|\mathbf{u}\|_{L^2(Y,\mathbb{C}^3)}
\end{align}
and the Poincar\'e inequality follows from (\ref{identl2}) and (\ref{energyl2}).
\end{proof}
For any $\mathbf{u} \in L^2_{\#} (\alpha,Y,\mathbb{C}^3)$, we  apply \eqref{alpha-poincare} to find:
\begin{align}
	\label{tenfive}
	\|\left(A^\alpha(z)-A^\alpha(0)\right)\mathbf{u}\|_{L^2(Y,\mathbb{C}^3)}&\leq|\alpha|^{-1}\|\nabla\times\left(A^\alpha(z)-A^\alpha(0)\right)\mathbf{u}\|_{L^2(Y,\mathbb{C}^3)}\notag\\
%&\leq|\alpha|^{-1}\|\nabla\times\left((T_k^\alpha)^{-1}-(T_0)^{-1}\right)(-\Delta_\alpha)^{-1}\mathbf{u}\|_{L^2(Y,\mathbb{C}^3)}\notag\\
&\leq|\alpha|^{-1}\|\left((T_k^\alpha)^{-1}-P_2^\alpha\right)\|_{\mathcal{L}\left[J_{\#}(\alpha,Y,\mathbb{C}^3);J_{\#}(\alpha,Y,\mathbb{C}^3)\right]}\|-\Delta_\alpha^{-1}\mathbf{u}\|
\end{align}
Applying  \eqref{energyl2} and \eqref{tenfive} delivers the upper bound:
\begin{equation*}
%	\label{tenten}
	\|A^\alpha(z)-A^\alpha(0)\|_{\mathcal{L}\left[L^2_{\#} (\alpha,Y,\mathbb{C}^3);L^2_{\#} (\alpha,Y,\mathbb{C}^3)\right]}\leq|\alpha|^{-2}\|\left((T_k^\alpha)^{-1}-P_2^\alpha\right)\|_{\mathcal{L}\left[J_{\#}(\alpha,Y,\mathbb{C}^3);J_{\#}(\alpha,Y,\mathbb{C}^3)\right]}.
\end{equation*}

The next step is to obtain an upper bound on $\|\left((T_k^\alpha)^{-1}-P_2^\alpha\right)\|_{\mathcal{L}\left[J_{\#}(\alpha,Y,\mathbb{C}^3);J_{\#}(\alpha,Y,\mathbb{C}^3)\right]}$.  By (\ref{Takinv}), for all $\mathbf{u}\in J_{\#}(\alpha,Y,\mathbb{C}^3)$, we have:
\begin{align*}
%	\label{teneleven}
	&\frac{\Big(\int_Y\left|\nabla\times\left((T_k^\alpha)^{-1}-P_2^\alpha\right)\mathbf{u}\right|^2\,dy\Big)^{1/2}}{\Vert\mathbf{u}\Vert}\\
	&=\left(\frac{\int_Y|\nabla\times(z\,P_1^\alpha\mathbf{u}+\sum_{-\frac{1}{2}< \mu_i(\alpha)<\frac{1}{2}}z\left[(1/2+\mu_i(\alpha))+z(1/2-\mu_i(\alpha))\right]^{-1}P_{\mu_i}^\alpha\mathbf{u})|^2\,dy}{\Vert\mathbf{u}\Vert^2}\right)^{1/2}\notag\\
	%&\leq|z|\left(\frac{\displaystyle\int_Y\left|\nabla\times P_1\mathbf{u}\right|^2dy+\sum_{i=1}^{\infty}\left|(1/2+\mu_i)+z(1/2-\mu_i)\right|^{-2}\int_Y\left|\nabla\times P_{\mu_i}\mathbf{u}\right|^2\,dy}{\displaystyle\int_Y\left|\nabla\times\mathbf{u}\right|^2\,dy}\right)^{1/2}\notag\\
	&=|z|\Big(w_o+\sum_{-\frac{1}{2}< \mu_i(\alpha)<\frac{1}{2}}\left|(1/2+\mu_i(\alpha))+z(1/2-\mu_i(\alpha))\right|^{-2}w_i\Big)^{1/2},
\end{align*}
%\begin{align*}
%	\label{teneleven}
%	&\frac{\left(\displaystyle\int_Y\left|\nabla\times\left((T_k^\alpha)^{-1}-P_2^\alpha\right)\mathbf{u}\right|^2\,dy\right)^{1/2}}{\left(\displaystyle\int_Y\left|\nabla\times\mathbf{u}\right|^2\,dy\right)^{1/2}}\\
%	&=\left(\frac{\displaystyle\int_Y|\nabla\times(z\,P_1\mathbf{u}+\sum_{-\frac{1}{2}\leq \mu_i\leq\frac{1}{2}}z\left[(1/2+\mu_i)+z(1/2-\mu_i)\right]^{-1}P_{\mu_i}\mathbf{u})|^2\,dy}{\displaystyle\int_Y\left|\nabla\times\mathbf{u}\right|^2\,dy}\right)^{1/2}\notag\\
%	&\leq|z|\left(\frac{\displaystyle\int_Y\left|\nabla\times P_1\mathbf{u}\right|^2dy+\sum_{i=1}^{\infty}\left|(1/2+\mu_i)+z(1/2-\mu_i)\right|^{-2}\int_Y\left|\nabla\times P_{\mu_i}\mathbf{u}\right|^2\,dy}{\displaystyle\int_Y\left|\nabla\times\mathbf{u}\right|^2\,dy}\right)^{1/2}\notag\\
%	&=|z|\left(w_o+\sum_{i=1}^{\infty}\left|(1/2+\mu_i)+z(1/2-\mu_i)\right|^{-2}w_i\right)^{1/2},\notag
%\end{align*}
where $w_o=\Vert P_1^\alpha\mathbf{u}\Vert^2/\Vert\mathbf{u}\Vert^2$, $w_i=\Vert P_{\mu_i}^\alpha\mathbf{u}\Vert^2/\Vert\mathbf{u}\Vert^2$, and $w_o+\sum w_i=c\leq 1$, $c>0$.  Hence, maximizing the right hand side is equivalent to calculating:
\begin{align*}
&\max \limits_{w_0+\sum w_i = c\leq1} \{w_0 + \sum_{-\frac{1}{2}< \mu_i(\alpha)<\frac{1}{2}}w_i |(1/2 + \mu_i(\alpha)) + z(1/2-\mu_i(\alpha))|^{-2}\}^{1/2}\\
 &\qquad=  \sup \{1, |(1/2 + \mu_i(\alpha)) + z(1/2-\mu_i(\alpha))|^{-2}\}^{1/2}.
\end{align*}
Thus, we maximize the function:
\begin{equation*}
f(x) = \left|\frac{1}{2} + x + z\left(\frac{1}{2} -x\right)\right|^{-2}
\end{equation*}
over $x \in [\mu^-(\alpha), {\mu}^+(\alpha)]$, for $z$ in a neighborhood about the origin.  Let $Re(z)=u$, $Im(z)=v$, and we write:
\begin{equation*}
    f(x)  = \left|\frac{1}{2} + x + (u+iv)\left(\frac{1}{2} -x\right)\right|^{-2}\leq  \left(\frac{1}{2}+x+u\left(\frac{1}{2}-x\right)\right)^{-2} =: g(Re(z),x)
\end{equation*}
%\begin{align*}
%f(x) & = \left|\frac{1}{2} + x + (u+iv)\left(\frac{1}{2} -x\right)\right|^{-2}\\
%& =  \left[\left(\frac{1}{2}+x+u\left(\frac{1}{2}-x\right)\right)^2 + v^2\left(\frac{1}{2}-x\right)^2\right]^{-1}\\
%& \leq  \left(\frac{1}{2}+x+u\left(\frac{1}{2}-x\right)\right)^{-2} =: g(Re(z),x),
%\end{align*}
to get the bound:
\begin{equation}
\label{tenfifteen}
\Vert ((T_k^\alpha)^{-1} - P_2^\alpha)\Vert_{\mathcal{L}\left[J_{\#}(\alpha,Y,\mathbb{C}^3);J_{\#}(\alpha,Y,\mathbb{C}^3)\right]} \leq |z| \sup\,\Big\{1, \sup \limits_{x\, \in\, [\mu^-(\alpha), \mu^+(\alpha)]} g(u,x)\Big\}^{1/2}.
\end{equation}

We now examine the poles of $g(u,x)$ and the sign of its partial derivative $\partial_x g(u,x)$ when $|u|<1$.  If $Re(z)=u$ is fixed, then $g(u,x) = ((\frac{1}{2} + x) + u(\frac{1}{2} - x))^{-2}$ has a pole when $(\frac{1}{2} + x) + u(\frac{1}{2} - x)=0$.
For $u$ fixed, this occurs when $x=\hat{x}$, given by:
\begin{equation*}
\hat{x}=\hat{x}(u)= \frac{1}{2} \left(\frac{1+u}{u-1}\right).
\end{equation*}
On the other hand, if $x$ is fixed, $g$ has a pole at:
\begin{equation*}
u=\frac{x+1/2}{x-1/2}.
\end{equation*}
The sign of $\partial_x g$ is determined by the formula:
\begin{equation}
\label{teneighteen}
\partial_x g(u,x)  = \displaystyle \frac{-2(1-u)}{\left[\frac{1}{2}+x+u\left(\frac{1}{2}-x\right)\right]^3}=\frac{-2(1-u)^2x-(1-u^2)}{\left[\frac{1}{2}+x+u\left(\frac{1}{2}-x\right)\right]^4}.
\end{equation}
Observe that the denominator on the right hand side of (\ref{teneighteen}) is positive.  A calculation shows that $\partial_x g<0$ for $x>\hat{x}$, i.e. $g$ is decreasing on $(\hat{x},\infty)$.  Similarly, we have $\partial_x g>0$ for $x<\hat{x}$ and $g$ is increasing on $(-\infty, \hat{x})$.

Now, we identify all $u=Re(z)$ for which $\hat{x}=\hat{x}(u)$ satisfies $\hat{x} < \mu^-(\alpha) < 0$. 
%\begin{equation*}
%\hat{x} < \mu^-(\alpha) < 0 \text{.}
%\end{equation*}
Indeed, for such $u$, the function $g(u,x)$ will be decreasing on $[\mu^-(\alpha), \mu^+(\alpha)]$, so that, for all $x \in [\mu^-(\alpha), \bar{\mu}]$, we have $g(u,\mu^-(\alpha)) \geq g(u,x)$, yielding an upper bound for \eqref{tenfifteen}.
\begin{lemma}
\label{identifyu}
The set $U$ of $u \in \mathbb{R}$ for which $-\frac{1}{2} < \hat{x}(u) < \mu^-(\alpha) < 0$ is given by $U := [z^*, 1]$, where:
$$-1\leq z^*:=\frac{\mu^-(\alpha)+\frac{1}{2}}{\mu^-(\alpha)-\frac{1}{2}}<0.$$
\end{lemma}
\begin{proof}
Note first that $\mu^-(\alpha)=\inf_{i\in\mathbb{N}}\{\mu_i\}\leq 0$ follows from the fact that zero is an accumulation point for the sequence $\{\mu_i\}_{i\in\mathbb{N}}$, so it follows that:
\begin{equation*}
	\mu^-(\alpha)\leq-\mu^-(\alpha)\implies 1/2+\mu^-(\alpha)\leq(-1)(\mu^-(\alpha)-1/2)\implies z^*\geq-1.
\end{equation*}
Observe that $\hat{x} = \hat{x}(u) = \displaystyle \frac{u+1}{2(u-1)}$, we invert and write $u =\displaystyle \frac{\hat{x}+1/2}{\hat{x}-1/2}.$  
%\begin{equation*}
%u = \frac{\hat{x}+1/2}{\hat{x}-1/2}.
%\end{equation*}
We now show that $z^*\leq u\leq 1$, for $ \hat{x} \leq \mu^-(\alpha)$.  Set $h(\hat{x}) = \displaystyle\frac{\hat{x}+1/2}{\hat{x}-1/2}$, then $h'(\hat{x}) = \displaystyle\frac{-1}{(\hat{x}-\frac{1}{2})^2}<0$, %\begin{equation*}
%h'(\hat{x}) = \frac{-1}{(\hat{x}-\frac{1}{2})^2}<0,
%\end{equation*}
and so, $h$ is decreasing on $(-\infty, \frac{1}{2})$.  Since $\mu^-(\alpha)<\frac{1}{2}$, $h$ attains a minimum over $(-\infty, \mu^-(\alpha)]$ at $x=\mu^-(\alpha)$.  Thus $\hat{x}(u) \leq \mu^-(\alpha)$ implies:
\begin{align*}
\label{boundz*}
	&\frac{1}{2}\left(\frac{u+1}{u-1}\right)\leq\mu^-(\alpha)\implies z^*=\frac{\mu^-(\alpha)+1/2}{\mu^-(\alpha)-1/2}\leq u\leq1
%	&\implies u(2\mu^-(\alpha)-1)\geq2\mu^-(\alpha)+1\notag\\
%	&\implies z^*=\frac{\mu^-(\alpha)+1/2}{\mu^-(\alpha)-1/2}\leq u\leq1
\end{align*}
as desired.
\end{proof}
Combining Lemma~\ref{identifyu} with the inequality \eqref{tenfifteen}, noting that $-|z|\leq Re(z) \leq |z|$, and on rearranging terms, we obtain the following corollary.
\begin{corollary}
\label{boundAz}
For $|z| < |z^*|$, the following holds:
\begin{equation*}
\Vert (A^{\alpha}(z) - A^{\alpha}(0)) \Vert_{\mathcal{L}[L^2_{\#} (\alpha,Y,\mathbb{C}^3);L^2_{\#} (\alpha,Y,\mathbb{C}^3)]}  \leq |\alpha |^{-2} |z| (-|z|-z^*)^{-1}\Big(\frac{1}{2}-\mu^-(\alpha)\Big)^{-1}.
\end{equation*}
\end{corollary}
\begin{proof}
Observe that:
\begin{align*}
	\Vert A^\alpha(z)-A^\alpha(0)\Vert_{\mathcal{L}\left[L^2_{\#} (\alpha,Y,\mathbb{C}^3);L^2_{\#} (\alpha,Y,\mathbb{C}^3)\right]}&\leq|\alpha|^{-2}\|\left((T_k^\alpha)^{-1}-P_2^\alpha\right)\|_{\mathcal{L}\left[J_{\#}(\alpha,Y,\mathbb{C}^3;J_{\#}(\alpha,Y,\mathbb{C}^3\right]}\notag\\
	&\leq|\alpha|^{-2}|z|\sup\Big\{1,\sup_{x\,\in\,\left[\mu^-(\alpha),\mu^+(\alpha)\right]}g({\rm Re}(z),x)\Big\}^{1/2}\notag\\
%	&\leq|\alpha|^{-2}|z|\left(\frac{1}{2}+\mu^-(\alpha)+u\left(\frac{1}{2}-\mu^-(\alpha)\right)\right)^{-1}\notag\\
	&\leq|\alpha|^{-2}|z|\left(-|z|-z^*\right)^{-1}\Big(\frac{1}{2}-\mu^-(\alpha)\Big)^{-1}.
\end{align*}
\end{proof}
From Corollary~\ref{boundAz}, \eqref{tenonedouble} and  \eqref{tenonedoubleRz}, it follows that:
\begin{align*}
%\label{tentwentyseven}
&\Vert (A^{\alpha}(z) - A^{\alpha}(0))R(\zeta,0) \Vert_{\mathcal{L}[L^2_{\#} (\alpha,Y,\mathbb{C}^3);L^2_{\#} (\alpha,Y,\mathbb{C}^3)]} \\
&\quad\leq\quad |\alpha |^{-2} |z| (-|z|-z^*)^{-1}\Big(\frac{1}{2}-\mu^-(\alpha)\Big)^{-1}d^{-1}=:B(\alpha,z).
\end{align*}
A straightforward calculation shows that $B(\alpha,z)<1$, for:
\begin{equation*}
|z| < r^*:= \frac{|\alpha|^2d|z^*(\alpha)|}{\frac{1}{\frac{1}{2} - \mu^-(\alpha)} + |\alpha|^2d}\,\,,
\end{equation*}
and property~$4$ of Theorem~\ref{separationandraduus-alphanotzero} is established, since $r^* < |z^*|$.

Now we establish properties $1$ through $3$ of Theorem~\ref{separationandraduus-alphanotzero}.
Inspection of \eqref{foursix} shows that, if \eqref{tenone} holds and if $\zeta\in\mathbb{C}$ belongs to the resolvent of $A^\alpha(0)$, then it also belongs to the resolvent of $A^\alpha(z)$.  Since \eqref{tenone} holds for $\zeta\in\Gamma_j$ and $|z|<r^*$, property~$1$ of Theorem~\ref{separationandraduus-alphanotzero} follows.
Formula \eqref{Project1} shows that $P(z)$ is analytic in a neighborhood of $z=0$, determined by the condition that \eqref{tenone}  holds for $\zeta\in\Gamma_j$. The set $|z|<r^*$ lies inside this neighborhood
and property~$2$ of Theorem~\ref{separationandraduus-alphanotzero} is proved. The isomorphism expressed in property~$3$ of Theorem~\ref{separationandraduus-alphanotzero} follows directly from Lemma~4.10 of \cite{KatoPerturb} (Chapter~I, \S~4), which is also valid in a Banach space.

To prove Theorem~\ref{separationandraduus-alphazero}, we need  the  following Poincar\'{e} inequality for $J_{\#}(0,Y,\mathbb{C}^3)$.
\begin{lemma}The following inequality holds:
\begin{equation}
\label{poincarealphazero}
\Vert\mathbf{v}\Vert_{L^2_{\#}(0,Y,\mathbb{C}^3)} \leq \frac{1}{2\pi}\Vert\mathbf{v}\Vert.
\end{equation}
\label{poincalphazero}
\end{lemma}

This inequality is established proceeding as in the proof of Lemma~\ref{poincarealpha}, with \eqref{Green-periodic}.  Using \eqref{poincarealphazero} in place of \eqref{alpha-poincare}, we argue, as in the proof of Theorem~\ref{separationandraduus-alphanotzero}, to show that: 
\begin{equation*}
\Vert (A^{0}(z) - A^{0}(0))R(\zeta,0) \Vert_{\mathcal{L}[(L^2_{\#}(0,Y,\mathbb{C}^3);L^2_\#(0,Y,\mathbb{C}^3)]} <1
\end{equation*}
holds provided $|z| < r^*$, where $r^*$ is given by \eqref{radiusalphazero}.   This establishes Theorem~\ref{separationandraduus-alphazero}. 

The error estimates presented in Theorem~\ref{errorestm} are easily recovered from the arguments in \cite{KatoPerturb} (Chapter~II, \S~3); for completeness, we restate them here.  We begin with the following application of Cauchy inequalities to the coefficients $\beta^{\alpha}_n$ of \eqref{foureleven}, from \cite{KatoPerturb} (Chapter~II, \S~3, pg~88):
\begin{equation*}
\left | \beta^{\alpha}_n \right | \leq d(r^*)^{-n}.
\end{equation*}
It follows immediately that, for $|z|<r^*$, we have:
\begin{equation*}
\left |\hat{\beta}^{\alpha}(z) - \sum \limits_{n = 0}^{p} z^n \beta^{\alpha}_n \right | \leq \sum \limits_{n=p+1}^{\infty} |z|^n |\beta^{\alpha}_n| \leq \frac{d|z|^{p+1}}{(r^*)^p(r^* - |z|)}\,,
\end{equation*}
completing the proof.

For completeness, we establish the boundedness and compactness of the operator $B^\alpha(k)$ in  (\ref{inverseoperatorquadform}).
\begin{theorem}
\label{bounded}
The operator $B^\alpha(k): L^2_{\#}(\alpha,Y,\mathbb{C}^3) \longrightarrow J_{\#}(\alpha,Y,\mathbb{C}^3)$ is bounded for $k\not\in Z$.
\end{theorem}
\begin{proof}
For $\alpha\not=0$ and for $\mathbf{v}\in L^2_{\#}(\alpha,Y,\mathbb{C}^3)$, we have:
\begin{align*}
\Vert B^\alpha(k) \mathbf{v}\Vert&=\Vert(T_k^\alpha)^{-1}(-\Delta_\alpha)^{-1} \mathbf{v}\Vert\\
&\leq \| (T_k^\alpha)^{-1}\|_{\mathcal{L}[J_{\#}(\alpha,Y,\mathbb{C}^3);J_{\#}(\alpha,Y,\mathbb{C}^3)]} \Vert-\Delta_\alpha^{-1} \mathbf{v}\Vert\\
&\leq |\alpha|^{-1}\Vert ((T_k^\alpha)^{-1}\Vert_{\mathcal{L}[J_{\#}(\alpha,Y,\mathbb{C}^3);J_{\#}(\alpha,Y,\mathbb{C}^3)]} \Vert \mathbf{v}\Vert_{L^2(Y,\mathbb{C}^3)},
%\label{operatorfield}
\end{align*}
where the last inequality follows from \eqref{energyl2}. The upper estimate on $\Vert ((T_k^\alpha)^{-1}\Vert_{\mathcal{L}[J_{\#}(\alpha,Y,\mathbb{C}^3);J_{\#}(\alpha,Y,\mathbb{C}^3)]}$ is obtained from:
\begin{equation*}
%\label{tenelevenpart2}
\frac{\Vert (T_k^\alpha)^{-1}\mathbf{v}\Vert}{\Vert \mathbf{v} \Vert} \leq \Big\{|z|\hat{w}+\tilde{w}+|\sum \limits_{i=1}^{\infty}w_i |(1/2 + \mu_i) + z(1/2-\mu_i)|^{-2}\Big\}^{1/2},
\end{equation*}
where $\hat{w}=\Vert P_1^\alpha \mathbf{v}\Vert^2/\Vert\mathbf{v}\Vert^2$, $\tilde{w}=\Vert P_2^\alpha\mathbf{v}\Vert^2/\Vert \mathbf{v}\Vert^2$, and $w_i=\Vert P_{\mu_i}^\alpha \mathbf{v}\Vert^2/\Vert\mathbf{v}\Vert^2$.  Since $\hat{w}+\tilde{w}+\sum_{i=1}^\infty w_i=c\leq 1$, one recovers the upper bound:
\begin{equation*}
%\label{tenelevenpart3}
\frac{\Vert (T_k^\alpha)^{-1}\mathbf{v}\Vert}{\Vert \mathbf{v} \Vert} \leq \bar{M},
\end{equation*}
where:
\begin{equation*}
\bar{M}= \max\,\Big\{1, |z|,\, \sup_{i}\, \big\{ |(1/2 + \mu_i) + z(1/2-\mu_i)|^{-1}\big\}\Big\}.
%\label{summupperbound}
\end{equation*}
A similar argument can be carried out for $\alpha=\mathbf{0}$.
\end{proof}
\begin{theorem}
For $k\not\in Z$, $B^\alpha(k): L^2_{\#}(\alpha,Y,\mathbb{C}^3) \longrightarrow L^2_{\#}(\alpha,Y,\mathbb{C}^3)$ is a bounded compact  operator mapping $L^2_{\#}(\alpha,Y,\mathbb{C}^3)$ into itself.
\label{compact2}
\end{theorem}
\begin{proof}
The Poincar\'e inequalities \eqref{alpha-poincare} and \eqref{poincarealphazero}, together with Theorem~\ref{bounded}, show that
$B^\alpha(k): L^2_{\#}(\alpha,Y,\mathbb{C}^3) \longrightarrow L^2_{\#}(\alpha,Y,\mathbb{C}^3)$ is a bounded linear operator mapping $L^2_{\#}(\alpha,Y,\mathbb{C}^3)$ into itself. The compact embedding of $J_\#(\alpha,Y,\mathbb{C}^3)$ into $L^2_{\#}(\alpha,Y,\mathbb{C}^3)$ shows the operator is compact on $L^2_{\#}(\alpha,Y,\mathbb{C}^3)$. 
\end{proof}

\begin{section}{Conclusions}
    In this paper, analytic representation formulas and power series describing the band structure inside non-magnetic  periodic photonic crystals, made from high dielectric contrast inclusions, are developed.  The spectral representation for the operator $-\nabla\times (k \chi_{H} + \chi_{D})\nabla\times$ is derived, as well as a power series representation of Bloch eigenfunctions.  The radius of convergence for the power series, together with explicit formulas for each of its terms, in terms of layer potentials, is obtained.  The spectrum in the high contrast limit is completely characterized for the $\alpha$-quasiperiodic and periodic ($\alpha=\mathbf{0}$) cases. Explicit conditions on the contrast are found that provide lower bounds on the convergence radius. These conditions are sufficient for the separation of spectral branches of the dispersion relation for any fixed quasi-momentum.
\end{section}

\begin{appendix}
%\section{Appendix}

\section{Helmholtz decomposition for periodic and quasiperiodic vector fields.}
\label{app:helm}
Here, we show how to obtain the Helmholtz decomposition (\ref{Helmoltz1}).  First, consider  $\alpha\in Y^*$, $\alpha\neq\mathbf{0}$.  For $\mathbf{h}(x)\in L^2_\#(\alpha,Y,\mathbb{C}^3)$, we have $\mathbf{h}(x)=\mathbf{h}_{\rm per}(x,\alpha)e^{2\pi i\alpha\cdot x}$, where: $$\mathbf{h}_{\rm per}(x,\alpha)=\sum_{k\in\mathbb{Z}^3}\mathbf{\hat{h}}_{\rm per}(k,\alpha)e^{2\pi i\,k\cdot x}.$$  In other words: $$\mathbf{h}(x)=\sum_{k\in\mathbb{Z}^3}\mathbf{\hat{h}}_{\rm per}(k,\alpha)e^{2\pi i(k+\alpha)\cdot x}.$$
Now, define the following: 
\begin{align*}
\hat{h}_{\rm pot}(k,\alpha)&=-\frac{i}{2\pi}\frac{(k+\alpha)\cdot\mathbf{\hat{h}}_{\rm per}(k,\alpha)}{|k+\alpha|^2},\\
\mathbf{\hat{h}}_{\rm curl}(k,\alpha)&=\frac{i}{2\pi}\frac{(k+\alpha)\times\mathbf{\hat{h}}_{\rm per}(k,\alpha)}{|k+\alpha|^2}.
\end{align*}
By the vector triple product formula, we observe that:
\begin{align*}
	&2\pi i(\alpha+k)\,\hat{h}_{\rm pot}(k,\alpha)+2\pi i(\alpha+k)\times\mathbf{\hat{h}}_{\rm curl}(k,\alpha)\\
	&=\frac{(\alpha+k)\left[(\alpha+k)\cdot\mathbf{\hat{h}}_{\rm per}(k,\alpha)\right]}{|k+\alpha|^2}-\left[\frac{(\alpha+k)[(\alpha+k)\cdot\mathbf{\hat{h}}_{\rm per}(k,\alpha)]}{|k+\alpha|^2}-\frac{\mathbf{\hat{h}}_{\rm per}(k,\alpha)[(\alpha+k)\cdot(\alpha+k)]}{|k+\alpha|^2}\right]\\
	&=\mathbf{\hat{h}}_{\rm per}(k,\alpha).
\end{align*} 
It follows that $\mathbf{h}(x)=\nabla h_{\rm pot}(x)+\nabla\times \mathbf{h}_{\rm curl}(x)$, where:
\begin{align*}
    h_{\rm pot}(x)&=\sum_{k\in\mathbb{Z}^3}\hat{h}_{\rm pot}(k,\alpha)e^{2\pi i(k+\alpha)\cdot x},\\
    \mathbf{h}_{\rm curl}(x)&=\sum_{k\in\mathbb{Z}^3}\mathbf{\hat{h}}_{\rm curl}(k,\alpha)e^{2\pi i(k+\alpha)\cdot x}.
\end{align*}
This is the Helmholtz decomposition for $\alpha$-quasiperiodic fields, for $\alpha\in Y^*$, $\alpha\neq\mathbf{0}$.

When $\alpha=\mathbf{0}$, we have $\mathbf{h}(x)=\displaystyle\sum_{k\in\mathbb{Z}^3}\mathbf{\hat{h}}(k)e^{2\pi ik\cdot x}$, or equivalently: $$\mathbf{h}(x)=\mathbf{\hat{h}}(0)+\sum_{\substack{k\in\mathbb{Z}^3\\
k\neq0}}\mathbf{\hat{h}}(k)e^{2\pi i\,k\cdot x},$$ with $\mathbf{\hat{h}}(0)=\displaystyle\int_Y\mathbf{h}(x)$.  Then, the Helmholtz decomposition for $\mathbf{h}\in L^2_\#(0,Y,\mathbb{C}^3)$ is given by: $$\mathbf{h}(x)=\nabla h_{\rm pot}(x)+\nabla\times \mathbf{h}_{\rm curl}(x)+\mathbf{c},\,\,\,\mathbf{c}\in\mathbb{C}^3,$$
where: 
\begin{align*}
    h_{\rm pot}(x)&=\sum_{\substack{k\in\mathbb{Z}^3\\
k\neq0}}-\frac{i}{2\pi}\frac{k}{|k|^2}\cdot\mathbf{\hat{h}}(k)e^{2\pi ik\cdot x},\\
    \mathbf{h}_{\rm pot}(x)&=\sum_{\substack{k\in\mathbb{Z}^3\\
k\neq0}}\frac{i}{2\pi}\frac{k}{|k|^2}\times\mathbf{\hat{h}}(k)e^{2\pi ik\cdot x}.
\end{align*}

\section{For $\mathbf{h}\in J_{\#}(\alpha,Y,\mathbb{C}^3)$, $\nabla h_{pot}=0$ in (\ref{Helmoltz1}):}
\label{app:nhp=0}

If $\alpha\neq\mathbf{0}$, from Appendix~\ref{app:helm}, we have $\mathbf{h}(x)=\nabla h_{\rm pot}(x)+\nabla\times \mathbf{h}_{\rm curl}(x)$.  Taking divergence on both sides, and since $\mathbf{h}\in J_{\#}(\alpha,Y,\mathbb{C}^3)$, we obtain that $\Delta h_{\rm pot}=0$ in $Y$ and, since $h_{\rm pot}$ is $\alpha$-quasiperiodic, we have: $$\int_Y|\nabla h_{\rm pot}|^2=\int_{\partial Y}h_{\rm pot}\overline{\partial_{\mathbf{n}}h_{\rm pot}}=0.$$ 
A similar argument works for to the case $\alpha=\mathbf{0}$.

\section{Necessary Lemmas}
\label{app:lemmas}
\begin{lemma}\label{lem:app:c1}
For $\mathbf{u}$ and $\mathbf{v}$ in $J_{\#}(\alpha,Y,\mathbb{C}^3)$, we have:
$$\int_Y\,\nabla\times\mathbf{u}\cdot\nabla\times\overline{\mathbf{v}}\,dx=\int_Y\,\nabla\mathbf{u}:\nabla\overline{\mathbf{v}}\,dx.$$
\end{lemma}
\begin{proof}
Let us write:
$$\mathbf{u}(y)=\sum_{k\in\mathbb{Z}^3}e^{2\pi\, i(k+\alpha)\cdot y}\mathbf{\hat{u}^k}\hspace{1cm}\text{and}\hspace{1cm}\mathbf{v}(y)=\sum_{k\in\mathbb{Z}^3}e^{2\pi\, i(k+\alpha)\cdot y}\mathbf{\hat{v}^k}.$$
Then: 
\begin{align*}
    &\int_Y\,\nabla\times\mathbf{u}\cdot\nabla\times\overline{\mathbf{v}}\,dx=\int_Y\,\sum_{k\in\mathbb{Z}^3}2\pi\,i\,e^{2\pi\,i(k+\alpha)\cdot y}(k+\alpha)\times\mathbf{\hat{u}^k}\cdot\sum_{m\in\mathbb{Z}^3}\overline{2\pi\,i\,e^{2\pi\,i(m+\alpha)\cdot y}(m+\alpha)\times\mathbf{\hat{v}^k}}\,dx\\
    &\quad=4\pi|Y|\sum_{k\in\mathbb{Z}^3}(k+\alpha)\times\mathbf{\hat{u}^k}\cdot(k+\alpha)\times\mathbf{\hat{v}^k}\\
    &\quad=4\pi|Y|\sum_{k\in\mathbb{Z}^3}\left(\left|k+\alpha\right|^2\mathbf{\hat{u}^k}\cdot\mathbf{\hat{v}^k}-(k+\alpha)\cdot\mathbf{\hat{u}^k}(k+\alpha)\cdot\mathbf{\hat{v}^k}\right)\\
     &\quad=\int_Y\,\nabla\mathbf{u}:\nabla\overline{\mathbf{v}}\,dx-\int_Y\,(\nabla\cdot\mathbf{u})(\nabla\cdot\overline{\mathbf{v}})\,dx=\int_Y\,\nabla\mathbf{u}:\nabla\overline{\mathbf{v}}\,dx.
\end{align*}
\end{proof}

\begin{lemma} (See \cite{Bouchitte2017}, Lemma~4.7 for proof.)  Let $\mathbf{u}\in {L^2_{\#}(Y,\mathbb{C}^3)}$ such that ${\rm curl}\,\mathbf{u}\in {L^2_{\#}(Y,\mathbb{C}^3)}$ and ${\rm div}\,\mathbf{u}\in  {L^2_{\#}(Y)}$.  Then $\mathbf{u}\in W^{1,2}_{\#}(Y,\mathbb{C}^3) $ and:
\begin{equation*}
%	\label{W12perY3-norm}
\int_{Y}|\nabla\mathbf{u}|^2\,{\rm dx}=\int_{Y}|{\rm curl}\,\mathbf{u}|^2\,{\rm dx}+\int_{Y}|{\rm div}\,\mathbf{u}|^2\,{\rm dx}.
\end{equation*}
\end{lemma}

\begin{lemma}\label{lem:bbf}  Let $\mathbf{u}\in {L^2_{\#}(\alpha,Y,\mathbb{C}^3)}$ such that ${\rm curl}\,\mathbf{u}\in L^2_{\#}(\alpha,Y,\mathbb{C}^3)$ and ${\rm div}\,\mathbf{u}\in  {L^2_{\#}(\alpha,Y)}$.  Then $\mathbf{u}\in {W^{1,2}_{\#}(\alpha,Y,\mathbb{C}^3)}$ and:
\begin{equation}
	\label{W12alphaperY3-norm}
\int_{Y}|\nabla\mathbf{u}|^2\,{\rm dx}=\int_{Y}|{\rm curl}\,\mathbf{u}|^2\,{\rm dx}+\int_{Y}|{\rm div}\,\mathbf{u}|^2\,{\rm dx}.
\end{equation}
\end{lemma}

\begin{proof}
Let us write: $$\mathbf{u}(y)=\sum_{k\in\mathbb{Z}^3}e^{2\pi\, i(k+\alpha)\cdot y}\mathbf{c^k}.$$  We then have that:
\begin{align*}
{\rm curl}\,\mathbf{u}&=\sum_{k\in\mathbb{Z}^3}2\pi\,i\,e^{2\pi\,i(k+\alpha)\cdot y}(k+\alpha)\times\mathbf{c^k},\\
{\rm div}\,\mathbf{u}&=\sum_{k\in\mathbb{Z}^3}2\pi\,i\,e^{2\pi\,i(k+\alpha)\cdot y}(k+\alpha)\cdot\mathbf{c^k}.
\end{align*}
Since $\left|(k+\alpha)\times \mathbf{c^k}\right|^2+\left|(k+\alpha)\cdot \mathbf{c^k}\right|^2=\left|k+\alpha\right|^2\left|\mathbf{c^k}\right|^2$, we infer that $\sum_{k\in\mathbb{Z}^3}\left|k+\alpha\right|^2\left|\mathbf{c^k}\right|^2<\infty$, thus $\mathbf{u}\in {W^{1,2}_{\#}(\alpha,Y)}^3$.  Moreover, (\ref{W12alphaperY3-norm}) follows.
\end{proof}

\section{For $\mathbf{u}\in J_{\#}(\alpha,Y,\mathbb{C}^3)$, the null space of $\nabla \times \mathbf{u}$ is $\{0\}$:}
\label{app:h=0inY}

Let $\mathbf{u}\in J_{\#}(\alpha,Y,\mathbb{C}^3)$ such that $\nabla \times \mathbf{u}=0$.  Then, from Lemma~\ref{lem:app:c1}, we have: $$\int_Y|\nabla\mathbf{u}|^2=\int_Y|{\rm curl}\,\mathbf{u}|^2=0.$$ 
Then $\mathbf{u}$ must be a constant in $Y$.  If $\alpha\neq\mathbf{0}$, since $\mathbf{u}$ is $\alpha$-quasiperiodic, we conclude it must be zero.  If $\alpha=\mathbf{0}$, since $\int_{Y}\mathbf{u}\,d\mathbf{x}=0$, then we can also conclude that $\mathbf{u}=0$. 

\section{Periodic and $\alpha$-quasiperiodic Green's functions and their relation to the free space Green's function}
\label{app:extension4.4}
Consider $G^0$ and $G^\alpha$, defined in (\ref{Green-periodic}) and (\ref{Green-quasi}), respectively, and the free-space Green's function given by:
\begin{equation*}
%    \label{Green-free}
    \Gamma(x,y)=-\frac{1}{4\pi|x-y|}.
\end{equation*}
Observe that, in the unit cell $Y$, we have: $$\Delta(\Gamma(x,y)-G^0(x,y))=\delta(x-y)-(\delta(x-y)-1)=1$$ and, from the regularity of the elliptic problem, we have that $R^0(x)=\Gamma(x,y)-G^0(x,y)$ is smooth in $Y$, see \cite{AMMARI2005401}.  A similar argument works for $G^\alpha$, $\alpha\neq\mathbf{0}$.  In that case:$$\Delta G^\alpha(x,y)=\sum_{n\in\mathbb{Z}^3}\delta(x-y-n)e^{i\alpha\cdot n}\,\,\,\text{ in }\mathbb{R}^3, $$ which implies that, in the unit cell $Y$, we have:$$\Delta(\Gamma(x,y)-G^\alpha(x,y))=0,$$ from where $R^\alpha(x)=\Gamma(x,y)-G^\alpha(x,y)$ is smooth in $Y$.  The generalization of Lemma~$4.4$ of \cite{Mitrea1996} to the periodic and $\alpha$-quasiperiodic cases follows from the above.

%\subsection{Characterization of $NC$ operator}
%For a domain $D \subset \mathbb{R}^3$ with suitably smooth boundary, we defined in %Theorem \eqref{setcriteria} the operator $NC: (H^{1/2}(\partial D))^3 \rightarrow (H^{-1/2}(\partial D))^3$ by $$NC\xi=\mathbf{n}\times\nabla\times\mathbf{v}\vert_{\partial D},$$ where $\mathbf{n}$ is the outward pointing normal vector to $\partial D$, and with $\mathbf{v}$ satisfying:
%\begin{align}
%\label{uniqueness-curlcurl}
%    \nabla\times\nabla\times\mathbf{v}&=0 \text{ in $D$}\\
%    \nabla \cdot \mathbf{v} &=0 \text{ in $D$}\\
%    \mathbf{v}&=\xi\text{ on $\partial D$}.
%\end{align}
%\begin{lemma}
%Equation \eqref{uniqueness-curlcurl} has a unique solution for each $\xi \in %(H^{1/2}(\partial D))^3$.
%\end{lemma}
%\begin{proof}
%Let $\xi \in (H^{1/2}(\partial D))^3$, and suppose $\mathbf{v}$ satisfies %\eqref{uniqueness-curlcurl}...
%\end{proof}
%\section{Conclusions}
%\label{conlusions}
\end{appendix}

%\section{Numerical Results and a Dispersion Relation}
%Here we would relate the above asymptotic expansions in $H^1_\alpha(Y)$ back to the $\alpha$-Laplacian on $H^1_{\#}(Y)$, and then explain how to calculate the solutions (and their energies) using COMSOL.  We would then display the dispersion relation as $\alpha$ traverses the boundary of the lower triangle in the Brillouin Zone, and reflect pensively upon it.

\section*{Acknowledgements}
This research work is supported in part by NSF Grants DMS-1813698, DMREF-1921707, and DMS-2110036.

\bibliographystyle{plain}
\bibliography{Photonicrefs1}

%\begin{thebibliography}{99}
%\bibitem{AmmariKangLee}
%\bibitem{Kato}
%T. Kato, {\em Perturbation Theory for Linear Operators}, Springer, Berlin, 1980.
%\bibitem{Kellog}
%O.D. Kellogg, {\em Foundations of Potential Theory}, Dover, New York, 1953

%\end{thebibliography}

%\bibliographystyle{plain}
%\bibliographystyle{plainnat}
%\bibliography{Photonicrefs1}

%\begin{thebibliography}{99}
%\bibitem{AmmariKangLee}
%\bibitem{Kato}
%T. Kato, {\em Perturbation Theory for Linear Operators}, Springer, Berlin, 1980.
%\bibitem{Kellog}
%O.D. Kellogg, {\em Foundations of Potential Theory}, Dover, New York, 1953

%\end{thebibliography}

\end{document}